\documentclass[11pt]{article}

\usepackage{lineno,hyperref}
\modulolinenumbers[5]

\usepackage{amsmath}
\usepackage{amsfonts}
\usepackage{amsthm}

\usepackage{caption}
\usepackage{graphicx}        
\usepackage{epstopdf}
\usepackage{epsfig}
\usepackage[table]{xcolor}
\usepackage{subfloat}
\usepackage{float}
\usepackage[thinlines]{easytable}
\usepackage{array}
\allowdisplaybreaks
\numberwithin{equation}{section}
\newtheorem{proposition}{Proposition}[section]
\newtheorem{assumption}{Assumption}[section]
\newtheorem{theorem}{Theorem}[section]
\newtheorem{lemma}{Lemma}[section]
\newtheorem{remark}{Remark}[section]
\newtheorem{algorithm}{Algorithm}[section]
\def\RR{\mathbb R}
\def\EE{\mathbb E}
\def\DH{{\mathcal{D}}}

\def\w{v}

\def\theta{z}
\def\var{{\rm Var}}
\def\cov{{\rm Cov}}
\def\C{{\rm Cost}}
\def\LL{L^1_2(\DH\times\RR^{d_v})}

\def\LH{L^1_2(\RR^{d_v})}

\def\LLBi{L^1_2(\DH\times\RR^{d_v};L^2(\Omega))}
\def\LHBi{L^1_2(\RR^{d_v};L^2(\Omega))}
\def\be{\begin{equation}}
\def\ee{\end{equation}}
\def\bea{\begin{eqnarray}}
\def\eea{\end{eqnarray}}

\begin{document}
\title{Multi-scale control variate methods for uncertainty quantification in kinetic equations\footnote{
		The research that led to the present article was partially supported by the INdAM-GNCS 2016 research grant \emph{Numerical methods for uncertainty quantification in balance laws and kinetic equations}.}}
\author{Giacomo Dimarco\footnote{Department of Mathematics \& Computer Science, University of Ferrara, Via Machiavelli 30, Ferrara, 44121, Italy (giacomo.dimarco{@}unife.it).} and Lorenzo Pareschi\footnote{Department of Mathematics \& Computer Science, University of Ferrara, Via Machiavelli 30, Ferrara, 44121, Italy (lorenzo.pareschi{@}unife.it).}} 

	\maketitle
%
%
%

		\begin{abstract}
			Kinetic equations play a major rule in modeling large systems of interacting particles. Uncertainties may be due to various reasons, like lack of knowledge on the microscopic interaction details or incomplete informations at the boundaries. These uncertainties, however, contribute to the curse of dimensionality and the development of efficient numerical methods is a challenge. In this paper we consider the construction of novel multi-scale methods for such problems which, thanks to a control variate approach, are capable to reduce the variance of standard Monte Carlo techniques.
		\end{abstract}
		


\section{Introduction}
Kinetic equations have been applied to model a variety of phenomena whose multiscale nature cannot be described by a standard macroscopic approach \cite{Cer, DPR, Vlas}.  In spite of the vast amount of existing research, both theoretically and numerically (see \cite{DP15, Vill} for recent surveys), the study of kinetic equations has mostly remained deterministic and only recently a systematic study of the effects of uncertainty has been undertaken \cite{HJ,ZJ, LJ, DPZ,JPH,RHS}. In reality, there are many sources of uncertainties that can arise in these equations, like incomplete knowledge of the interaction mechanism between particles/agents, imprecise measurements of the initial and boundary data and other sources of uncertainty like forcing and geometry, etc. 

Besides physics and engineering, the legacy of classical kinetic theory have found recently interesting applications in socio-economic and life sciences \cite{BS2012, NPT, PT2}. The construction of kinetic models describing the above processes has to face the difficulty of the lack of fundamental principles since physical forces are replaced by empirical social forces. These empirical forces are typically constructed with the aim to reproduce qualitatively the observed system behaviors, like the emergence of social structures, and are at best known in terms of statistical information of the modeling parameters \cite{DPZ}.

Understanding the impact of these uncertainties is critical to the simulations of the complex kinetic systems to validate the kinetic models, and will allow scientists to obtain more reliable predictions and perform better risk assessment.  These uncertainties, however, contribute to the curse of dimensionality of kinetic equations and the development of efficient numerical methods for uncertainty quantification is a challenge. We refer the reader to \cite{DPL, JinPareschi, LeMK, NTW, PIN_book, ZLX} for a recent overview on various approaches to uncertainty quantification for hyperbolic and kinetic equations.

More precisely, we consider the development of efficient numerical methods for uncertainty quantification in kinetic equations of the general form 
\be
\partial_t f+ {\w}\cdot \nabla_x f = \frac1{\varepsilon}Q(f,f),
\label{eq:FP_general}
\ee
where $f=f(\theta,x,{\w},t)$, $t\ge 0$, $x\in\mathcal{D}\subseteq\RR^{d_x}$, ${\w}\in\RR^{d_{\w}}$, $d_x,d_{\w}\ge 1$, and $\theta\in\Omega\subseteq\RR^{d_\theta}$, $d_\theta \geq 1$, is a {random variable}. In \eqref{eq:FP_general} the parameter $\varepsilon > 0$ is the Knudsen number and the particular structure of the interaction term $Q(f,f)$ depends on the kinetic model considered.

Well know examples are given by the nonlinear Boltzmann equation of rarefied gas dynamics 
\be
Q(f,f)(z,v)=\int_{S^{d_{\w}-1}\times\RR^{d_{\w}}} B({\w},{\w}_*,\omega,\theta) (f(z,v')f(z,v'_*)-f(z,v) f(z,v_*))\,d{\w}_*\,d\omega
\label{eq:Boltzmann}
\ee
where the dependence from $x$ and $t$ has been omitted and
\be
v'=\frac12(v+v_*)+\frac12(|v-v_*|\omega),\quad v_*'=\frac12(v+v_*)-\frac12(|v-v_*|\omega),
\ee
or by nonlinear Vlasov-Fokker-Planck type models 
\be
Q(f,f)(z,x,v)=\nabla_{\w} \cdot \left[ \mathcal B[f](z,x,v)f(z,x,v)+\nabla_{\w} (D(z,x,v)f(z,x,v)) \right]
\label{eq:FP}
\ee
where the time dependence has been omitted and $\mathcal B[\cdot]$ is a non--local operator of the form 
\be
\label{eq:Bf}
\mathcal B[f](\theta,x,{\w}) = \int_{\RR^{d_x}}\int_{\RR^{d_{\w}}}P(x,x_*;{\w},{\w}_*,\theta)({\w}-{\w}_*)f(\theta,x_*,{\w}_*)d{\w}_*dx_*,
\ee
with $D(\theta,{\w})\ge 0$, for all ${\w}\in\RR^{d_{\w}}$, describing the local relevance of the diffusion. 

 {Additional examples of kinetic equations with uncertainties are found in \cite{JinPareschi}. In particular, for concrete examples of kinetic models in social dynamics and the various uncertainties associated we refer to the recent survey \cite{DPZ}.}

The development of numerical methods for kinetic equations presents several difficulties due to the high dimensionality and the intrinsic structural properties of the solution. Non negativity of the distribution function, conservation of invariant quantities, entropy dissipation and steady states are essential in order to compute physically correct solutions \cite{DP15,DQP,JinParma,PZ1,Son}. Preservation of these structural properties is even more challenging in presence of uncertainties which contribute to increase the dimensionality of the problem. 

The simplest class of numerical methods for quantifying uncertainty in partial differential equations are the stochastic collocation methods. In contrast to stochastic Galerkin (SG) methods based on generalized Polynomial Chaos (gPC) (see \cite{albi2015MPE, HJ, ZJ, LJ, DPZ,JPH,RHS} and the volume \cite{JinPareschi} for applications to kinetic equations and the references therein), stochastic collocation methods are non-intrusive, so they preserve all features of the deterministic numerical scheme, and easy to parallelize \cite{Xu}. Here we consider the closely related class of statistical sampling methods based on Monte Carlo (MC) techniques. 

Our motivations in focusing on non-intrusive MC sampling are the following: 
\begin{itemize}
\item nonintrusive methods allow simple integration of existing deterministic numerical solvers. In particular, they afford the use of fast spectral solvers and parallelization techniques which are essential to reduce the computational complexity of Boltzmann-type equations \cite{DP15, DLNR, MP}. 
\item due to their nonintrusive nature and their statistical approach, MC sampling methods have a lower impact on the curse of dimensionality compared to SG methods, especially when the dimension of the uncertainty space becomes very large \cite{Caflisch, Giles, PT2}. 
\item MC methods are effective when the probability distribution of the random inputs is not known analytically or lacks of regularity since other approaches based on stochastic orthogonal polynomials may be impossible to use or may produce poor results \cite{MSS, MS2, Xu}.     
\item kinetic equations close to fluid regimes are well approximated by nonlinear systems of hyperbolic
conservation laws, like the Euler equations. The application of SG methods to such systems often results in systems of gPC coefficients which are not globally hyperbolic since their Jacobian matrices may contain complex eigenvalues \cite{CJK, PDL}.
\end{itemize}

In order to address the slow convergence of MC methods, various techniques have been proposed \cite{Caflisch, Giles, HH, PT2}. Here, we discuss the development of low variance methods based on a control variate approach which makes use of the knowledge of the equilibrium state, or, more in general, the knowledge of a suitable approximated solution of the kinetic system.  These methods, inspired originally by micro--macro decomposition techniques \cite{CrLe, LemMie, LY, DPZ}, take advantage of the multiscale nature of the problem and can be regarded also as multi-fidelity methods \cite{ZLX}.

The rest of the article is organized as follows. In Section \ref{sec:2} we introduce the basic notions concerning MC sampling in uncertainty quantification. Section \ref{sec:3} is devoted to present our approach first for simpler space homogeneous problems and then its generalization to the space non homogeneous case. Several numerical examples in the case of the Boltzmann equation are contained in Section \ref{sec:4} and show the good performance of the new methods when compared to standard MC techniques. Some concluding remarks and future development are discussed in the last Section. 



\section{Preliminaries}
\label{sec:2}
Before starting our presentation we introduce some notations that will be used in the sequel.
If $z\in \Omega$ is distributed as $p(z)$ we denote the expected value of $f(z,x,v,t)$ by
\be
\EE[f](x, {\w}, t) = \int_{\Omega} f(z,x, {\w}, t)p(z)\,dz,
\ee
and its variance by
\be
\var(f)(x, {\w}, t) = \int_{\Omega} (f(z,x, {\w}, t)-\EE[f](x, {\w}, t))^2 p(z)\,dz.
\ee
Moreover, we introduce the following $L^p$-norm with polynomial weight \cite{PP}
\be
\| f(z,\cdot,t)\|_{L^p_s(\mathcal{D}\times\RR^{d_v})} = \int_{\mathcal{D}\times\RR^{d_v}} |f(z,x,v,t)|^p(1+|v|^s)\,dv\,dx.
\ee
 {Next, for a random variable $Z$ taking values in ${L^p_s(\DH\times\RR^{d_v})}$, we define 
\be
\|Z\|_{{L^p_s(\DH\times\RR^{d_v};L^2(\Omega))}}=\|\EE[Z^2]^{1/2}\|_{{L^p_s(\DH\times\RR^{d_v})}}.
\label{eq:norm1}
\ee
The above norm, if $p\neq 2$, differs from \cite{MSS, MS2}  
\be
\|Z\|_{L^2(\Omega;{L^p_s(\DH\times\RR^{d_v}))}}=\EE\left[\|Z\|^2_{L^p_s(\DH\times\RR^{d_v})}\right]^{1/2}.
\label{eq:norm2}
\ee
Note that by Jensen inequality we have
\be
\|Z\|_{{L^p_s(\DH\times\RR^{d_v};L^2(\Omega))}} \leq \|Z\|_{L^2(\Omega;{L^p_s(\DH\times\RR^{d_v}))}}.
\ee
In the following, to avoid non-essential difficulties, we will refer to norm \eqref{eq:norm1} for $p=1$. The same results hold true for $p=2$ (the two norms coincides) but their extension to norm \eqref{eq:norm2} for $p=1$ is not trivial and typically requires $Z$ to be compactly supported (see Remark \ref{rk:norm} in the next Section). We will also set $s=2$ since boundedness of energy is a natural assumption in kinetic equations.}

Rather than presenting a particular deterministic scheme used to discretize the kinetic equation, we state a general abstract convergence results required in the ensuing error analysis.
In the sequel we assume the following result (see \cite{CDM,DP15, RSS, Son}). 
\begin{assumption}
For an initial data $f_0$ sufficiently regular the deterministic solver used for \eqref{eq:FP_general} satisfies the error bound 
\be
\|f(\cdot,t^n)-f_{\Delta x,\Delta {\w}}^n\|_{L^1_2(\mathcal{D}\times\RR^{d_v})} \leq C (T,f_0)\left(\Delta x^p+\Delta {\w}^q \right),
\label{eq:det}
\ee
with $C$ a positive constant which depends on the final time $T$ and on the initial data $f_0$, and $f_{\Delta x,\Delta {\w}}^n$ the computed approximation of the deterministic solution $f(x,v,t)$ on the mesh $\Delta x$, $\Delta v$ at time $t^n$. 
\end{assumption}
Here, the positive integers $p$ and $q$ characterize the accuracy of the discretizations in the phase-space. For simplicity, we ignored errors due to the time discretization and to the truncation of the velocity domain in the deterministic solver \cite{DP15}.  {Finally, the regularity required by the initial data depends on the specific kinetic model under study  and typically refers to the assumptions needed for existence and uniqueness of the solution} \cite{Cer, DPR, PP, ToVil, Trist, Vill, Vlas}.

\subsection{Standard Monte-Carlo}
\label{sec:mc}
First we recall the standard Monte Carlo method when applied to the solution of a kinetic equation of the type (\ref{eq:FP_general}) with deterministic interaction operator $Q(f,f)$ and random initial data $f(\theta,x,{\w},0)=f_0(\theta,x,{\w})$. 

In this setting, the simplest Monte Carlo (MC) method for UQ is based on the following steps. 

\begin{algorithm}[Standard Monte Carlo (MC)  method]~
\begin{enumerate}
\item {\bf Sampling}: Sample $M$ independent identically distributed (i.i.d.) initial data $f_0^k$, $k=1,\ldots,M$ from
the random initial data $f_0$ and approximate these over the grid.
\item {\bf Solving}: For each realization $f_0^k$ the underlying kinetic equation (\ref{eq:FP_general}) is solved numerically by the deterministic solver. We denote the solutions at time $t^n$ by $f^{k,n}_{\Delta x,\Delta {\w}}$, $k=1,\ldots,M$, where  $\Delta x$ and $\Delta {\w}$ characterizes the discretizations  in $x$ and ${\w}$. 
\item {\bf Estimating}: Estimate the expected value of the random solution field with the sample mean of the approximate solution
\be
E_M[f^{n}_{\Delta x,\Delta {\w}}]=\frac1{M} \sum_{k=1}^M f^{k,n}_{\Delta x,\Delta {\w}}.
\label{mcest}
\ee
\end{enumerate}
\end{algorithm}
Similarly, one can approximate higher statistical moments. The above algorithm is straightforward to implement in any existing deterministic code for the particular kinetic equation. Furthermore, the only (data) interaction between different samples { is in step $3$}, when ensemble averages are computed. Thus, the MC algorithms for UQ are non-intrusive {and easily parallelizable as well}.

 {Starting from the fundamental estimate \cite{Caflisch, Lo77}
\be
\EE\left[(\EE[f]-E_M[f])^2\right] \leq C\, \var(f)M^{-1},
\ee
the typical error bound that one obtains using the standard Monte Carlo method described above reads (see \cite{MSS, MS2} for more details)
\begin{proposition}
Consider a deterministic scheme which satisfies \eqref{eq:det} for a kinetic equation of the form \eqref{eq:FP_general} with deterministic interaction operator $Q(f,f)$ and random initial data $f(\theta,x,{\w},0)=f_0(\theta,x,{\w})$. Assume that the initial data is sufficiently regular.
Then, the Monte Carlo estimate defined in \eqref{mcest} satisfies the error bound 
\be
\|\EE[f](\cdot,t^n)-E_M[f^{n}_{\Delta x,\Delta {\w}}]\|_{{\LLBi}} \leq C \left(\sigma_f M^{-1/2} + \Delta x^p + \Delta {\w}^q\right),
\label{eq:MCest}
\ee
where $\sigma_f = \|\var(f)^{1/2}\|_{L^1_2(\mathcal{D}\times\RR^{d_v})}$,
the constant $C=C(T, f_0)>0$ depends on the final time $T$ and the initial data $f_0$.
\end{proposition}
\begin{proof}
The above bound follows from
\begin{eqnarray*}
\nonumber
&\|\EE[f](\cdot,t^n)-E_M[f^{n}_{\Delta x,\Delta {\w}}]\|_{{\LLBi}} \\  
\nonumber
&&\hskip -2cm \leq \|\EE[f](\cdot,t^n)-E_M[f](\cdot,t^n)\|_{{\LLBi}}\\[-.25cm] 
\\[-.25cm]
\nonumber
&&\hskip -2cm+\|E_M[f](\cdot,t^n)-E_M[f_{\Delta x,\Delta {\w}}^n]\|_{{\LLBi}}\\ 
\nonumber
&&\hskip -2cm\leq C_1 M^{-1/2}\|\var(f)^{1/2}\|_{L^1_2(\mathcal{D}\times\RR^{d_v})}+C_2\left(\Delta x^p + \Delta {\w}^q\right)\\ 
&&
\hskip -2cm\leq C \left(\sigma_f M^{-1/2} + \Delta x^p + \Delta {\w}^q\right),
\end{eqnarray*}
where the first term in the inequality is a Monte Carlo error bound and the second term is essentially a discretization error.
\end{proof}}

Clearly, it is possible to equilibrate the discretization and the sampling errors in the a-priori estimate taking $M={\cal O}(\Delta x^{-2p})$ and
$\Delta x={\cal O}(\Delta {\w}^{q/p})$. This means that in order to have comparable errors the number of samples should be extremely large, especially when dealing with high order deterministic discretizations. This may make the Monte Carlo approach very expensive in practical applications.  


 {
\begin{remark}
\label{rk:norm}
Concerning the relationships between the norms \eqref{eq:norm1} and \eqref{eq:norm2} let us remark that for norm \eqref{eq:norm1} and $p\geq 1$ we have
\bea
\nonumber
\|\EE[f]-E_M[f]\|^p_{{L^p_2(\DH\times\RR^{d_v};L^2(\Omega))}} &=& \int_{\DH\times\RR^{d_v}} \EE\left[(\EE[f]-E_M[f])^2\right]^{p/2} (1+|v|)^2\,dv\,dx\\
\nonumber
&\leq& C{M^{-p/2}} \int_{\DH\times\RR^{d_v}} \var(f)^{p/2} (1+|v|)^2\,dv\,dx\\
\nonumber
&=& C{M^{-p/2}} \|\var(f)^{1/2}\|^p_{{L^p_2(\DH\times\RR^{d_v})}}. 
\eea
For norm \eqref{eq:norm2}, we have
\bea
\nonumber
\|\EE[f]-E_M[f]\|^2_{L^2(\Omega;{L^p_2(\DH\times\RR^{d_v}))}} &=& \EE\left[\left(\int_{\DH\times\RR^{d_v}} |\EE[f]-E_M[f]|^p (1+|v|)^2\,dv\,dx\right)^{2/p}\right],
\eea
thus for $p=1$ (if $f$ is compactly supported) thanks to Cauchy-Schwartz inequality
\bea
\nonumber
\|\EE[f]-E_M[f]\|^2_{L^2(\Omega;{L^1_2(\DH\times\RR^{d_v}))}} &=& \EE\left[\left(\int_{\DH\times\RR^{d_v}} |\EE[f]-E_M[f]| (1+|v|)^2\,dv\,dx\right)^{2}\right]\\
\nonumber
&\leq& C_1\,\EE\left[\int_{\DH\times\RR^{d_v}} (\EE[f]-E_M[f])^2 (1+|v|)^4\,dv\,dx\right]\\
\nonumber
&=& C_1\, \|\EE[f]-E_M[f]\|^2_{L^2(\Omega;{L^2_4(\DH\times\RR^{d_v}))}}\\
\nonumber
&\leq & C M^{-1} \|\var(f)^{1/2}\|^2_{L^2_4(\DH\times\RR^{d_v})}. 
\eea
The last estimate makes it possible to extend all the results presented in the rest of the article for norm \eqref{eq:norm1} to norm \eqref{eq:norm2} for $p=1$ under the additional assumption of a compactly supported function. Note, however, that in general this is not the case for a kinetic equation in the velocity space. We leave possible generalizations under weaker assumptions to future research.
\end{remark}
}

\section{Multi-scale control variate methods}
\label{sec:3}
In order to improve the performances of standard MC methods, we introduce a novel class of variance reduction Multi-Scale Control Variate (MSCV) methods. The key idea is to take advantage of the knowledge of the steady state solution of the kinetic equation, or more in general, of an approximate time dependent solution with the same asymptotic behavior close to the fluid limit, in order to accelerate the convergence of the Monte Carlo estimate.
In the sequel, we first describe the method in the space homogeneous case and subsequently we discuss its generalization to space non homogeneous problems.


\subsection{Space homogeneous case}

In order to illustrate the general principles of the method let us consider the space homogeneous problem  
\be
\frac{\partial f}{\partial t} = Q(f,f),
\label{eq:BH}
\ee
where $f=f(\theta,\w,t)$ and with initial data $f(\theta,{\w},0)=f_0(\theta,{\w})$. In \eqref{eq:BH}, without loss of generality, we have fixed $\varepsilon=1$ since in the space homogeneous case the Knudsen number scales with time.
 
We introduce the classical micro-macro decomposition
\be 
f(\theta,{\w},t)=f^{\infty}(\theta,{\w})+g(\theta,{\w},t),
\label{eq:f_MM}
\ee 
where $f^{\infty}(\theta,{\w})$ is the steady state solution of the interaction operator considered $Q(f^\infty,f^\infty)=0$ and $g(\theta,{\w},t)$ is such that 
\be
m_{\phi}(g)=\int_{\RR^{d_{\w}}}\phi({\w})g(\theta,{\w},t)d{\w} = 0,
\label{eq:mom}
\ee
for some moments, for example $\phi({\w})=1,{\w},|{\w}|^2/2$ (the classical setting of conservation of mass, momentum and energy). The above decomposition \eqref{eq:f_MM} applied to the homogeneous kinetic equation \eqref{eq:BH} where $Q(f,f)$ is the nonlinear Boltzmann operator \eqref{eq:Boltzmann} yields the following result.
\begin{proposition}
If the homogeneous equations (\ref{eq:BH}) admits the unique equilibrium state $f^{\infty}(\theta,{\w})$,
the interaction operator $Q(\cdot,\cdot)$ defined in \eqref{eq:Boltzmann} may be rewritten as
\be
Q(f,f)(\theta,{\w},t) = Q(g,g)(\theta,{\w},t) + \mathcal L(f^{\infty},g)(\theta,{\w},t),
\ee
where $\mathcal L(\cdot,\cdot)$ is a linear operator defined as 
\[
\mathcal L(f^{\infty},g)(\theta,{\w},t) = Q(g,f^\infty)(\theta,{\w},t)+Q(f^\infty,g)(\theta,{\w},t).
\]
 The only admissible steady state solution of the problem 
\begin{equation}
\label{eq:micro-macro}
\begin{cases}
\partial_t g(\theta,{\w},t) = Q(g,g)(\theta,{\w},t)+ \mathcal L(f^{\infty},g)(\theta,{\w},t),\\
f(\theta,{\w},t) = f^{\infty}(\theta,{\w})+g(\theta,{\w},t)
\end{cases}\end{equation}
is given by $g^{\infty}(\theta,{\w})\equiv 0$. 
\end{proposition}
The proof is an immediate consequence of the structure of the collision operator and the fact that at the steady state we have $Q(f^{\infty},f^{\infty})=0$. 
Under suitable assumptions, {one can show that} $f(\theta,{\w},t)$ exponentially decays to the equilibrium solution \cite{Vill, Trist, ToVil}. As a consequence, the non--equilibrium part of the micro--macro approximation $g(\theta,{\w},t)$ exponentially decays to $g^{\infty}(\theta,{\w})\equiv 0$ for all $\theta\in\Omega$. Similar conclusions are obtained in the case of Fokker-Planck type operators \cite{DPZ, Tos1999}.


\subsubsection{Control variate strategies}

The crucial aspect is that the equilibrium state $g^\infty(\theta,{\w})$ is zero and therefore, independent from $\theta$. More precisely, we can decompose the expected value of the distribution function in
an equilibrium and non equilibrium part 
\be
\begin{split}
{\mathbb E}[f]({\w},t)&=\int_{\Omega} f(\theta,{\w},t)p(\theta)d\theta\\
&=\int_{\Omega} f^\infty(\theta,{\w})p(\theta)d\theta+
\int_{\Omega} g(\theta,{\w},t)p(\theta)d\theta\\
&=\EE[f^\infty](v)+\EE[g](v,t),
\end{split}
\label{eq:mme}
\ee
and then exploit the fact that since $f^\infty(\theta,{\w})$ is known $\EE[f^\infty](v)$ can be evaluated with a negligible error, to have an estimate of the error using $M$ samples of the 
type
\[ 
\| {\EE}[f](\cdot,t)-\EE[f^\infty](\cdot)-E_M[g](\cdot,t)\|_{{\LHBi}} \simeq \sigma_g M^{-1/2},
\] instead of the standard MC estimate
\[ 
\| {\EE}[f](\cdot,t)-{E}_M[f](\cdot,t)]\|_{{\LHBi}} \simeq \sigma_f M^{-1/2}, 
\]
where  {$\sigma_g=\|\var(g)^{1/2}\|_{\LH}$ and $\sigma_f=\|\var(f)^{1/2}\|_{\LH}$}. Now, since it is known that the non equilibrium part $g$ goes to zero in time exponentially fast, then also its variance goes to zero, which means that for long times the Monte Carlo integration based on the micro-macro decomposition becomes exact, since it only depends on the way in which $\EE[f^\infty]$ is computed.

We can improve this micro-macro Monte Carlo method using a parameter dependent control variate approach. More precisely, given $M$ i.i.d. samples $f^{k}(v,t)$, $k=1,\ldots,M$ of our solution at time $t$ we can write 
\be
{\mathbb E}[f]({\w},t) \approx E^{\lambda}_M[f](v,t)=\frac1{M} \sum_{k=1}^M f^{k}(v,t) - \lambda\left(\frac1{M} \sum_{k=1}^M f^{\infty,k}(v)-{\bf f}^{\infty}({\w})\right),
\label{eq:nest}
\ee
where $\lambda\in\RR$ and ${\bf f}^{\infty}({\w})={\mathbb E}[f^{\infty}](\w)$ or an approximation of it with a negligible error.

\begin{lemma}
The control variate estimator \eqref{eq:nest} is unbiased and consistent for any choice of $\lambda\in\RR$. In particular, for $\lambda=0$ we obtain $E^{0}_M[f]=E_M[f]$ the standard MC estimator and for $\lambda=1$ we get
\be
E^{1}_M[f](v,t) = {\bf f}^{\infty}({\w}) + \frac1{M} \sum_{k=1}^M (f^{k}(v,t) - f^{\infty,k}(v)) = {\bf f}^{\infty}({\w})+E_M[g](v,t),
\ee
the micro-macro estimator based on \eqref{eq:mme}.
\label{le:1}
\end{lemma}
\begin{proof}
The expected value of the control variate estimator $E^{\lambda}_M[f]$ in \eqref{eq:nest} yields the unbiasedness for any choice of $\lambda\in\RR$ 
\[
\EE[E^{\lambda}_M[f]]=\frac1{M} \sum_{k=1}^M \EE[f^{k}]-\lambda\left(\frac1{M} \sum_{k=1}^M \EE[f^{\infty,k}]-\EE[f^{\infty}]\right)=\EE[f],
\]
since $\EE[f^{k}]=\EE[f]$ and $\EE[f^{\infty,k}]=\EE[f^{\infty}]$ for $k=1,\ldots,M$.
Moreover, since $f^k$ and $f^{\infty,k}$ are i.i.d. random variables
\[
\displaystyle\lim_{M\to \infty} \frac1{M} \sum_{k=1}^M f^{k} - \lambda\left(\frac1{M} \sum_{k=1}^M f^{\infty,k}-\EE[f^{\infty}]\right) \overset{p}{=}  \EE[f],
\]
from the consistency of the standard MC estimator and where the last identity has to be understood in probability sense \cite{HH,Lo77}. The last part of Lemma follows directly from \eqref{eq:nest}.
\end{proof}
Let us consider the random variable
\[
f^{\lambda}(z,v,t)=f(v,z,t)-\lambda(f^{\infty}(z,v)-\EE[{f}^{\infty}](\cdot,v)).
\]
Clearly, ${\mathbb E}[f^\lambda]={\mathbb E}[f]$, $E_M[f^\lambda]=E_M^{\lambda}[f]$ and we can quantify its variance at the point $(v,t)$ as 
\be
\var(f^\lambda)=\var(f)+\lambda^2 \var(f^\infty)-2\lambda\cov(f,f^\infty).
\label{eq:var1}
\ee
We have the following
\begin{theorem}
If $\var(f^\infty)\neq 0$ the quantity   
\be
\lambda^* = \frac{\cov(f,f^\infty)}{\var(f^\infty)}
\label{eq:lambdas}
\ee
minimizes the variance of $f^\lambda$ at the point $(v,t)$ and gives
\be
\var(f^{\lambda^*}) = (1-\rho_{f,f^\infty}^2)\var(f), 
\label{eq:var2}
\ee
where $\rho_{f,f^\infty} \in [-1,1]$ is the correlation coefficient between $f$ and $f^\infty$. In addition, we have  
\be
\lim_{t\to\infty} \lambda^*(v,t) =1,\qquad \lim_{t\to\infty} \var(f^{\lambda^*})(v,t)=0\qquad \forall\, v \in \RR^{d_v}.
\label{eq:lambdasa}
\ee
\label{pr:1}
\end{theorem}
\begin{proof}
Equation \eqref{eq:lambdas} is readily found by direct differentiation of \eqref{eq:var1} with respect to $\lambda$ and then observing that $\lambda^*$ is the unique stationary point. The fact that $\lambda^*$ is a minimum follows from the positivity of the second derivative $2\var(f^\infty)>0$. Then, by substitution in \eqref{eq:var1} of the optimal value $\lambda^*$ one finds \eqref{eq:var2} where
\[
\rho_{f,f^\infty} = \frac{\cov(f,f^\infty)}{\sqrt{\var(f)\var(f^\infty)}}.
\] 
In addition, since as $t\to\infty$ we have $f\to f^\infty$, asymptotically $\lambda^*\to 1$ and $\var(f^{\lambda^*})\to 0$ independently of $v$.
\end{proof}
In practice, $\cov(f,f^\infty)$ appearing in $\lambda^*$ is not known and has to be estimated. Starting from the $M$ samples we can compute unbiased estimators for the variance and the covariance as
\bea
\label{eq:varm}
{\var}_M(f^\infty) &=& \frac1{M-1}\sum_{k=1}^M (f^{\infty,k}-E_M[f^\infty])^2,\\
\label{eq:covm}
{\cov}_M(f,f^\infty) &=& \frac1{M-1}\sum_{k=1}^M (f^{k}-E_M[f])(f^{\infty,k}-E_M[f^\infty]),
\eea
and estimate
\be
{\lambda}_M^*= \frac{{\cov}_M(f,f^\infty)}{{\var}_M(f^\infty)}.
\ee 
It can be verified easily that ${\lambda}_M^*\to 1$ as $f\to f^\infty$.  
\begin{remark}~
\label{rk:31}
\begin{itemize}
\item 
The method relies on the possibility to have an accurate estimate of the expected value of the equilibrium state ${\mathbb E}[f^{\infty}]$. For many space homogeneous kinetic models the equilibrium state can be computed directly from the initial data thanks to the conservation properties of $Q(f,f)$, therefore, its expectation can be evaluated with arbitrary accuracy at a negligible cost  since it does not depend on the solution computed at each time step. 
\item
Note that, we are implicitly assuming that for large times the underlaying deterministic numerical method is able to capture correctly the large time behavior of the kinetic model. We refer to \cite{DP15, Gosse, JinParma, PZ1} and the references therein for an overview of such schemes for kinetic equations.
\end{itemize}
\end{remark}

\subsubsection{Time dependent control variate}
\label{sec:BGK}
The control variate approach described above can be generalized with the aim to improve the MC estimate also for shorter times. In fact, the importance of high correlation with the control variate for effective variance reduction can in \eqref{eq:var2} be seen with clarity.

To this goal one can consider as a control variate a time dependent approximation of the solution $\tilde{f}(\theta,{\w},t)$, whose evaluation is significantly cheaper than computing $f(\theta,{\w},t)$, such that 
$m_\phi(\tilde{f})=m_\phi(f)$ for some moments and that
$\tilde{f}(\theta,{\w},t)\to f^{\infty}(\theta,{\w})$ as $t\to\infty$. In terms of function decomposition this would correspond to write 
\be
f(\theta,{\w},t)={{\tilde{f}}}(\theta,{\w},t)+\tilde{g}(\theta,{\w},t),
\label{eq:f_MM2}
\ee
with $m_\phi(\tilde{g})=0$ for the same moments. Note that, even in this case, the perturbation $\tilde{g}(\theta,v,t) \to 0$ as $t\to \infty$.
 
As an illustrative example, we consider the space homogeneous Boltzmann equation \eqref{eq:BH} where $Q(f,f)$ is given by  \eqref{eq:Boltzmann} and assume 
 $\tilde{f}(\theta,{\w},t)$ to be the exact solution of the space homogeneous BGK approximation 
\be
\frac{\partial \tilde{f}}{\partial t} =\nu (\tilde{f}^\infty - \tilde{f}), 
\label{eq:BGKh}
\ee 
where $\nu>0$ is a constant, and 
for the same initial data $f_0(\theta,{\w})$. Thanks to the time invariance of the equilibrium state we have $\tilde{f}^\infty=f^\infty$ and we can write the exact solution to \eqref{eq:BGKh} as
\be
\tilde{f}(\theta,{\w},t) = e^{-\nu t} f_0(\theta,{\w})+(1-e^{-\nu t}) f^{\infty}(\theta,{\w}).
\label{eq:BGKexa}
\ee
We can assume again that the expected value of the control variate $\EE[\tilde{f}]({\w},t)$ is computed with arbitrary accuracy at a negligible cost since it is a convex combination of the initial data and the equilibrium part. We denote this value by 
\be
\tilde{\bf f}(v,t)=e^{-\nu t} {\bf f}_0({\w})+(1-e^{-\nu t}) {\bf f}^{\infty}({\w}),
\label{eq:BGKexas}
\ee
where ${\bf f}_0={\mathbb E}[f_0(\cdot,{\w})]$ and ${\bf f}^{\infty}={\mathbb E}[f^{\infty}](\w)$ or accurate approximations of the same quantities.
 
The control variate estimate then reads
\be
{\mathbb E}[f]({\w},t) \approx \tilde{E}^{\lambda}_M[f](v,t)=\frac1{M} \sum_{k=1}^M f^{k}(v,t) - \lambda\left(\frac1{M} \sum_{k=1}^M \tilde{f}^{k}(v,t)-\tilde{\bf f}({\w},t)\right).
\label{eq:nest2}
\ee
By Lemma \ref{le:1} the above control variate estimator is unbiased and consistent for any $\lambda\in\RR$. For $\lambda=0$ we recover again the standard MC estimator, whereas for $\lambda=1$ we have the estimator $\tilde{E}^{1}_M[f](v,t)=\tilde{\bf f}(v,t)+E_M[\tilde g](v,t)$ given by the function decomposition \eqref{eq:f_MM2}.

By Theorem \ref{pr:1}, if $\var(\tilde{f})\neq 0$ the optimal value for $\lambda$ in terms of variance reduction is given by
\be
\lambda^* = \frac{\cov(f,\tilde{f})}{\var(\tilde{f})},
\label{eq:lambdas2}
\ee
and can be estimated using the analogous of expressions \eqref{eq:varm} and \eqref{eq:covm}. Since for large times we have $f\to f^\infty$ and $\tilde f \to f^\infty$ again $\lambda^*\to 1$ and $\var(f^{\lambda^*})\to 0$.

The resulting multi-scale control variate algorithm, which can be easily generalized to other control variate functions, is summarized in the following steps:
\begin{algorithm}[Multi-scale Control Variate (MSCV) method - homogeneous case]~
\label{alg:3}
\begin{enumerate}
\item {\bf Sampling}: Sample $M$ i.i.d. initial data $f_0^k$, $k=1,\ldots,M$ from
the random initial data $f_0$ and approximate these over the grid.  
\item {\bf Solving:} For each realization $f_0^k$, $k=1,\ldots,M$
\begin{enumerate}
\item Compute the control variate $\tilde{f}_{\Delta {\w}}^{k,n}$, $k=1,\ldots,M$ at time $t^n$ using \eqref{eq:BGKexa} and denote by $\tilde{\bf f}_{\Delta {\w}}^{n}$ an accurate estimate of $\EE[\tilde{f}_{\Delta {\w}}^{n}]$ obtained from $\tilde{\bf f}_{\Delta {\w}}^{0}$ and $\tilde{\bf f}_{\Delta {\w}}^{\infty}$ using \eqref{eq:BGKexas}. 

\item Solve numerically
the underlying kinetic equation (\ref{eq:BH})  by the corresponding deterministic solvers. We denote the solution at time $t^n$ by $f^{k,n}_{\Delta {\w}}$, $k=1,\ldots,M$. 
\end{enumerate} 
\item {\bf Estimating}: 
\begin{enumerate}
\item
Estimate the optimal value of $\lambda^*$ as
\[
{\lambda}_M^{*,n}= \frac{\sum_{k=1}^M (f_{\Delta \w}^{k,n}-E_M[f_{\Delta \w}^n])(\tilde f_{\Delta \w}^{k,n}-E_M[\tilde f_{\Delta \w}^n])}{\sum_{k=1}^M (\tilde f_{\Delta \w}^{k,n}-E_M[\tilde f_{\Delta \w}^n])^2}.
\]
\item
Compute the expectation of the random solution with the control variate estimator
\be
\tilde{E}^{\lambda^*}_M[f^n_{\Delta v}]=\frac1{M} \sum_{k=1}^M f^{k,n}_{\Delta {\w}} - \lambda_M^{*,n}\left(\frac1{M} \sum_{k=1}^M \tilde{f}_{\Delta {\w}}^{k,n}-\tilde{\bf f}_{\Delta {\w}}^{n}\right).
\label{mcest2}
\ee
\end{enumerate}
\end{enumerate}
\end{algorithm}
Using such an approach, by similar arguments as in \cite{MSS, MS2}, one obtains the following error bound.
 {\begin{proposition}
Consider a deterministic scheme which satisfies \eqref{eq:det} in the velocity space for the solution of the homogeneous kinetic equation \eqref{eq:BH} with deterministic interaction operator $Q(f,f)$ and random initial data $f(\theta,{\w},0)=f_0(\theta,{\w})$. Assume that the initial data is sufficiently regular.
Then, the MSCV estimate defined in \eqref{mcest2} satisfies the error bound 
\be
\|\EE[f](\cdot,t^n)-\tilde{E}^{\lambda^*}_M[f^{n}_{\Delta {\w}}]\|_{{\LHBi}}
\leq  C(T,f_0)\left\{\sigma_{f^{\lambda_*}} M^{-1/2}+ \Delta {\w}^q\right\} 
\label{eq:errHMMC}
\ee
where $\sigma_{f^{\lambda_*}}=\|(1-\rho^2_{f,\tilde{f}})^{1/2}\var(f)^{1/2}\|_{\LL}$, 
the constant $C>0$ depends on the final time $T$ and the initial data $f_0$. 
\end{proposition}
\begin{proof}
The bound follows from
\begin{eqnarray*}
\nonumber
\|\EE[f](\cdot,t^n)-\tilde{E}^{\lambda^*}_M[f^{n}_{\Delta {\w}}]\|_{{\LHBi}}&& \\
\nonumber
&& \hskip -2cm \leq \|\EE[f](\cdot,t^n)-\tilde{E}^{\lambda^*}_M[f](\cdot,t^n)\|_{{\LHBi}}\\
\\[-.25cm]
\nonumber
&& \hskip -2cm + 
\|\tilde{E}_M^{\lambda^*}[f](\cdot,t^n)-\tilde{E}^{\lambda^*}_M[f_{\Delta \w}^n]\|_{{\LHBi}}\\
\nonumber
&& \hskip -2cm \leq  C(T,f_0)\left\{\sigma_{f^{\lambda_*}}  M^{-1/2}+ \Delta {\w}^q\right\}, 
\end{eqnarray*}
where the Monte Carlo bound in the first term now make use of \eqref{eq:var2} and the second term is bounded by the discretization error of the deterministic scheme.
\end{proof}}

Here we ignored the statistical errors due to the approximation of the control variate expectation and to the estimate of $\lambda^*$. 
Note that, since $\rho^2_{f,\tilde{f}}\to 1$ as $t\to\infty$ the statistical error will  vanish for large times. 
Therefore, the effect of the variance reduction becomes stronger in time and asymptotically the accuracy of the MSCV method depends only on the way the expectation of the control variate is evaluated.

\begin{remark}~
\label{rk:32}
\begin{itemize}
\item
If the collision frequency in the Boltzmann model, and therefore in the control variate BGK model, depends on the random input, then $\nu=\nu(z)$ and \eqref{eq:BGKexas} is no more valid. In this case we can write the expectation of the exact solution as
\be
\tilde{\bf f}(v,t)={\bf f}^{\infty}({\w})+\EE[e^{-\nu t} ({f}_0-{f}^{\infty})](v),
\label{eq:BGKexas2}
\ee
where now the second expectation depends on time and should be estimated during the simulation. Note, however, that for $t\to\infty$ we have $\tilde{\bf f}(v,t)\to {\bf f}^{\infty}({\w})$ and therefore the statistical error of the MSCV method again vanishes in time. 
To estimate the second term in \eqref{eq:BGKexas2} one can use for example
\[
\EE[e^{-\nu t} ({f}_0-{f}^{\infty})](v) \approx E_{M_E}[(e^{-\nu t}-e^{-\bar\nu t}) ({f}_0-{f}^{\infty})](v)+e^{-\bar\nu t}\left({\bf f}_0(v)-{\bf f}^\infty(v)\right)
\]
where $\bar{\nu}=\EE[\nu]$ and $M_E\gg M$.

\item
In a multi-fidelity setting \cite{ZLX} the control variate function computed through the reduced complexity kinetic model (BGK or the stationary state) represents the low fidelity solution whereas the full kinetic model yields the high fidelity solution. 


\end{itemize}
\end{remark}

\subsection{The space non homogeneous case}
The main difficulty one has to tackle when extending the MSCV method described above to space non homogeneous problems is that the moments of the solution, which may be necessary to define the function used as a control variate, change in time.  Therefore, they have to be estimated and cannot be computed in advance once for all as in the space homogeneous case. 

Again the idea is to compute the control variate function with a simplified model which can be evaluated at a fraction of the computational cost of the full model.
For example, 
if we integrate \eqref{eq:FP_general} with respect to the collision invariants $\phi({\w})=1,{\w},|{\w}|^2/2$ (we assume here the classical setting of conservation of mass, momentum and energy) we obtain the coupled system
 {\begin{eqnarray}
\label{eq:smallscale}
\partial_t U +{\rm div}_x
{\mathcal F}(U)+{\rm div}_x \int_{\RR^{d_{\w}}} {\w}\otimes \phi g\,d{\w}&=&0,\\
\frac{\partial}{\partial t} f + {\w} \cdot \nabla_x f &=& \frac1{\varepsilon} Q(f,f),
\label{eq:largescale}
\end{eqnarray} 
with initial data $f(z,x,v,0)=f_0(z,x,v)$.} 

 {In the above system $g=f-f^\infty$, $U=(\rho,\rho u,E)^T$, $\rho$, $u$ and $E$ are the density, mean velocity and energy of the gas defined as 
\be
 \rho=\int_{\RR^{d_{\w}}} f d{\w}, \ u=\frac{1}{\rho}\int_{\RR^{d_{\w}}} {\w}\, f d{\w}, \ E=\frac{1}{2}
 \int_{\RR^{d_{\w}}}|{\w}|^{2} f d{\w},
\label{eq:Mo} \ee
and moreover
\[
\int_{\RR^{d_{\w}}} \phi g d{\w} = 0,\quad {\mathcal F}(U)=\int_{\RR^{d_{\w}}} {\w}\otimes \phi f^\infty\,d{\w}=\left(
\begin{array}{c}
 \rho u   \\
 \rho u \otimes u + pI   \\
  Eu + pu   
\end{array}
\right),\quad \phi({\w})=1,{\w},|{\w}|^2/2,
\]
where $I$ is the $d\times d$ identity matrix, $p=\rho T$ is the pressure and $T=(2E/\rho-|u|^2)/d_v$ the temperature.}
Now, generalizing the space homogeneous method based on the local equilibrium $f_\infty$ as control variate we can consider the Euler closure as control variate, namely to assume $g=0$ in \eqref{eq:smallscale}. 
%
%

If we denote by $U_F=(\rho_F,u_F,E_F)^T$ the solution of the fluid model  
\be
\partial_t U_F +{\rm div}_x {\mathcal F}(U_F) = 0,
\label{eq:Euler}
\ee
for the same initial data, and with $f_F^\infty$ the corresponding equilibrium state, the control variate estimate based on $M$ i.i.d. samples reads 
\be
E^{\lambda}_M[f](x,v,t)=\frac1{M} \sum_{k=1}^M f^{k}(x,v,t) - \lambda\left(\frac1{M} \sum_{k=1}^M f_F^{\infty,k}(x,v,t)-{\bf f}_F^{\infty}(x,{\w},t)\right),
\label{eq:nesth1}
\ee
where ${\bf f}_F^{\infty}(x,{\w},t)$ is an accurate approximation of ${\mathbb E}[f_F^{\infty}(\cdot,x,{\w},t)]$. Consistency and unbiasedness of \eqref{eq:nesth1} for any $\lambda\in\RR$ follows again from Lemma \ref{le:1}. 

The fundamental difference is that now the variance of
\[
f^{\lambda}(z,x,v,t)=f(z,x,v,t)-\lambda({f_F}^{\infty}(z,x,{\w},t)-\EE[{f}_F^{\infty}](\cdot,x,{\w},t))
\]
will not vanish asymptotically in time since $f^\infty\neq f_F^\infty$, unless the kinetic equation is close to the fluid regime, namely for small values of the Knudsen number. 

We can state the following
\begin{theorem}
If $\var(f_F^\infty)\neq 0$ the quantity   
\be
\lambda^* = \frac{\cov(f,f_F^\infty)}{\var(f_F^\infty)}
\label{eq:lambdash}
\ee
minimizes the variance of $f^\lambda$ at the point $(x,v,t)$ and gives
\be
\var(f^{\lambda^*}) = (1-\rho_{f,f_F^\infty}^2)\var(f), 
\label{eq:var2h}
\ee
where $\rho_{f,f_F^\infty} \in [-1,1]$ is the correlation coefficient between $f$ and $f_F^\infty$. In addition, we have  
\be
\lim_{\varepsilon\to 0} \lambda^*(x,v,t) =1,\qquad \lim_{\varepsilon\to 0} \var(f^{\lambda^*})(x,v,t)=0\qquad \forall\, (x,v) \in \mathcal{D}\times\RR^{d_v}.
\label{eq:alambda}
\ee
\label{pr:2}
\end{theorem}
\begin{proof}
The first part of the theorem follows the same lines of Theorem \ref{pr:1} observing that
\[
\var(f^\lambda)=\var(f)+\lambda^2 \var(f_F^\infty)-2\lambda\cov(f,f_F^\infty).
\] 
Now, since as $\varepsilon\to 0$ from \eqref{eq:largescale}
 we formally have $Q(f,f)=0$ which implies $f=f^\infty$ and $f_F^\infty = f^\infty$, from \eqref{eq:lambdash} and \eqref{eq:var2h} we obtain \eqref{eq:alambda}.
\end{proof}
Similarly to the homogeneous case, the generalization to an improved control variate based on a suitable approximation of the kinetic solution can be done with the aid of a more accurate fluid approximation, like the compressible Navier-Stokes system, or a simplified kinetic model. In the latter case, following the approach of Section \ref{sec:BGK} we can solve a BGK model
\be
\frac{\partial}{\partial t} \tilde{f} + {\w} \cdot \nabla_x \tilde{f} = \frac{\nu}{\varepsilon} (\tilde{f}^\infty-\tilde{f}),
\label{eq:BGK}
\ee 
for the same initial data and apply the estimator
\be
\tilde{E}^{\lambda}_M[f](x,v,t)=\frac1{M} \sum_{k=1}^M f^{k}(x,v,t) - \lambda\left(\frac1{M} \sum_{k=1}^M \tilde{f}^{k}(x,v,t)-\tilde{\bf f}(x,{\w},t)\right),
\label{eq:nesth2}
\ee 
where $\tilde{\bf f}(x,{\w},t)$ is an accurate approximation of ${\mathbb E}[\tilde{f}(\cdot,x,{\w},t)]$. 

Now, if $\var(\tilde{f})\neq 0$ the gain in variance reduction obtained using the optimal value 
\be
\lambda^* = \frac{\cov(f,\tilde{f})}{\var(\tilde{f})}
\label{eq:lambdash2}
\ee
depends on the correlation between $f$ and $\tilde{f}$ and Theorem \ref{pr:2} holds true simply replacing $f^\infty_F$ with $\tilde f$.

Of course, solving the control variate models \eqref{eq:Euler} or \eqref{eq:BGK} requires the adoption of a suitable numerical method for the space (and velocity) approximation and, moreover, their accurate expectations have to be computed in time. The fundamental aspect is that, since we avoid the computation of the full kinetic model, simulating the control variate system is much cheaper and therefore its expected value can be evaluated more accurately. 

However, since the computational cost of the control variate as well as its variance are no more negligible we cannot ignore them. In the sequel, we assume that the control variate model is computed over a fine grid of $M_E \gg M$ samples and to use approximations
\[
{\bf f}_F^{\infty}(x,{\w},t)=E_{M_E}[f_F^\infty](x,v,t),\qquad \tilde{\bf f}(x,{\w},t)=E_{M_E}[\tilde f](x,v,t),
\]  
in the estimators \eqref{eq:nesth1} and \eqref{eq:nesth2}.

With the above notations, Algorithm \ref{alg:3} can be extended to the estimator \eqref{eq:nesth2} based on the BGK model \eqref{eq:BGK} as follows.

\begin{algorithm}[Multi-scale Control-Variate (MSCV) method]~
\label{alg:4}
\begin{enumerate}
\item {\bf Sampling}:
\begin{enumerate}
\item Sample $M_E$ i.i.d. initial data $\tilde f_0^k$, $k=1,\ldots,M_E$ from the random initial data $f_0$ and approximate these over the grid characterized by $\Delta x$ and $\Delta v$.
\item Sample $M \ll M_E$ i.i.d. initial data $\tilde f_0^k$, $k=1,\ldots,M$ from the random initial data $f_0$ and approximate these over the grid characterized by $\Delta x$ and $\Delta v$.
\end{enumerate}
\item {\bf Solving}
\begin{enumerate} 
\item For each realization $\tilde f_0^k$, $k=1,\ldots,M_E$ compute the solution of the control variate $\tilde{f}_{\Delta x,\Delta {\w}}^{k,n}$, $k=1,\ldots,M_E$ at time $t^n$ with a suitable deterministic scheme for \eqref{eq:BGK} and denote by
\[
\tilde{\bf f}_{\Delta x,\Delta {\w}}^{n}=\frac1{M_E}\sum_{k=1}^{M_E} \tilde{f}_{\Delta x,\Delta v}^{k,n}.
\] 
\item For each realization $f_0^k$, $k=1,\ldots,M$ 
the underlying kinetic equation (\ref{eq:BH}) is solved numerically by the corresponding deterministic solvers. We denote the solution at time $t^n$ by $f^{k,n}_{\Delta x,\Delta {\w}}$, $k=1,\ldots,M$. 
\end{enumerate} 
\item {\bf Estimating}: 
\begin{enumerate}
\item
Estimate the optimal value of $\lambda^*$ as
\[
{\lambda}_M^{*,n}= \frac{\sum_{k=1}^M (f_{\Delta x,\Delta {\w}}^{k,n}-E_M[f_{\Delta x,\Delta {\w}}^n])(\tilde f_{\Delta x,\Delta {\w}}^{k,n}-E_M[\tilde f_{\Delta x,\Delta {\w}}^n])}{\sum_{k=1}^M (\tilde f_{\Delta x,\Delta {\w}}^{k,n}-E_M[\tilde f_{\Delta x,\Delta {\w}}^n])^2}.
\]
\item
Compute the expectation of the random solution with the control variate estimator
\be
E^{\lambda^*}_{M,M_E}[f^n_{\Delta x,\Delta {\w}}]=\frac1{M} \sum_{k=1}^M f^{k,n}_{\Delta x,\Delta {\w}} - \lambda_M^{*,n}\left(\frac1{M} \sum_{k=1}^M \tilde{f}_{\Delta x,\Delta {\w}}^{k,n}-\tilde{\bf f}_{\Delta x,\Delta {\w}}^{n}\right).
\label{mcest2b}
\ee
\end{enumerate}
\end{enumerate}
\end{algorithm}
If, for simplicity, we assume that the deterministic error in space and velocity of the control variate solver is of the same order as that of the full model solver, by the same arguments in \cite{MSS, MS2}, we obtain the following error estimate 
 {
\begin{proposition}
Consider a deterministic scheme which satisfies \eqref{eq:det} for the solution of the kinetic equation of the form \eqref{eq:FP_general} with deterministic interaction operator $Q(f,f)$ and random initial data $f(\theta,x,{\w},0)=f_0(\theta,x,{\w})$. 
Assume that 
the initial data is sufficiently regular.
Then, the MSCV estimator \eqref{mcest2b} gives the following bound
\bea
\nonumber
&&\|\EE[f](\cdot,t^n)-{E}^{\lambda^*}_{M,M_E}[f^{n}_{\Delta x,\Delta {\w}}]\|_{\LLBi}\\[-.2cm]
\label{eq:errHMMC2}
\\
\nonumber
&& \hskip 4cm \leq {C}\left\{\sigma_{f^{\lambda_*}} M^{-1/2}+\tau_{f^{\lambda_*}} M_E^{-1/2}+\Delta {x}^p+\Delta {\w}^q\right\} 
\eea
where $\sigma_{f^{\lambda_*}}=\| (1-\rho^2_{f,\tilde{f}})^{1/2}\var(f)^{1/2}\|_{\LL}$, $\tau_{f^{\lambda_*}}=\| \rho_{f,\tilde{f}}\var(f)^{1/2}\|_{\LL}$, 
the constant $C=C(T, f_0)>0$ depends on the final time $T$ and the initial data $f_0$. 
\end{proposition}
\begin{proof}
We have
\begin{eqnarray*}
\nonumber
\|\EE[f](\cdot,t^n)-{E}^{\lambda^*}_{M,M_E}[f^{n}_{\Delta x,\Delta {\w}}]\|_{\LLBi} &&\\
&& \hskip -2.5cm\leq\|\EE[f](\cdot,t^n)-{E}^{\lambda^*}_{M}[f^{n}_{\Delta x,\Delta {\w}}]\|_{\LLBi} \\
\nonumber
&&\hskip -2.5cm+\|{E}^{\lambda^*}_{M}[f^{n}_{\Delta x,\Delta {\w}}]-{E}^{\lambda^*}_{M,M_E}[f^{n}_{\Delta x,\Delta {\w}}]\|_{\LLBi}\\
&&\hskip -2.5cm = I_1+I_2. 
\end{eqnarray*}
Since $\EE[f^{\lambda^*}]=\EE[f]$, the first term $I_1$ can be bounded similarly to \eqref{eq:errHMMC} to get \begin{eqnarray*}
\nonumber
&&\|\EE[f^{\lambda^*}](\cdot,t^n)-{E}^{\lambda^*}_{M}[f^{n}_{\Delta x,\Delta {\w}}]\|_{\LLBi} \\
&&\hskip 4cm \leq C_1(T, f_0)\left\{\sigma_{f^{\lambda_*}} M^{-1/2}+\Delta {x}^p+\Delta {\w}^q\right\}. 
\end{eqnarray*}
Using the fact that from \eqref{mcest2} and \eqref{mcest2b} we have
\[
{E}^{\lambda^*}_{M}[f^{n}_{\Delta x,\Delta {\w}}]-{E}^{\lambda^*}_{M,M_E}[f^{n}_{\Delta x,\Delta {\w}}] = \lambda_M^{*,n} \left(\EE[{\tilde f}^n_{\Delta x,\Delta v}]-E_{M_E}[\tilde f^n_{\Delta x,\Delta v}]\right).
\]
Now, ignoring, as before, the statistical error in estimating $\lambda^{*}$ and using \eqref{eq:lambdash2} the second term $I_2$ can be bounded by
\begin{eqnarray*}
\nonumber
&&\|{E}^{\lambda^*}_{M}[f^{n}_{\Delta x,\Delta {\w}}]-{E}^{\lambda^*}_{M,M_E}[f^{n}_{\Delta x,\Delta {\w}}]\|_{\LLBi}\\
&& \hskip 4cm \leq C_2(T, f_0)\left\{\tau_{f^{\lambda_*}} M_E^{-1/2}+\Delta {x}^p+\Delta {\w}^q\right\}. 
\end{eqnarray*}
\end{proof}}
From \eqref{eq:lambdash2} we have that  $\rho^2_{f,\tilde{f}}\to 1$ as $\varepsilon\to 0$, therefore from \eqref{eq:errHMMC2} in the fluid limit we recover the statistical error of the fine scale control variate model.

\begin{remark}~
\label{rk:33}
\begin{itemize} 
\item For multi-dimensional simulations, the approach can be simplified by computing the optimal value of $\lambda$ with respect to some moment or polynomial norm of the solution. 
For example, since we typically are interested in the evolution of the moments, one can eventually compute the optimal value of $\lambda$ with respect to a given moment like
\be
\lambda_{\phi}^* = \frac{\cov(m_{\phi}(f),m_{\phi}(f_F^\infty))}{\var(m_{\phi}(f_F^\infty))}
\label{eq:momcv}
\ee
so that $\lambda^*_{\phi}=\lambda^*_{\phi}(x,t)$. 
In this way the value of $\lambda^*$ used in the control variate, being independent of the velocity, is sub-optimal in terms of minimizing the variance of the solution but the storage requirements are strongly reduced.
\item Note that, by the central limit theorem we have
\[
\var(E_M[f])=M^{-1}\var(f),
\]
and therefore, using the independence of the estimators $E_M[\cdot]$ and $E_{M_E}[\cdot]$, the total variance of the estimator \eqref{mcest2b} for a general $\lambda$ is 
\bea
\nonumber
\var(E^{\lambda}_{M,M_E}[f]) &=&
M^{-1}\var(f-\lambda\tilde f)+M_E^{-1}\var(\lambda\tilde f)\\
\nonumber
&= & M^{-1}\left(\var(f)-2\lambda\cov(f,\tilde f)\right)+(M^{-1}+M_E^{-1})\lambda^2\var(\tilde f).
\eea
Minimizing the above quantity with respect to $\lambda$ yields the optimal value
\[
\tilde\lambda^* = \frac{M_E}{M+M_E}\lambda^*,
\]
where $\lambda^*$ is given by \eqref{eq:lambdash2}.
If we denote with $\C(f)$ the cost of the full model and with $\C({\tilde f})$ the cost of the control variate then the total cost of the estimator is $M \C(f) + M_E \C({\tilde f})$. Fixing a given cost for both models $M \C(f) = M_E \C({\tilde f})$, we also obtain 
\[
\tilde \lambda^* = \frac{\C(f)}{\C({f})+\C({\tilde f})}\lambda^*.
\]
This correction may be relevant in the cases when $\C(f)$ and $\C(\tilde f)$ do not differ too much.  
In our setting, however, $\C(f) \gg \C(\tilde f)$ (or equivalently  $M_E \gg M$) so that we can assume $\tilde \lambda^* \approx \lambda^*$.
\end{itemize}
\end{remark}

\section{Numerical examples}
\label{sec:4}
In this Section, we discuss several numerical examples with the aim of illustrating
the characteristics of the control variate strategies described in the previous Sections.
Due to the relevance in applications and the intrinsic difficulties of its numerical solution we focus our attention on the challenging case of the Boltzmann equation \eqref{eq:FP_general} with collision integral \eqref{eq:Boltzmann}. Of course, the same methodology can be applied to a large class of kinetic equations. 

We start by considering the space homogeneous case with uncertainties 
which permits an in depth analysis of the behaviors of the different
estimators and the different variance reduction techniques. In this case the control variates used have no impact over the simulation cost, therefore, we can assume to know with arbitrary accuracy their expected values. Subsequently, we will consider space non homogeneous problems with randomness in the initial data, in the collision frequency and in the boundary conditions. Here, the computational cost of the control variate model used plays a relevant role and we discuss in details the relationships between computational cost and performances of the methods.

%
%
%
%

\subsection{Homogeneous Boltzmann equation}
We consider the space homogeneous Boltzmann equation (\ref{eq:BH}). The velocity space is two dimensional $d_v=2$, the velocity domain is truncated to $[-v_{\min},v_{\max}]^2$ and discretized with $N_v=64^2$ points. The collision integral is discretized by the fast spectral algorithm \cite{MP,DP15} and the time integration has been performed with a $4$-th order Runge-Kutta method. We restrict ourselves to the analysis of the methods regardless of a computational cost analysis. In fact, the cost of both the estimators considered, i.e. the one based on the equilibrium distribution and the one based on the BGK model, is negligible as discussed in Remark \ref{rk:31}. In the special case in which the randomness comes from the collision kernel we refer to Remark \ref{rk:32} for the computation of the BGK control variate.

\subsubsection{Test 1. Uncertain initial data}
The initial condition is a two bumps problem with uncertainty
\be
f_0(z,v)=\frac{\rho_0}{2\pi} \left(\exp\left(-\frac{|v-(2+sz)|^2}{\sigma}\right)+\exp\left({-\frac{|v+(1+sz)|^2}{\sigma}}\right)\right)
\ee  
with $s=0.2$, $\rho_0=0.125$, $\sigma=0.5$ and $z$ uniform in $[0,1]$. We choose $v_{\min}=v_{\max}=16$. We perform two different computations, with $\Delta t=1$ and final time $T_f=70$ to observe the long time behavior of the solution and with $\Delta t=0.05$ and $T_f=10$ to detail the first part of the relaxation process. 
In Figure \ref{Figure1}, we report the expectation of the initial data, the expectation of the final equilibrium state and their difference. 

\begin{figure}[ht!]
	\begin{center}
		\includegraphics[width=0.47\textwidth]{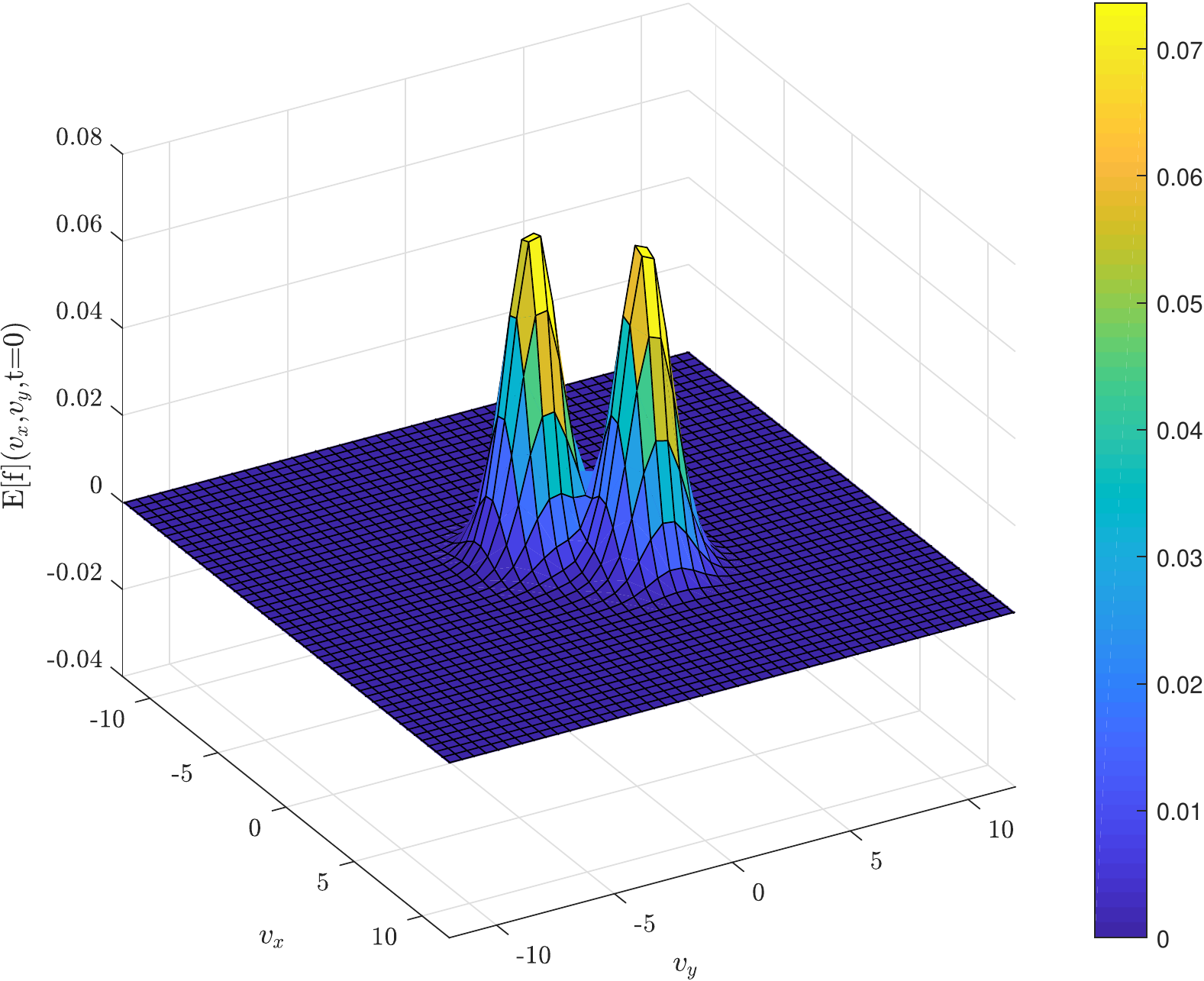}\hspace{0.5cm}
		\includegraphics[width=0.47\textwidth]{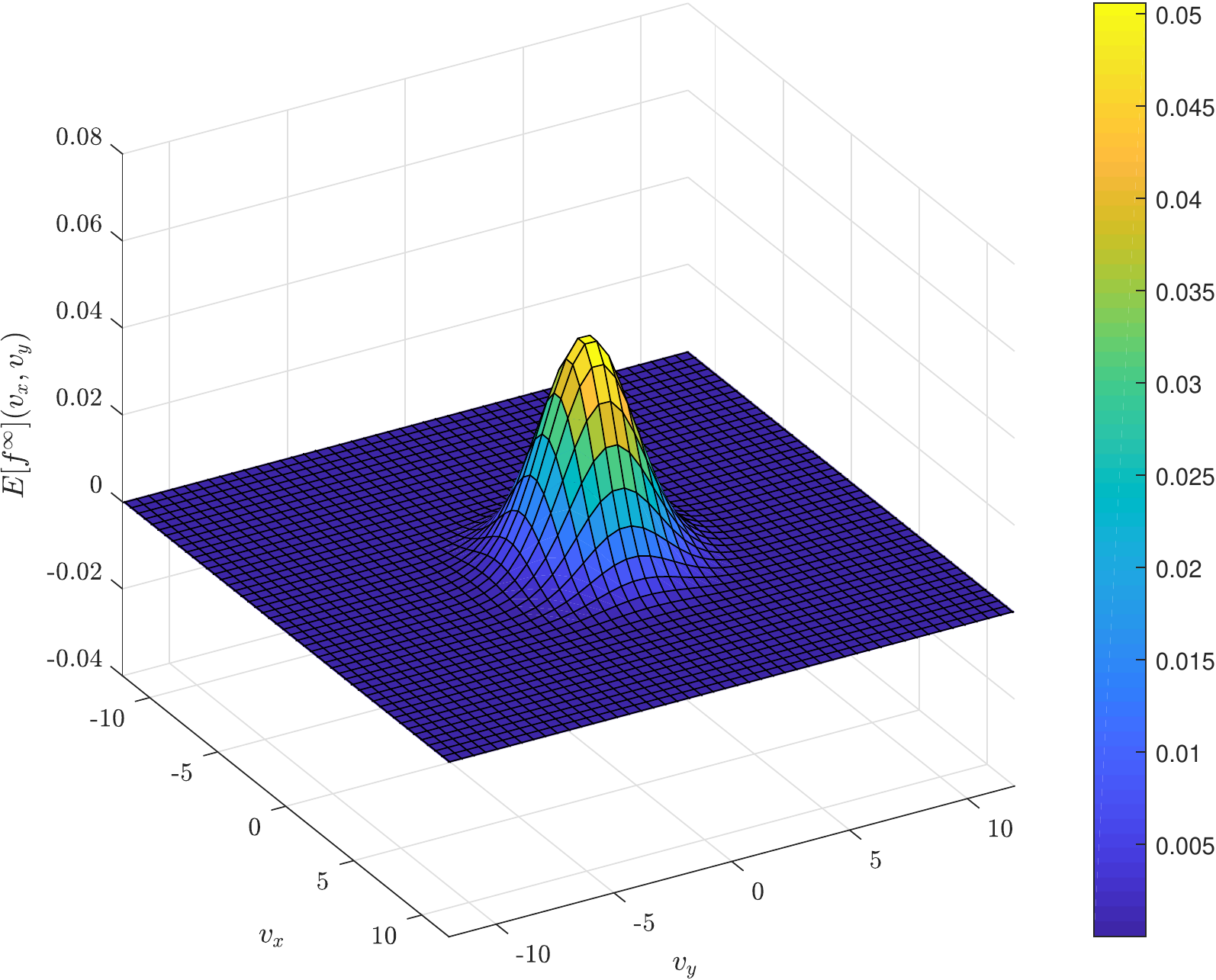}\\
		\vspace{1cm}
		\includegraphics[width=0.47\textwidth]{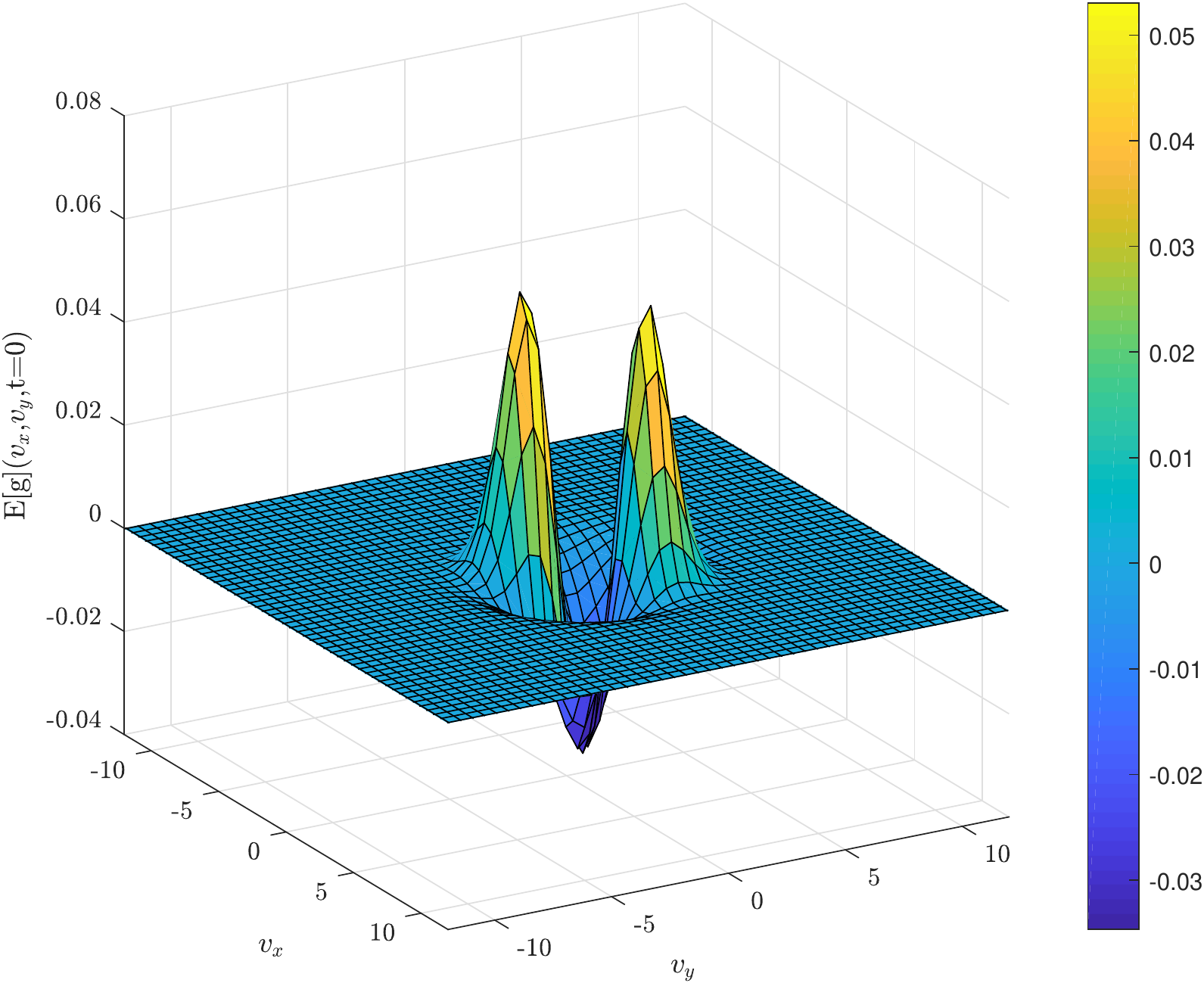}
		\captionof{figure}{Test 1. Top left: Expectation of the initial distribution $\EE[f_0](v)$. Top right: Expectation of the equilibrium distribution $\EE[f^\infty](v)$. Bottom: Expectation of $\EE[g_0](v)=\EE[f_0-f^\infty](v)$.}
		\label{Figure1}
	\end{center}
\end{figure}

Next, in Figure \ref{Figure2} we consider the $L_2$ error with respect to the random variable and the $L_1$ error in the velocity field for the various methods in the computation of the expected value for the distribution function $\EE[f](v,t)$.
On the bottom, we report the long time behavior while on the top a magnification of the numerical solution at the beginning of the relaxation process. The number of samples used to compute the expected solution for the Boltzmann equation is $M=10$ (left images) and $M=100$ (right images). 
\begin{figure}[ht!]
	\begin{center}
		\includegraphics[width=.45\textwidth]{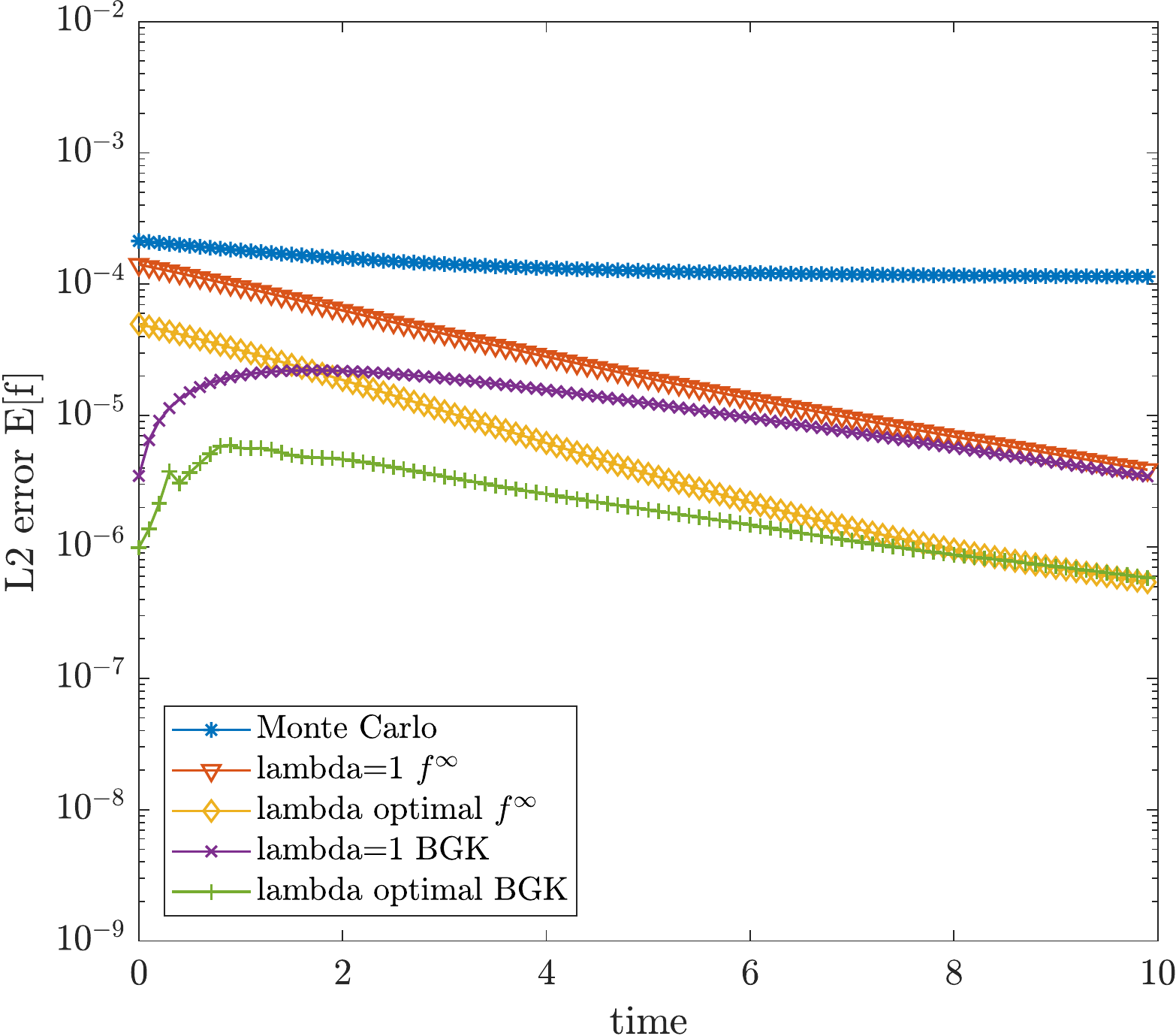}\hspace{1cm}
		\includegraphics[width=.45\textwidth]{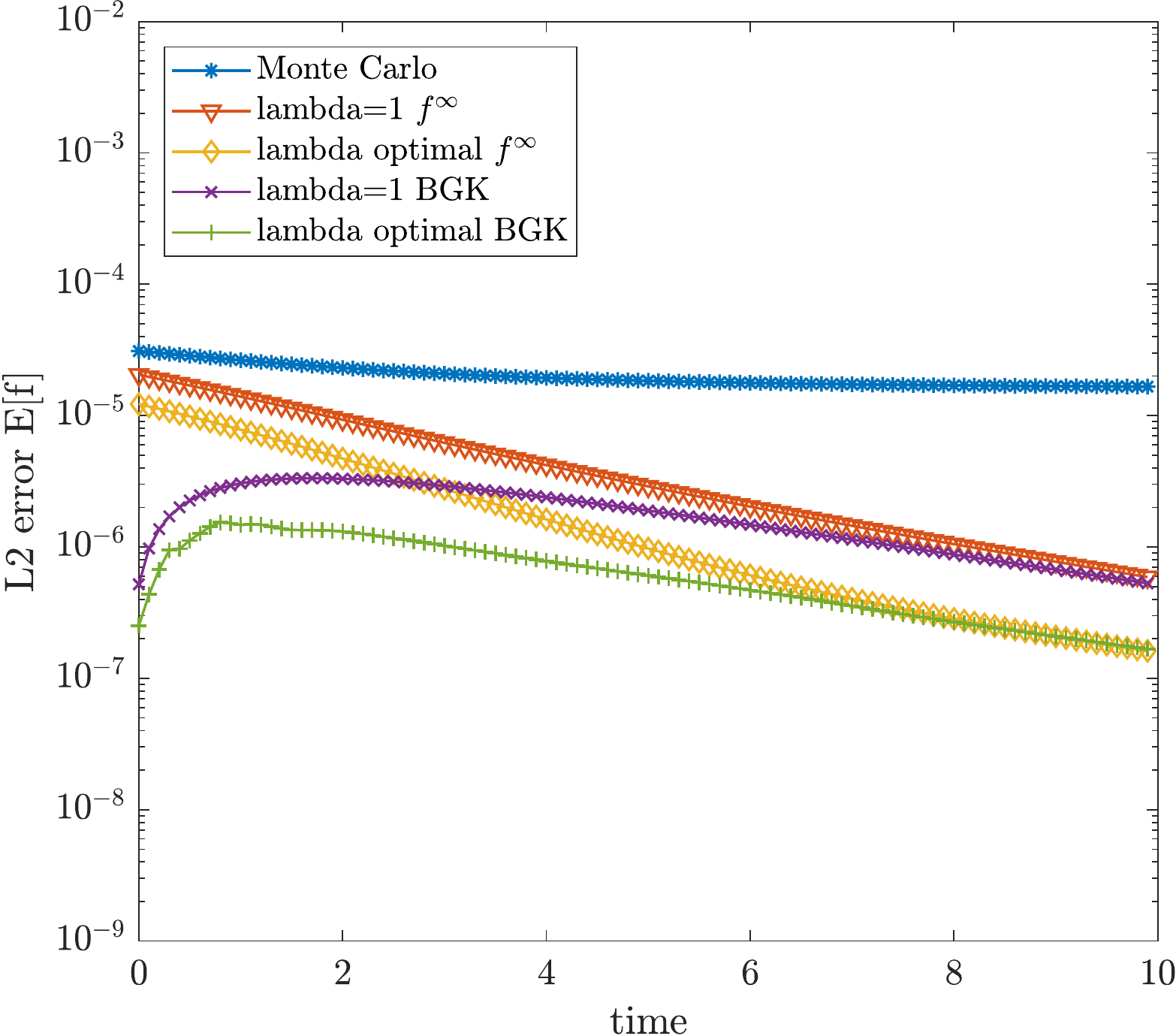}\\
		\vspace{1cm}
		\includegraphics[width=.45\textwidth]{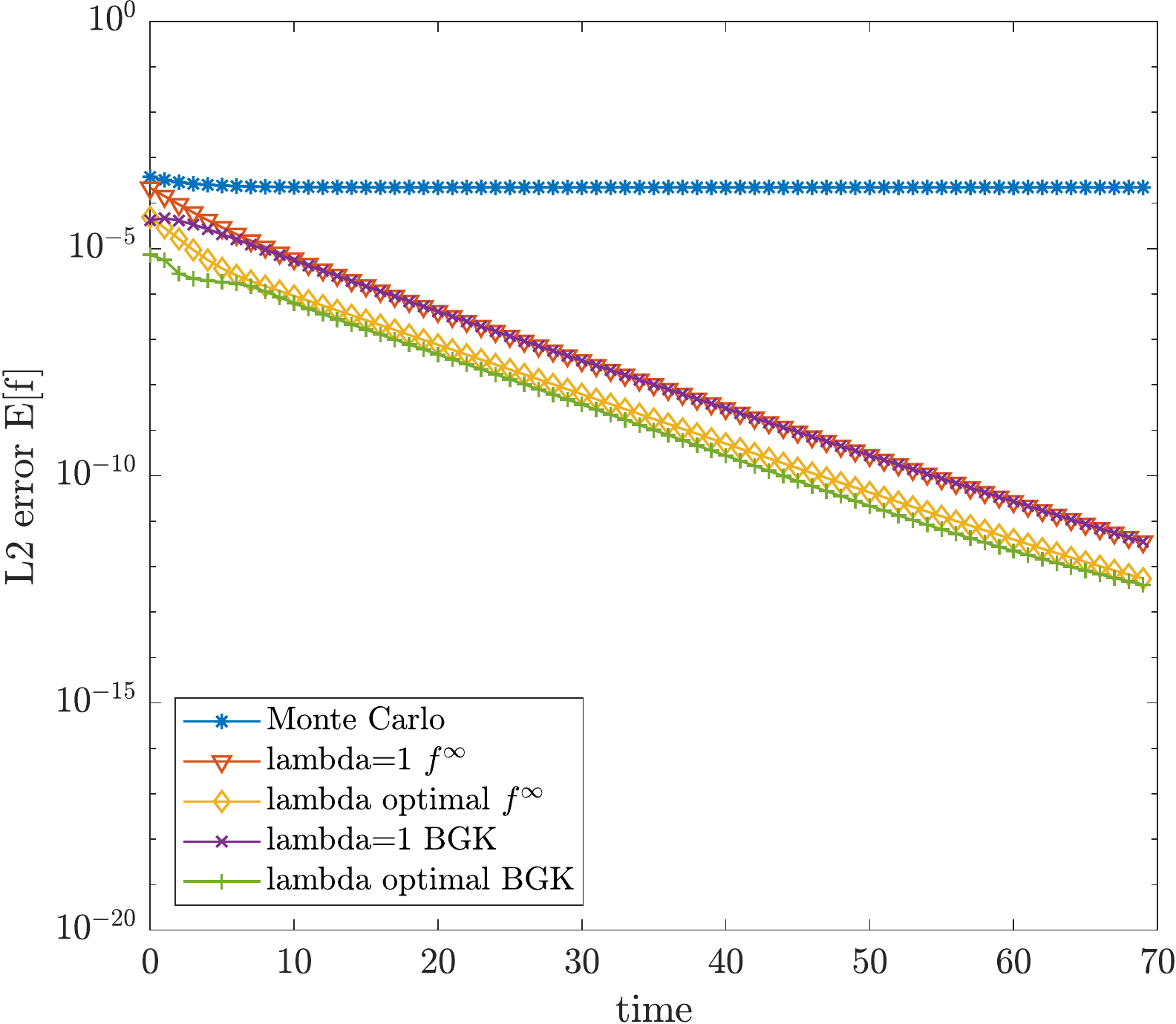}\hspace{1cm}
		\includegraphics[width=.45\textwidth]{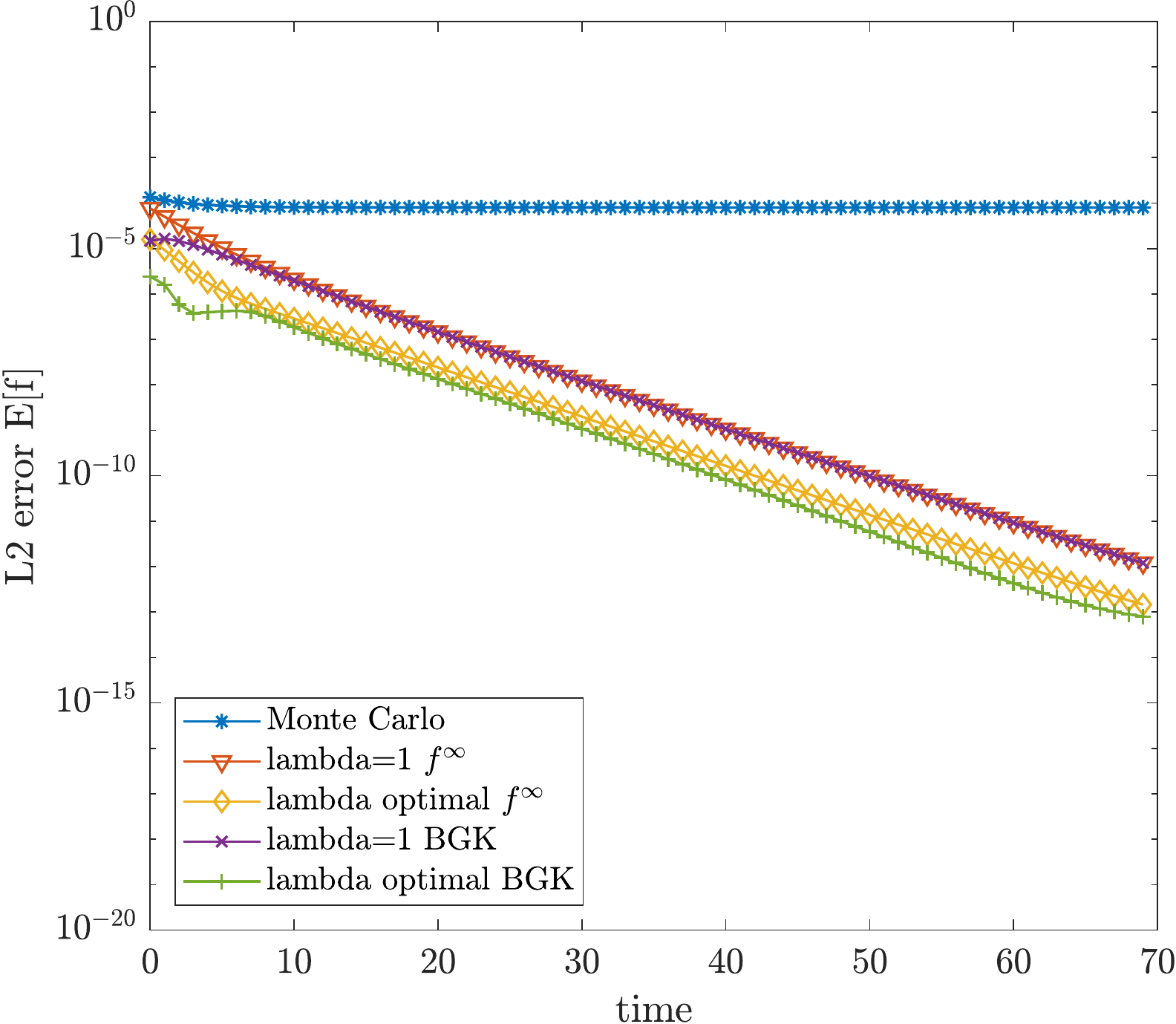}
		\captionof{figure}{Test 1. $L_2$ norm of the error for the expectation of the distribution function for the Monte Carlo method and the MSCV method for various control variates strategies. Top left: $M=10$ samples and long time behavior. Top right: $M=100$ samples  and magnification for $t\in[0,10]$. Bottom left: $M=10$ samples and long time behavior. Bottom right: $M=100$ samples and magnification for $t\in[0,10]$.}\label{Figure2}
	\end{center}
\end{figure}

The gain in accuracy obtained with the MSCV methods is of several orders of magnitudes with respect to standard Monte Carlo. In particular, the case of optimal $\lambda^*$ and BGK control variate approach is the one which gives the best results during all the time evolution of the solution. Note that, thanks to the properties of MSCV methods asymptotically the solution errors are close to machine precision. 

In Figure \ref{Figure4} we report the shape of the optimization coefficient $\lambda^*(v,t)$ at different times in the case of the BGK control variate strategy.
\begin{figure}[ht!]
	\begin{center}
		\includegraphics[width=.48\textwidth]{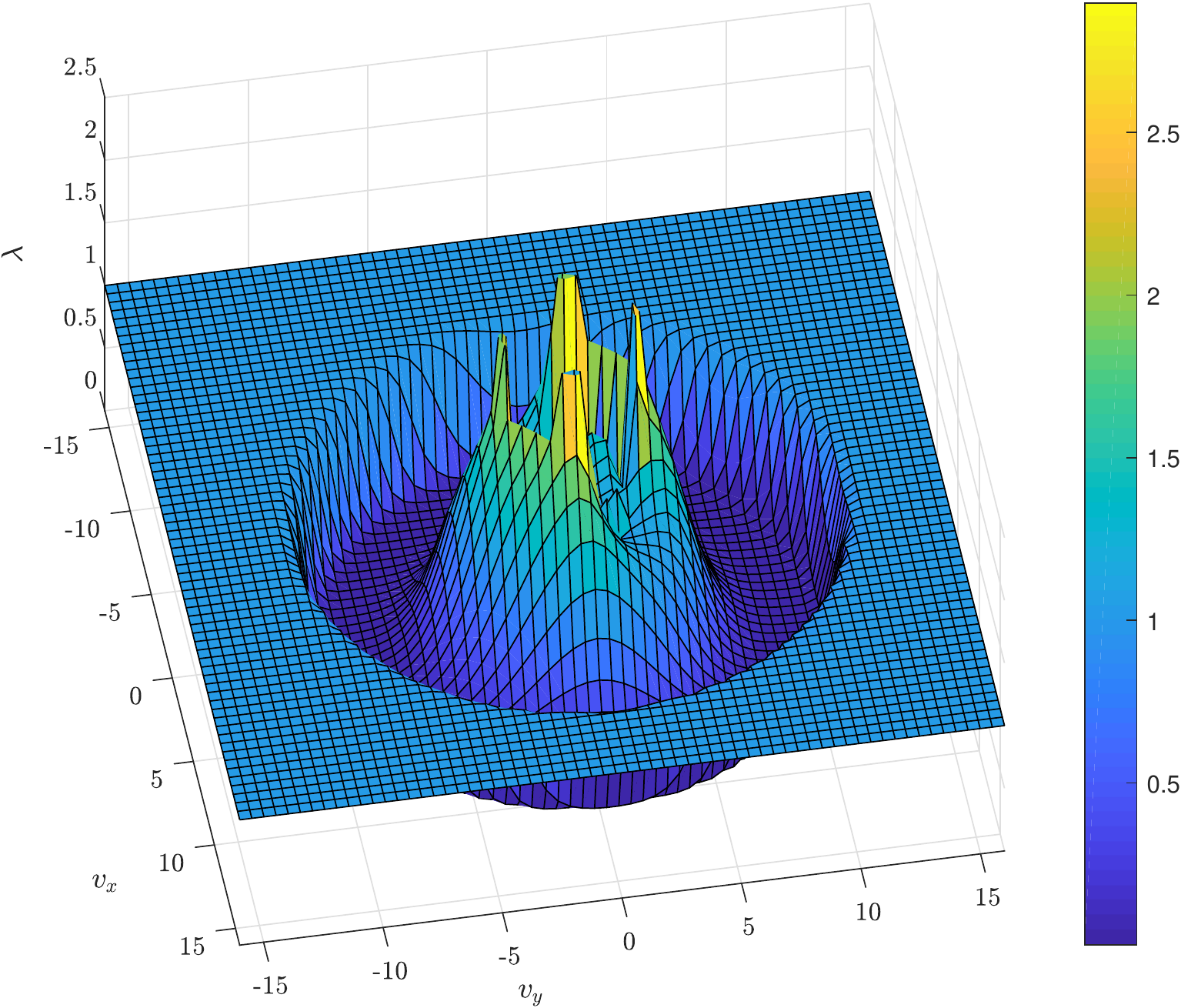}\hspace{0.5cm}
		\includegraphics[width=.48\textwidth]{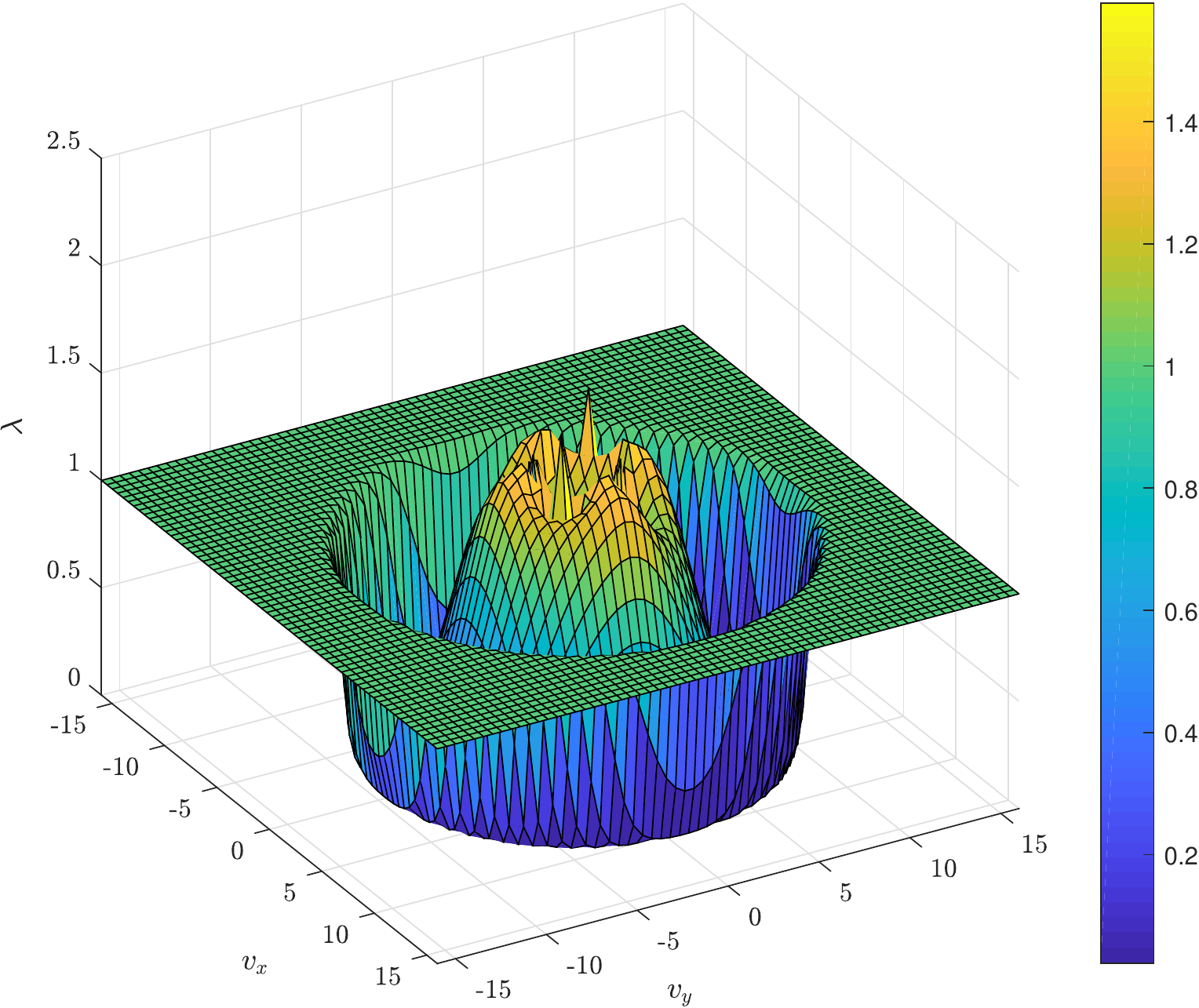}\\
		\vspace{1cm}
		\includegraphics[width=.48\textwidth]{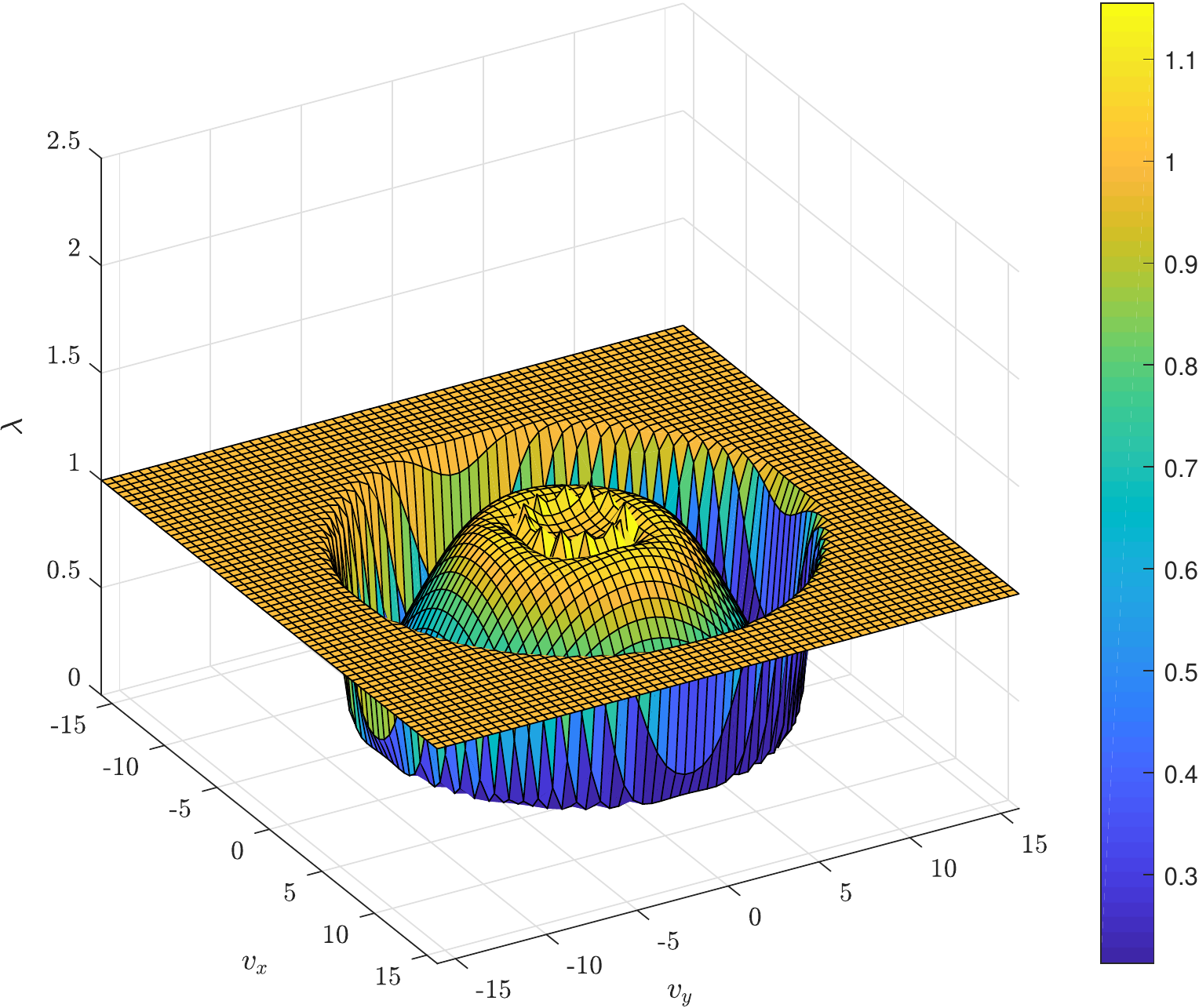}\hspace{0.5cm}
		\includegraphics[width=.48\textwidth]{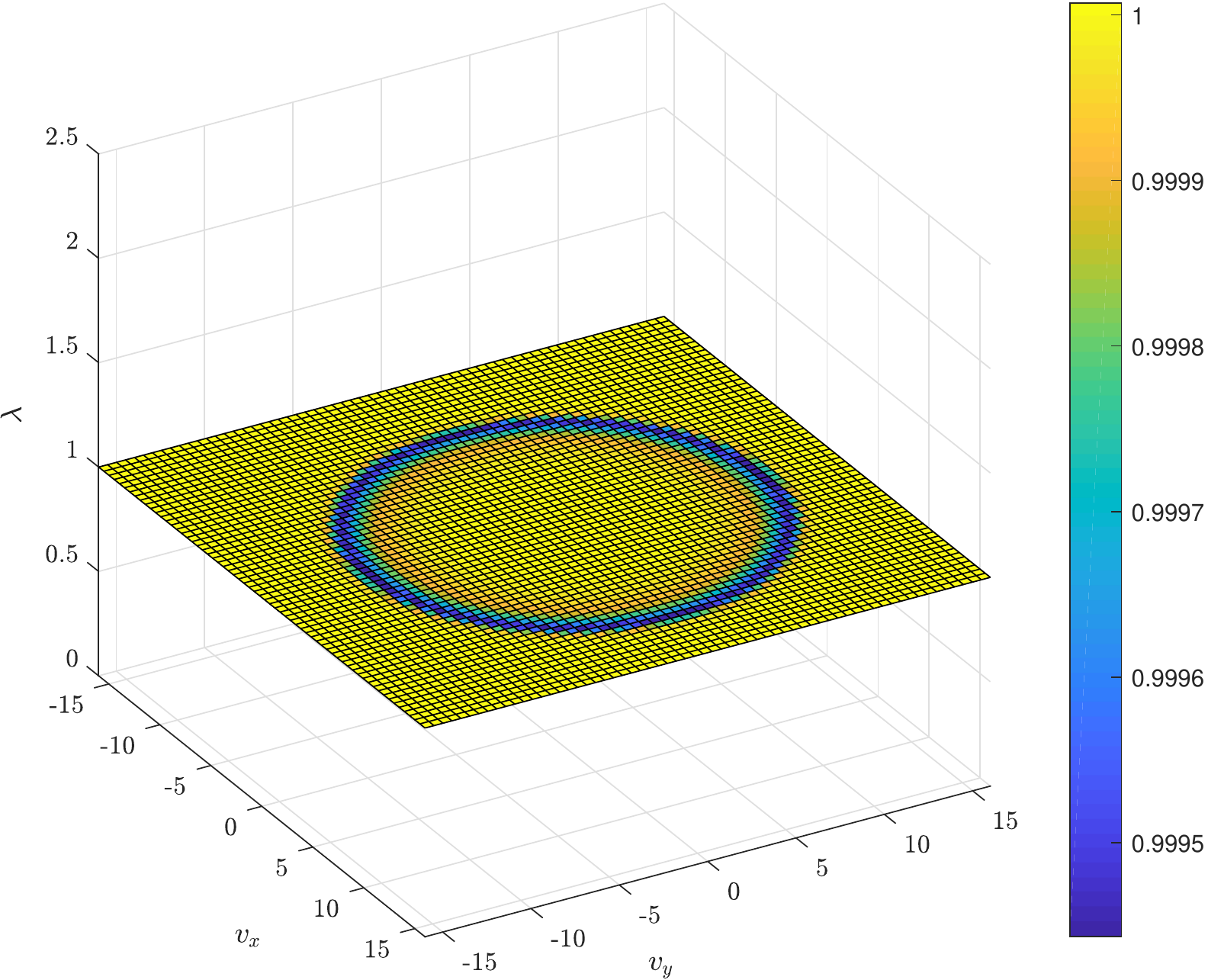}
		\captionof{figure}{Test 1. Optimal $\lambda^*(v,t)$ for the MSCV method based on a BGK control variate strategy. Top left: $t=0$. Top right: $t=5$. Bottom left: $t=10$. Bottom right: $t=50$.}
		\label{Figure4}
	\end{center}
\end{figure} 
One can clearly observe that, as theoretically predicted, $\lambda^*\rightarrow 1$ when $t\rightarrow \infty$ while initially its value is different from unity in the regions of the velocity space in which the perturbation $f_0-f^\infty$ is larger.

\subsubsection{Test 2. Uncertain collision kernel}
In this second test case, we consider the following initial deterministic condition 
\be
f_0(z,v)=\frac{1}{2\pi^2} |v|^2 \exp\left({-\frac{|v|^2}{2}}\right).
\ee
The uncertainty is in the frequency of collision and it is such that the collision kernel is
\be
B(z)=1+sz
\ee
with $s=0.2$ and $z$ uniform in $[0,1]$.

As a consequence, the equilibrium distribution is independent of the random variable $z$ and it is given by
\[
M(v)=\frac{\rho}{2\pi } \exp\left({-\frac{|v|^2}{2}}\right).
\]

In Figure \ref{Figure3}, we report the initial data, the final equilibrium state and their difference. 

As for the first case, we perform two different computations, one with $\Delta t=1$ and final time $T_f=150$ and one with $\Delta t=0.05$ and $T_f=10$. In this second case, the behavior of the standard MC method is different. In fact, since the initial data and the long time behavior are deterministic, the error is zero at $t=0$ and as $t\to +\infty$ while it grows at the beginning of the relaxation phase and it exponentially decays to zero in the second part of the relaxation phase. 
Clearly, in this case, the MSCV method with the equilibrium state as control variate coincides with the standard MC approach since $\EE[{f}^{\infty}]({\w})={f}^{\infty}({\w})$.

On the contrary, as shown in Figure \ref{Figure5}, the MSCV method with BGK control variate strategy is able to reduce the error of several orders of magnitude in the central part of the relaxation phase collapsing to the MC method when the distribution function approaches the equilibrium state. More precisely, following Remark \ref{rk:31}, we used a BGK solution with $
\nu(z)=\nu(1+sz)$ and a number of samples for the control variate variable $M_E=10^5$.
%
\begin{figure}[ht!]
	\begin{center}
		\includegraphics[width=0.47\textwidth]{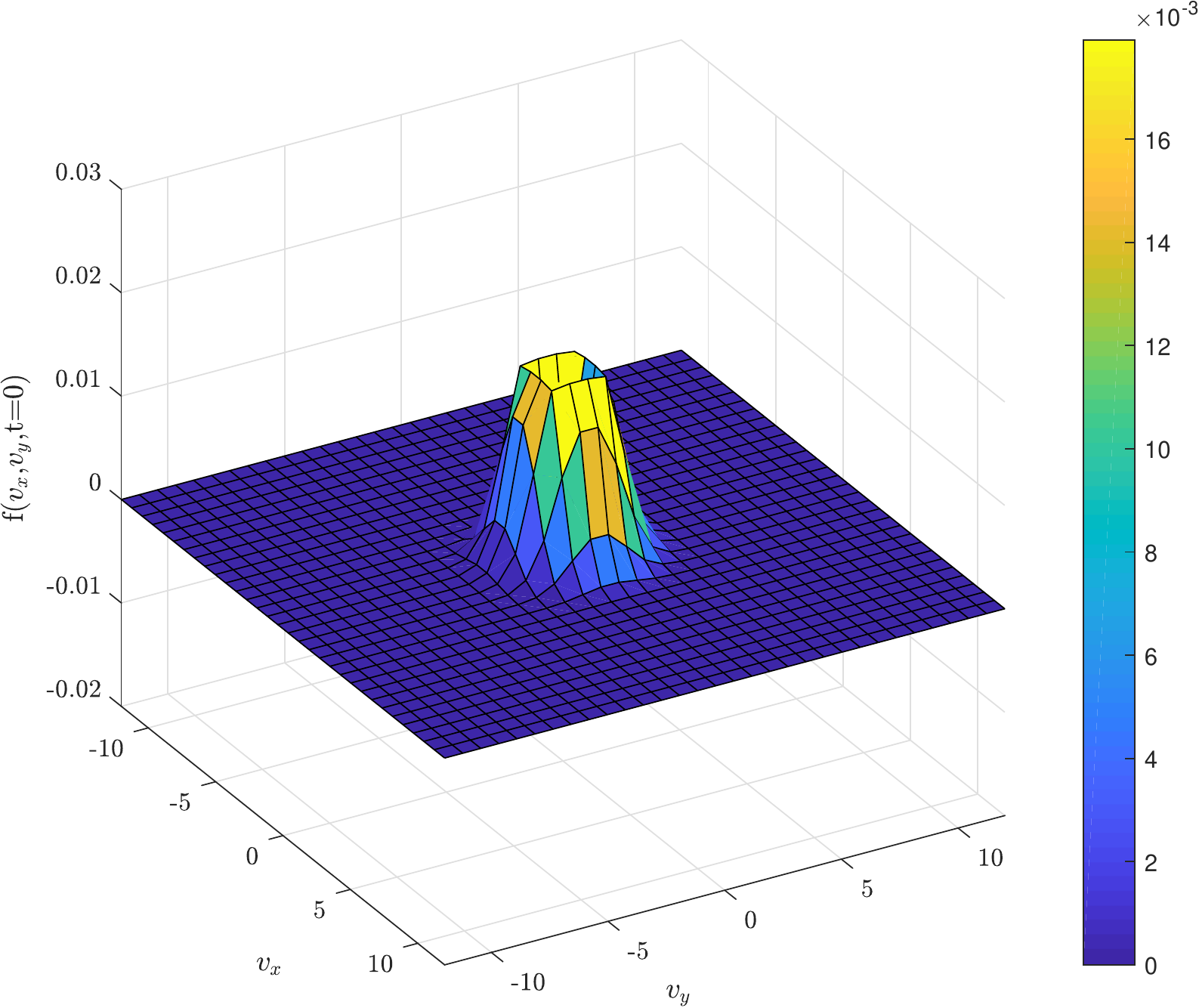}\hspace{0.5cm}
		\includegraphics[width=0.47\textwidth]{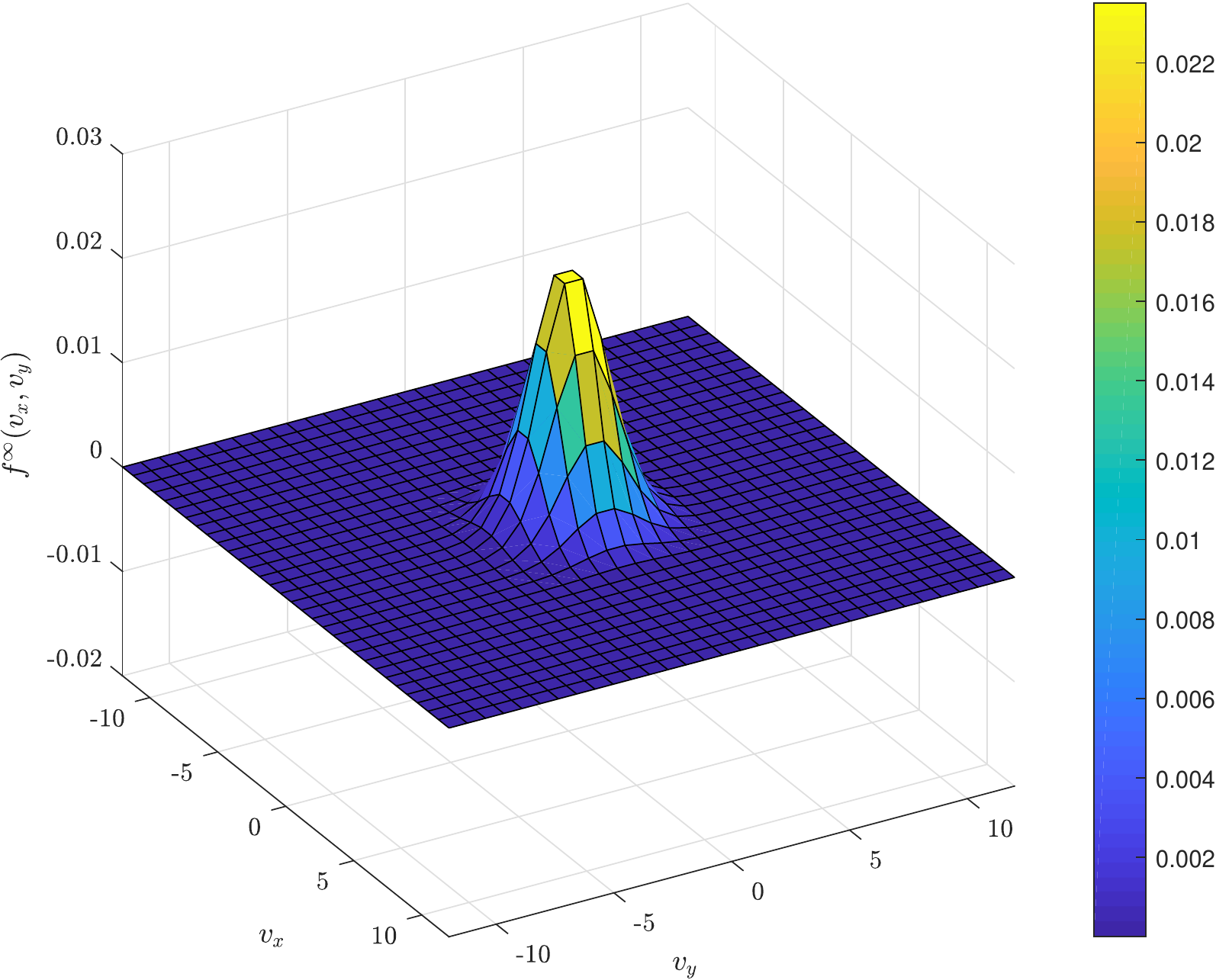}\\
		\vspace{1cm}
		\includegraphics[width=0.47\textwidth]{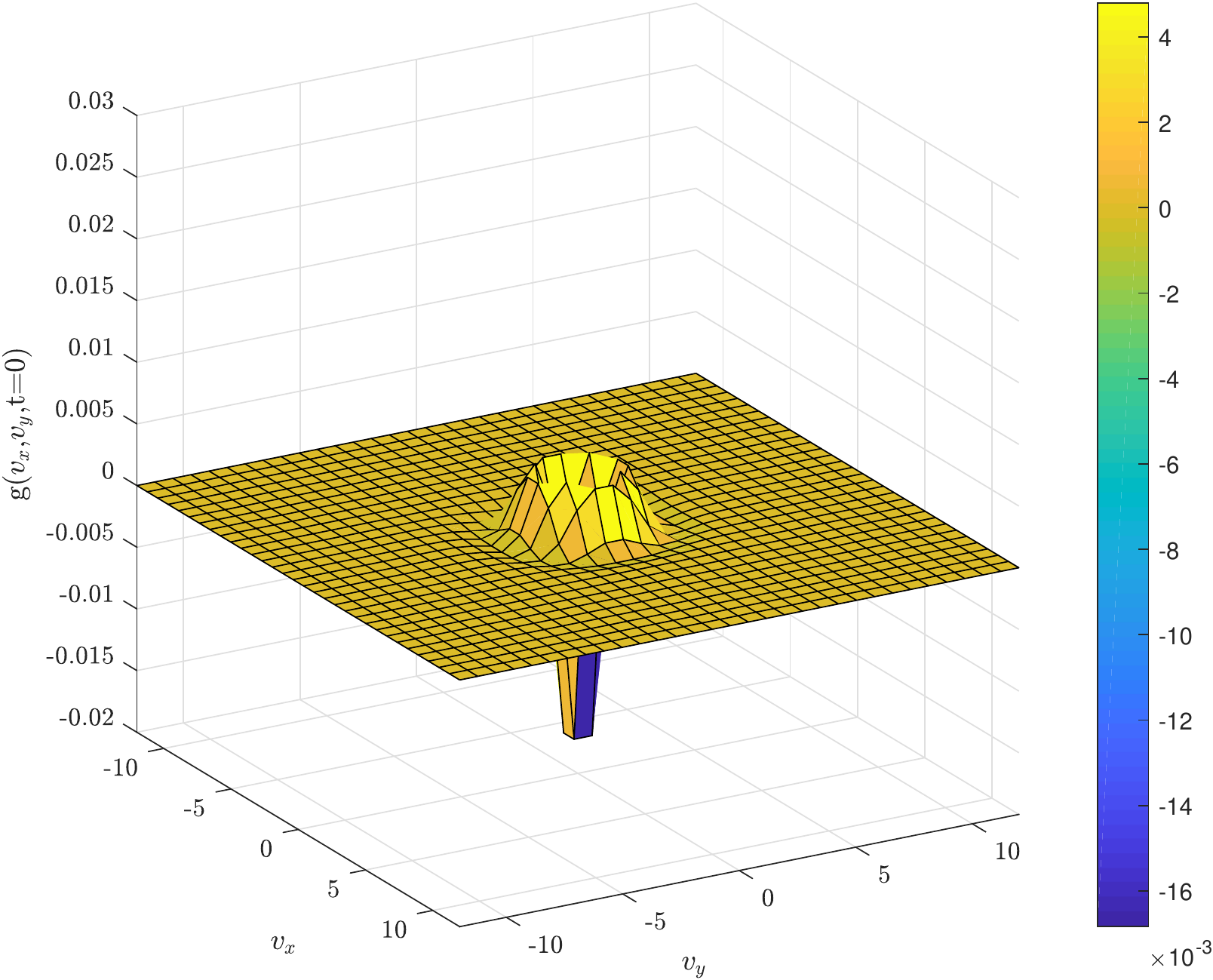}
		\captionof{figure}{Test 2. Top left: Deterministic initial distribution $f_0(v)$. Top right: Deterministic equilibrium distribution $f^\infty(v)$. Bottom: Deterministic difference $g_0(v)=f_0(v)-f^\infty(v)$.}		\label{Figure3}
	\end{center}
\end{figure}

\begin{figure}[ht!]
	\begin{center}
		\includegraphics[width=.45\textwidth]{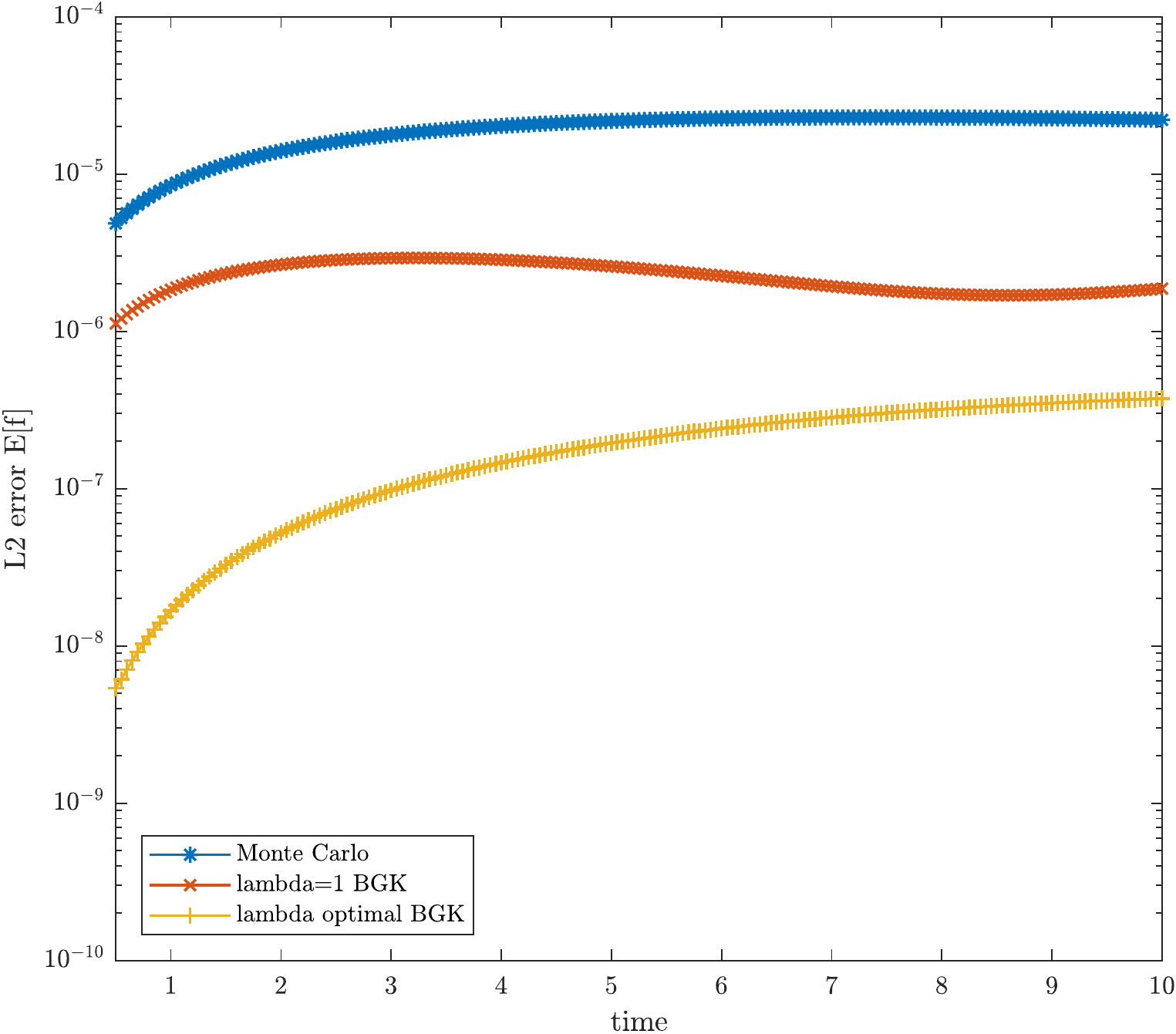}\hspace{1cm}
		\includegraphics[width=.45\textwidth]{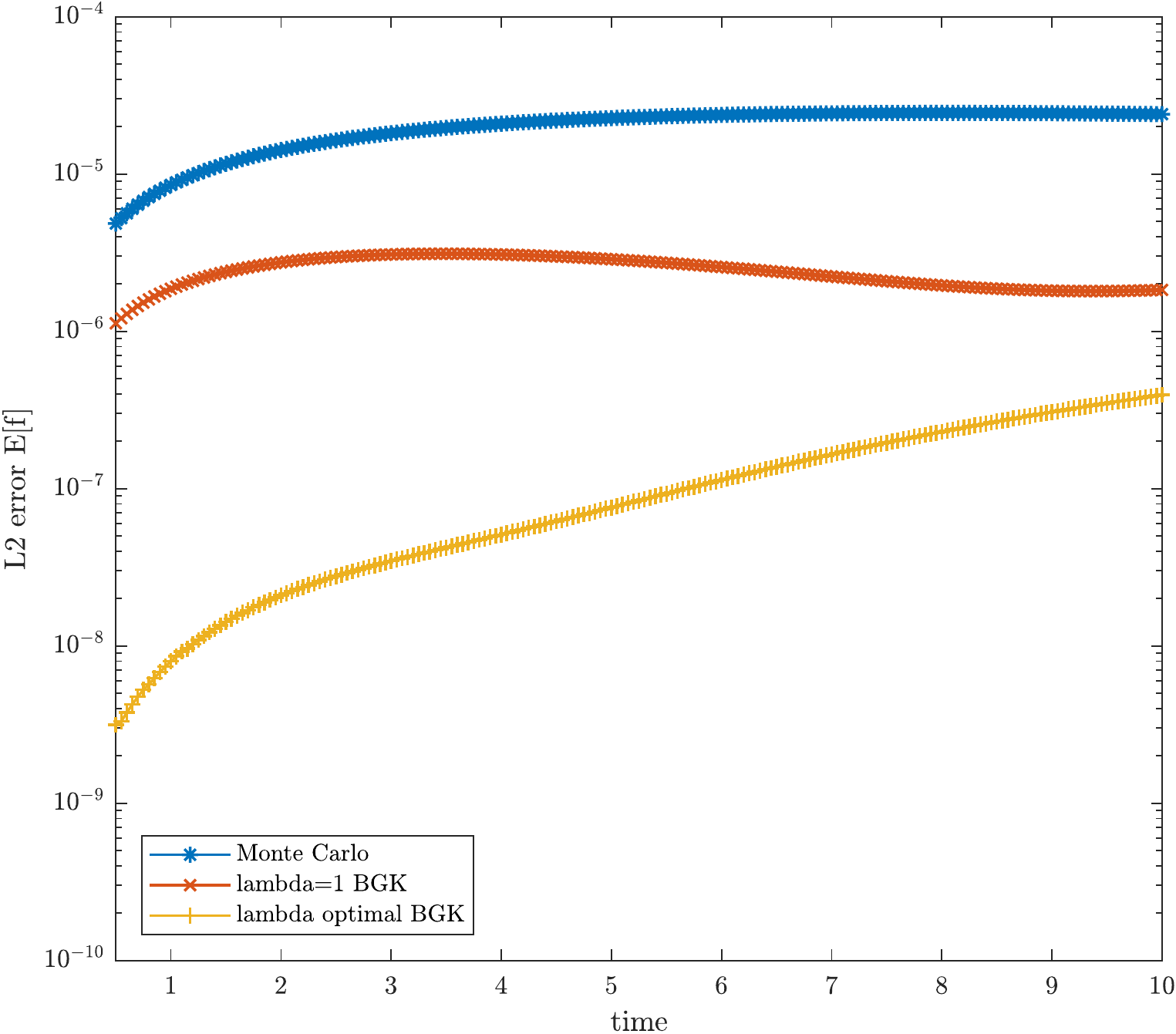}\\
		\vspace{1cm}
		\includegraphics[width=.45\textwidth]{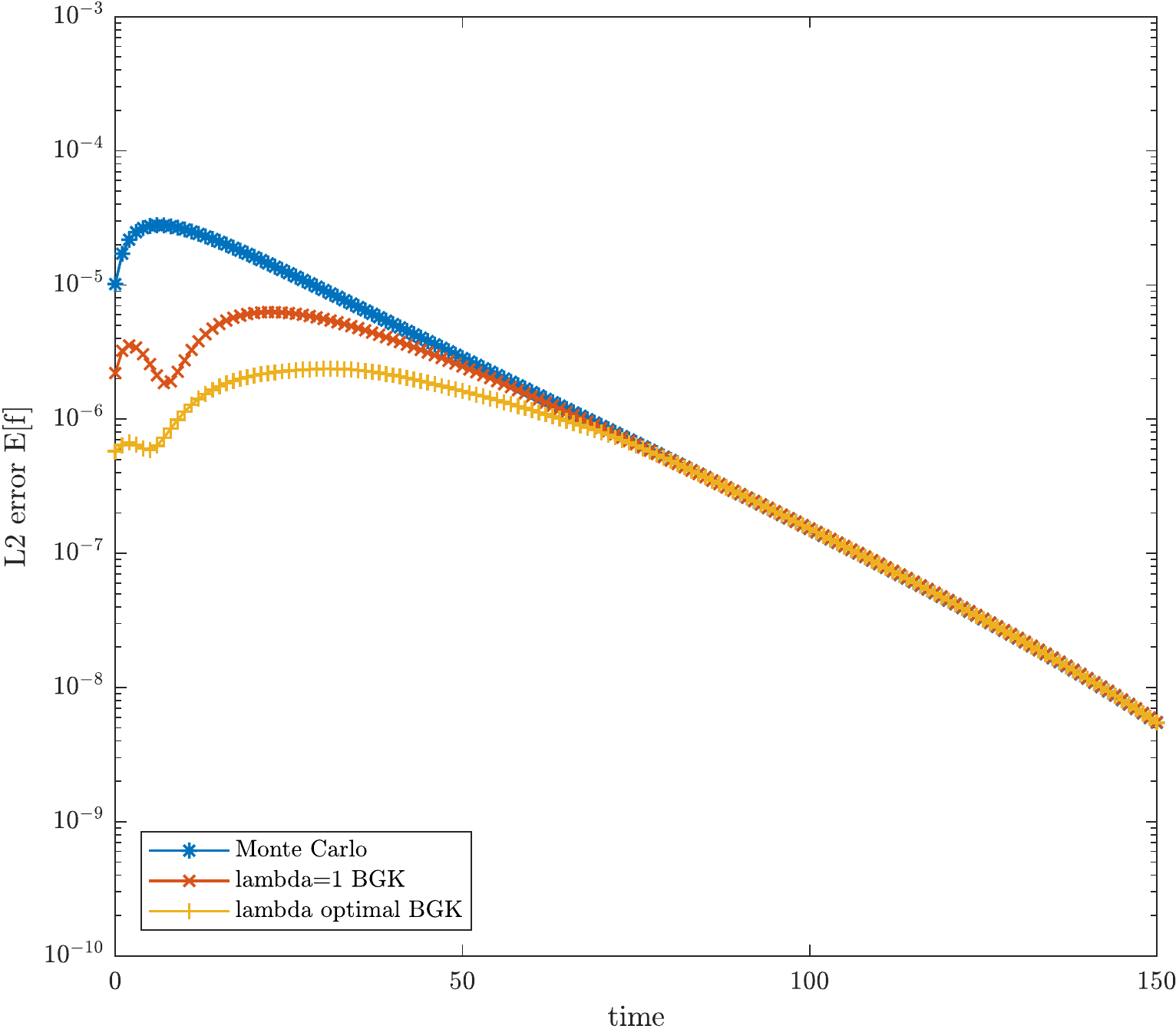}\hspace{1cm}
		\includegraphics[width=.45\textwidth]{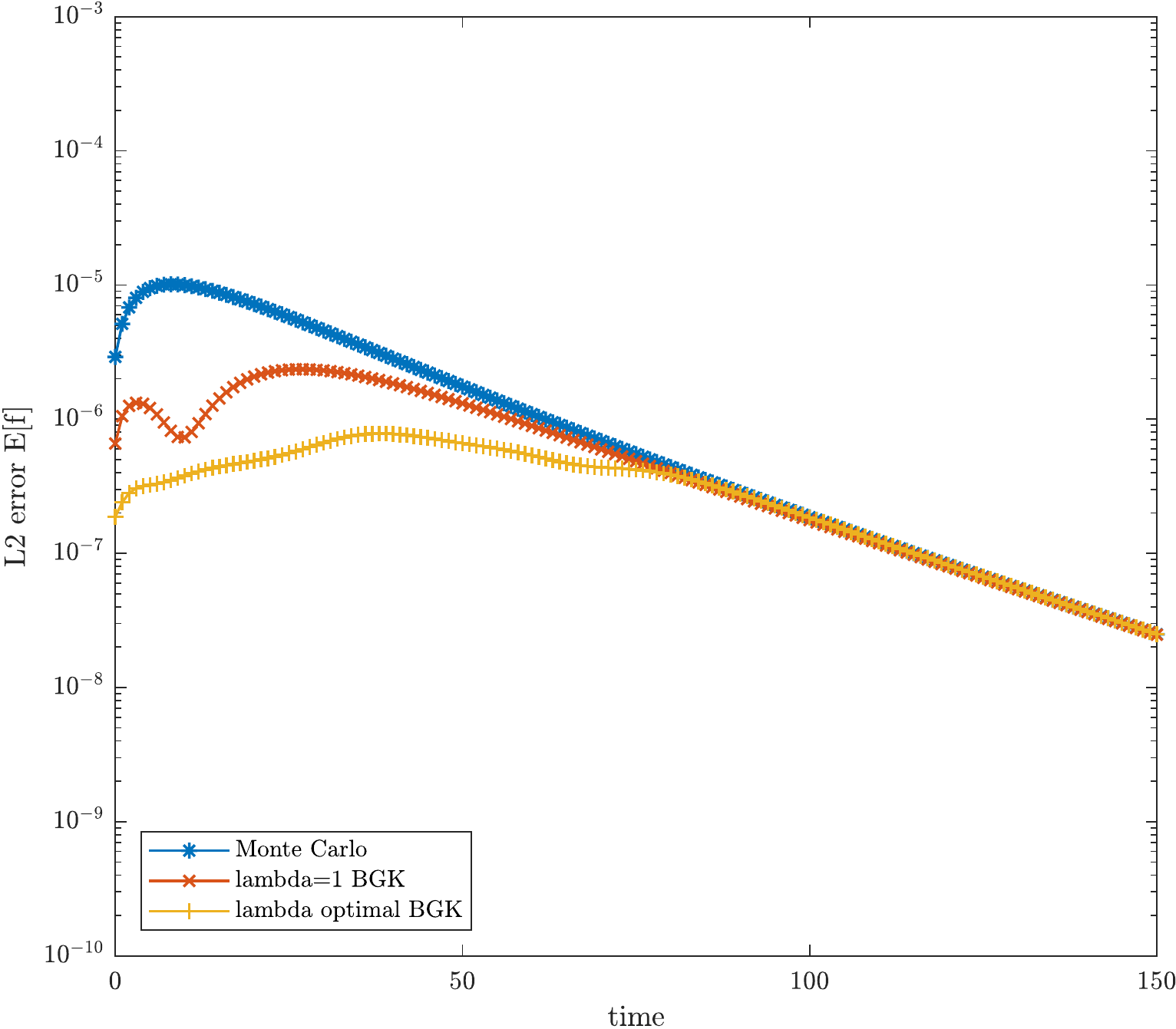}
		\captionof{figure}{Test 2. $L_2$ norm of the error for the expectation of the distribution function for the Monte Carlo method and the MSCV method for various control variates strategies. Top left: $M=10$ samples and long time behavior. Top right: $M=100$ samples  and magnification for $t\in[0,10]$. Bottom left: $M=10$ samples and long time behavior. Bottom right: $M=100$ samples and magnification for $t\in[0,10]$.}		\label{Figure5}
	\end{center}
\end{figure}

In Figure \ref{Figure7}, we finally report the shape of the optimization coefficient $\lambda^*(v,t)$ at different times. Here, as opposite to the first test, the optimization coefficient does not go to one as $t\rightarrow +\infty$. In fact, in this case $\var(f^{\infty})(v)=0$, and the optimal $\lambda^*$ cannot be computed in this situation. 
\begin{figure}[ht!]
	\begin{center}
		\includegraphics[width=.48\textwidth]{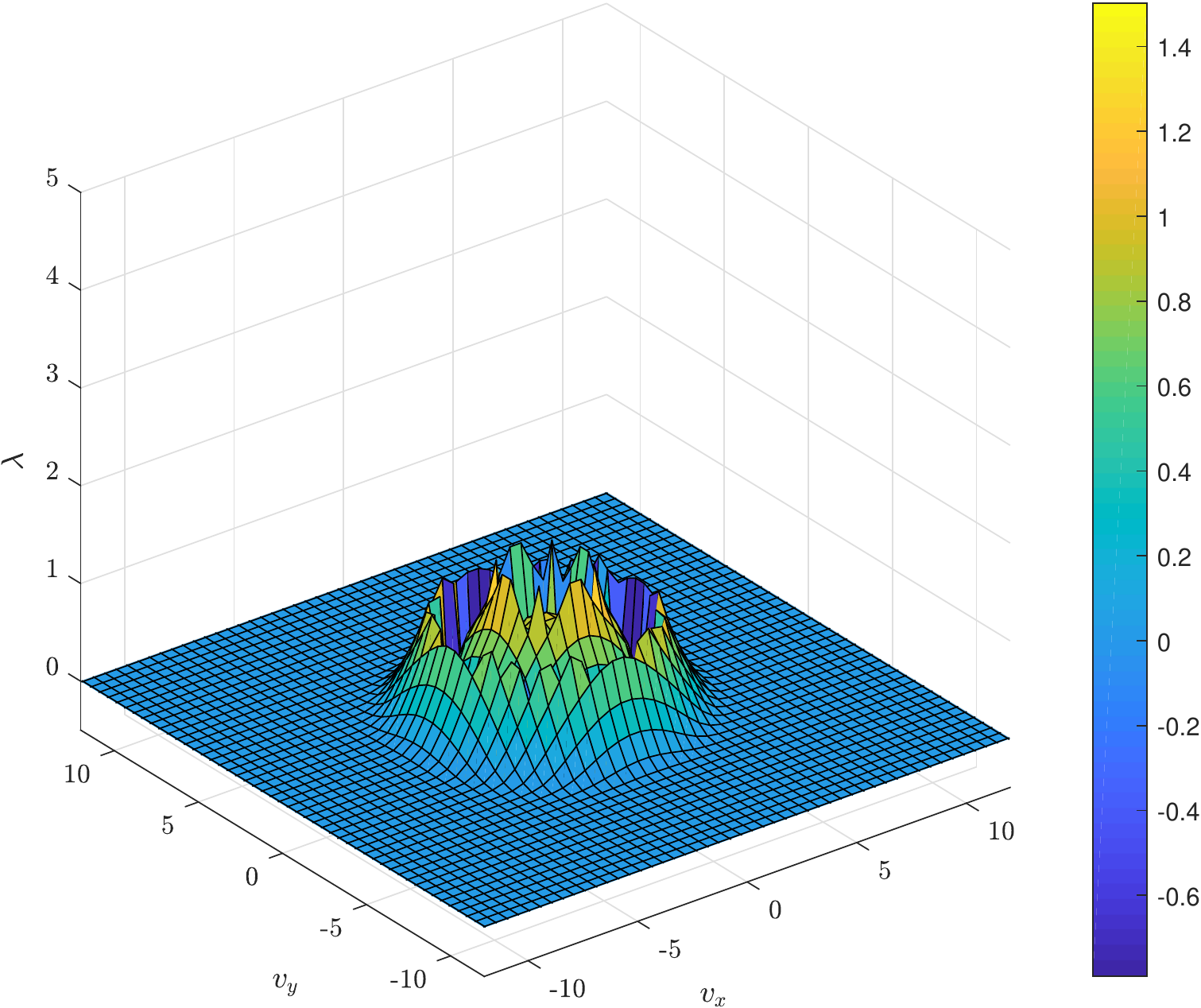}\hspace{0.5cm}
		\includegraphics[width=.48\textwidth]{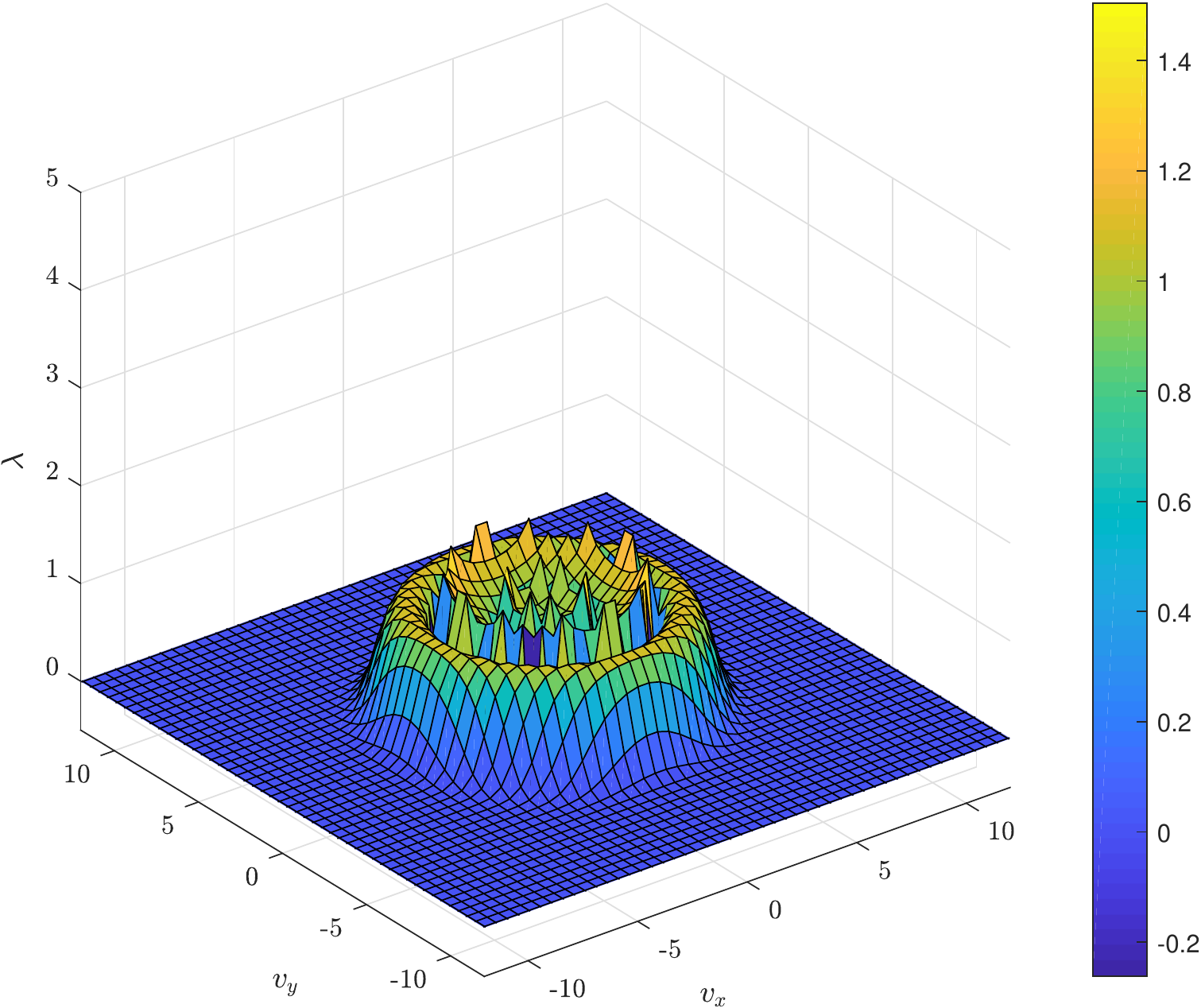}\\
		\vspace{1cm}
		\includegraphics[width=.48\textwidth]{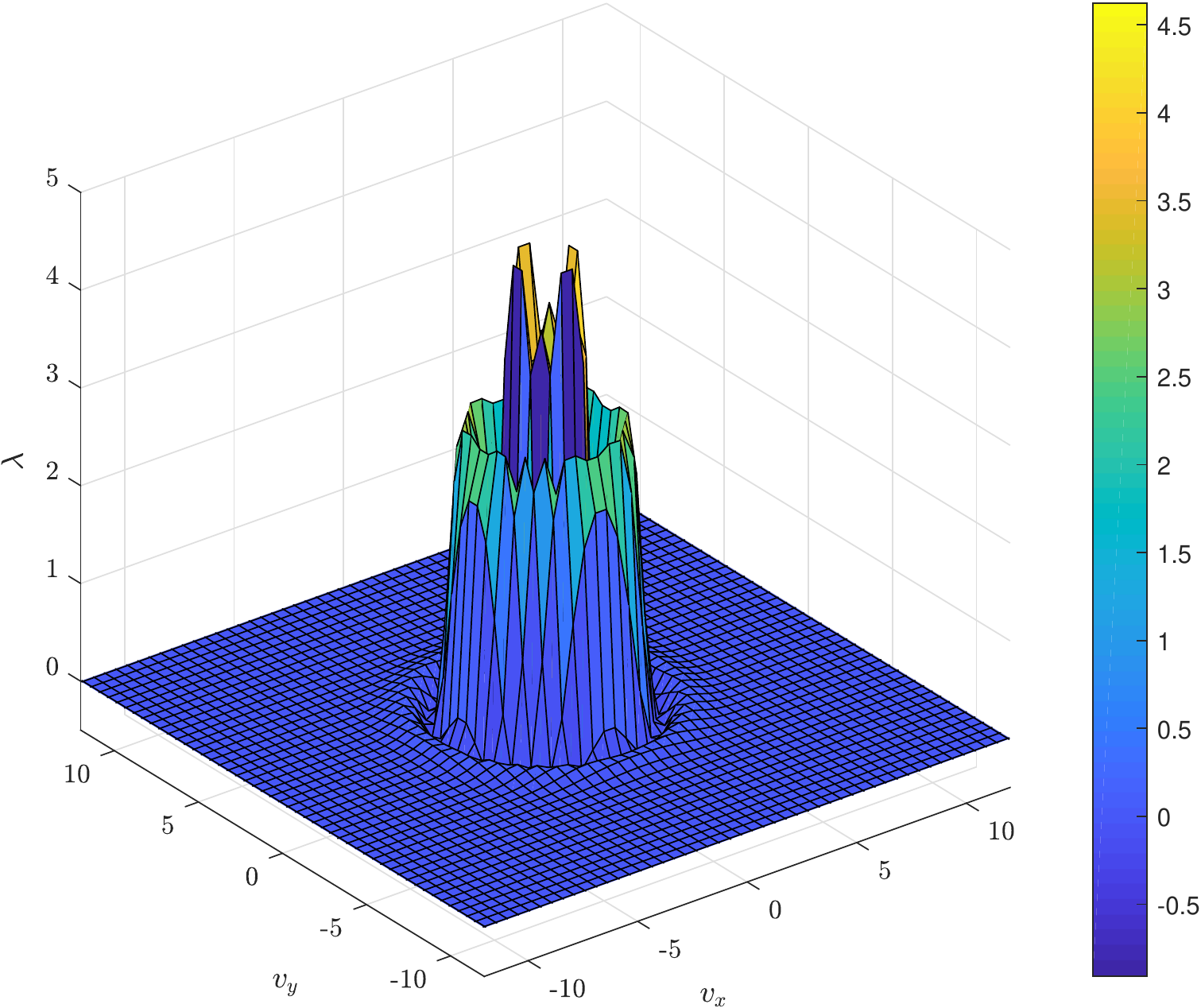}\hspace{0.5cm}
		\includegraphics[width=.48\textwidth]{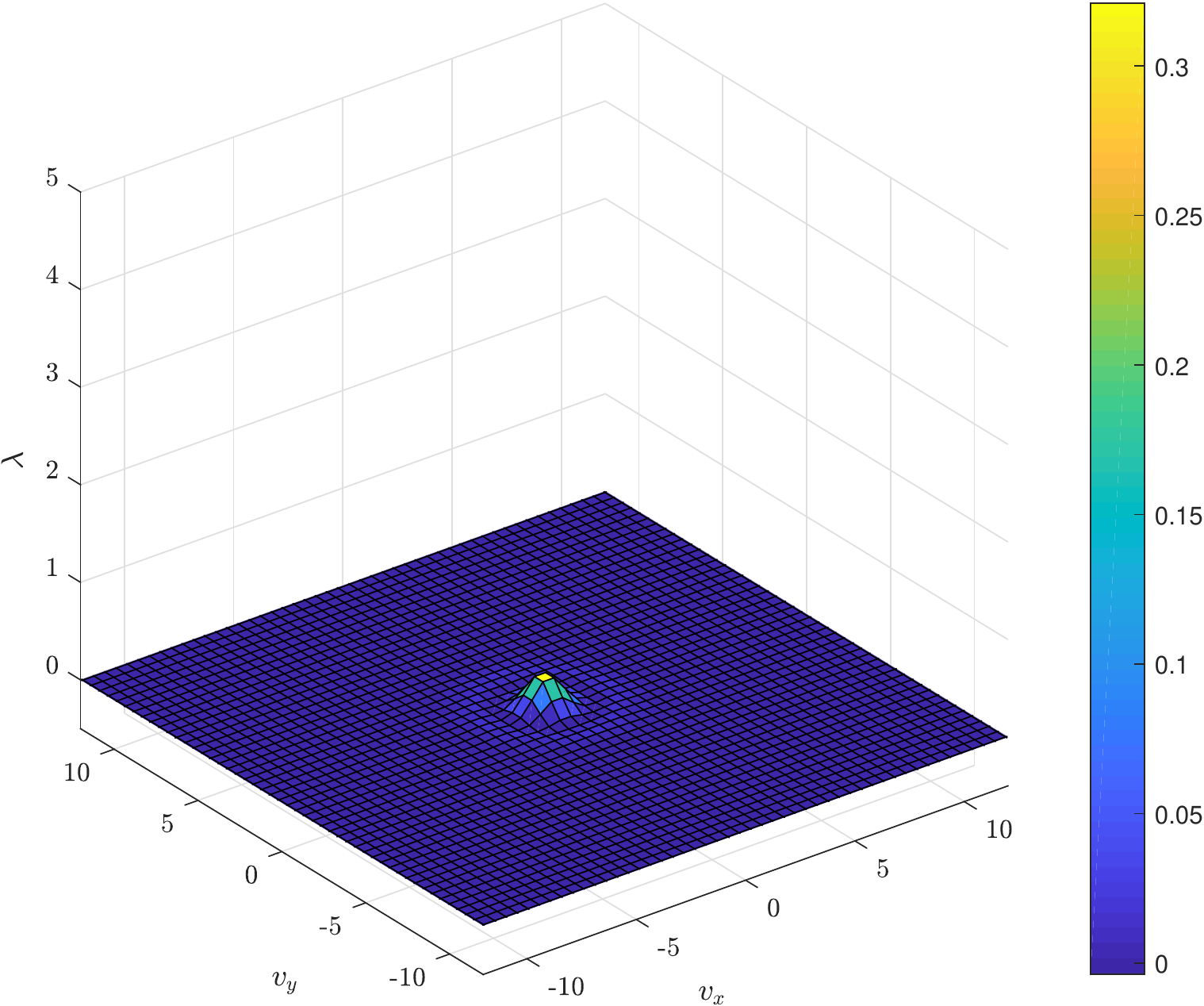}
		\captionof{figure}{Test 2. Optimal $\lambda^*(v,t)$ for the time dependent control variate strategy. Top left: $t=1$. Top right: $t=10$. Bottom left: $t=50$. Bottom right: $t=100$.}
		\label{Figure7}
	\end{center}
\end{figure}

\subsection{Space non homogeneous Boltzmann equation}
In this case, on the contrary to the space homogeneous case, the computational cost of the control variates cannot be ignored since their solution needs to be computed in time. We focus on two different control variates, the compressible Euler equations and the BGK model. 
In the sequel we use the fast spectral method for the solution of the Boltzmann collision operator \cite{MP, DP15}, whereas for the space derivatives we apply a fifth order WENO method for all different models \cite{JS}. The time discretization is performed by a second order explicit Runge-Kutta method in all cases.

More precisely, the cost of the solution of the Boltzmann equation can be estimated by 
\[ 
\C(f)=C N_a^{d_v-1} N_v^{d_v}\log_2(N_v^{d_v})N_x^{d_x}
\]
while the cost of the solution of the BGK equation by
\[ 
\C(\tilde f)= C_1 N_v^{d_v}N_x^{d_x}
\]
and similarly the cost of the solution of the compressible Euler system by 
\[ 
\C(f_F^\infty) = C_2 N_x^{d_x}, 
\]
where $C, \ C1$ and $C_2$ are suitable constants, $N_a$ is the number of angular directions in the velocity space, $N_v$ the number of grid points in velocity space, $N_x$ the number of grid points in physical space and $d_v$ and $d_x$ respectively the dimensions in velocity and space. Now, fixing the total cost of the simulation \[M\C(f)=M_{E_1}\C(\tilde f)=M_{E_2}\C(f^\infty_F),\] where $M$, $M_{E_1}$ and $M_{E_2}$ are the number of samples used respectively in the Boltzmann model, the BGK model and the Euler equations,
we get the number of samples in the control variates for a given cost  
\be M_{E_1}=M\frac{C N_a^{d_v-1} \log_2(N_v^{d_v})}{C_1}\ee
and
\be M_{E_2}=M\frac{ C N_a^{d_v-1} N_v^{d_v}\log_2(N_v^{d_v})}{C_2}.\ee
In the following tests, we choose $M=10$, $N_x=100$, $Nv=32$, $d_x=1$, $N_a=8$ and $d_v=2$. In our case, a rough estimation of the ratio between the coefficients $C$, $C1$ and $C_2$ gives ${C}/{C_1} \approx 1.25$ and ${C}/{C_2} \approx 1$.
This gives approximatively $M_{E_1}=10^3$ and $M_{E_2}=10^6$.

\subsubsection{Test 3. Sod test with uncertain initial data}
The initial conditions are
\bea
&\rho_0(x)=1, \ \ T_0(z,x)=1+sz \qquad &\textnormal{if} \ \ 0<x<L/2 \\
&\rho_0(x)=0.125,\ T_0(z,x)=0.8+sz  \qquad &\textnormal{if} \ \ L/2<x<1
\eea
with $s=0.25$, $z$ uniform in $[0,1]$ and equilibrium initial distribution
\[
f_0(z,x,\w)=\frac{\rho_0(x)}{2\pi } \exp\left({-\frac{|v|^2}{2T_0(z,x)}}\right).
\] 
The velocity space is truncated with $v_{\min}=v_{\max}=8$. 
The time step is the same for all methods and is taken as $\Delta t=\min\{\Delta x/(2 v_{\max}), \varepsilon\}$ with $\varepsilon$ the Knudsen number. Note that, here we are interested only in the accuracy in the random variable, therefore we selected a simple explicit method satisfying the above CFL condition. We refer to \cite{DP15} for other choices which avoids the limitation with respect to $\varepsilon$.
We perform three different computations corresponding to $ \varepsilon=10^{-2}$, $ \varepsilon=10^{-3}$ and $ \varepsilon=2 \times 10^{-4}$. The final time is fixed to $T_f=0.875$. 
In Figure \ref{Figure8}, we report the expectation of the solution at the final time together with the confidence bands $\EE[T]-\sigma_T, \EE[T]+\sigma_T$ with $\sigma_T$ the standard deviation. These reference values have been computed by an orthogonal polynomial collocation method.  \begin{figure}[ht!]
	\begin{center}
		\includegraphics[width=0.45\textwidth]{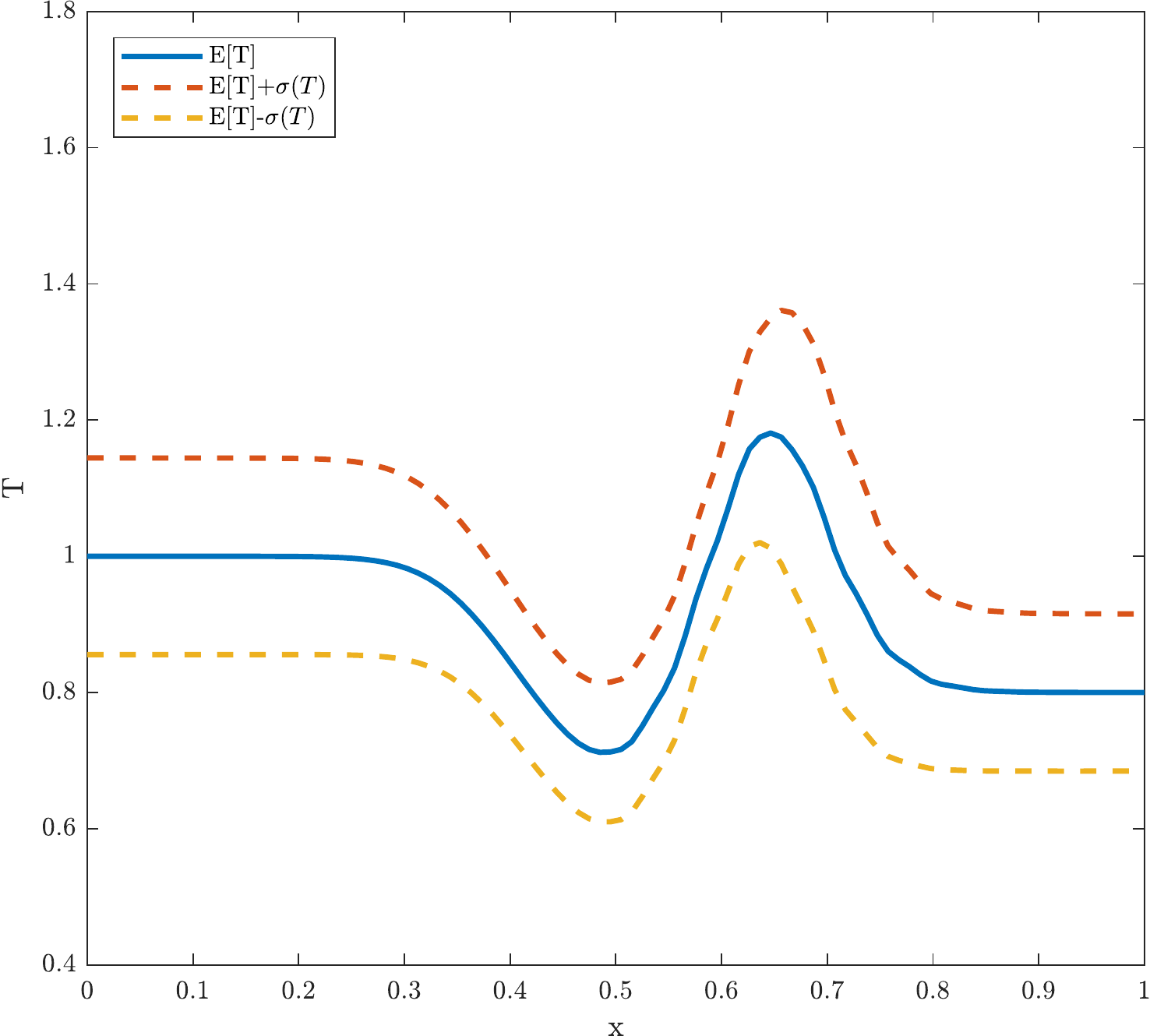}\hspace{1cm}
		\includegraphics[width=0.45\textwidth]{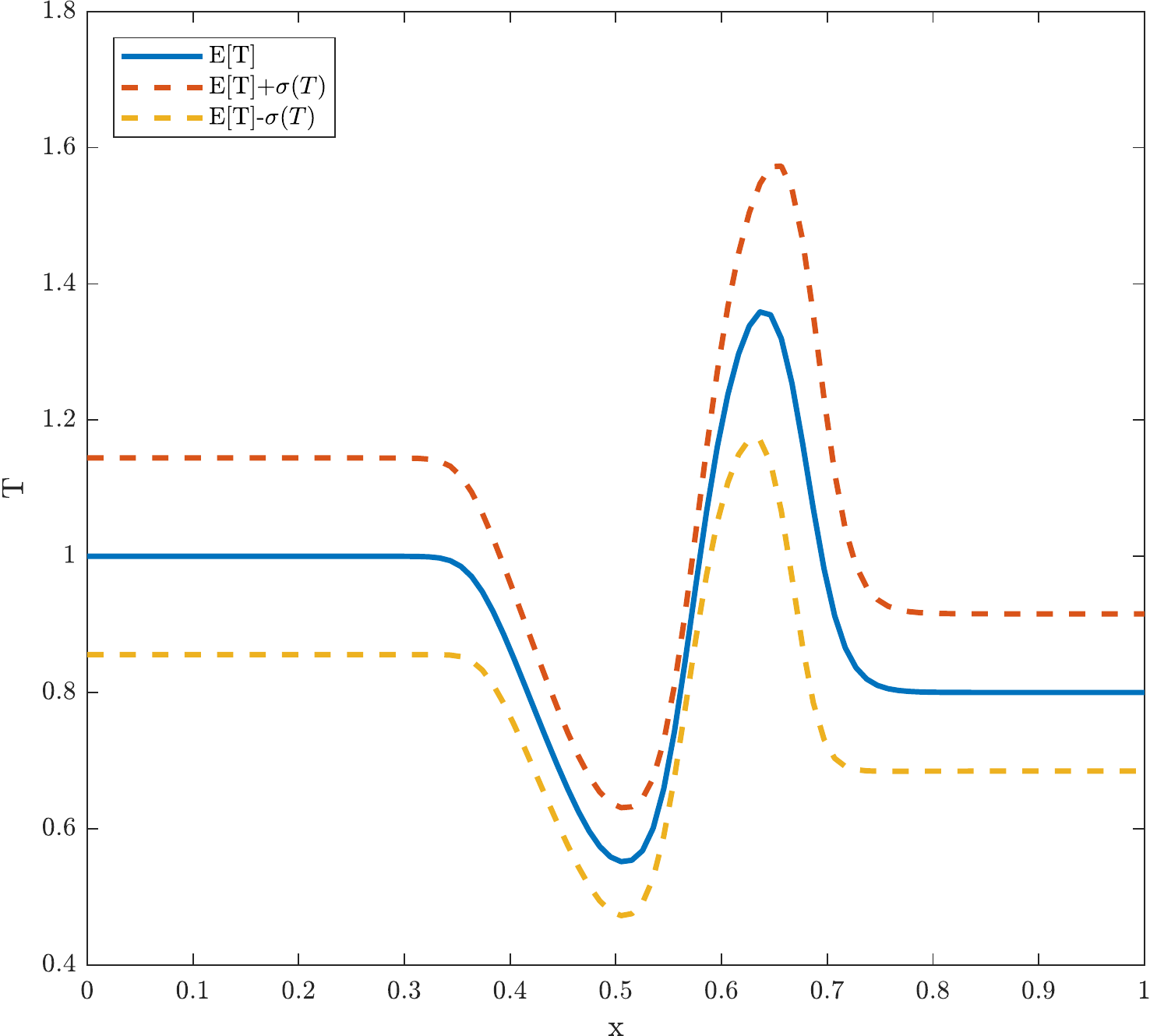}\\
		\includegraphics[width=0.45\textwidth]{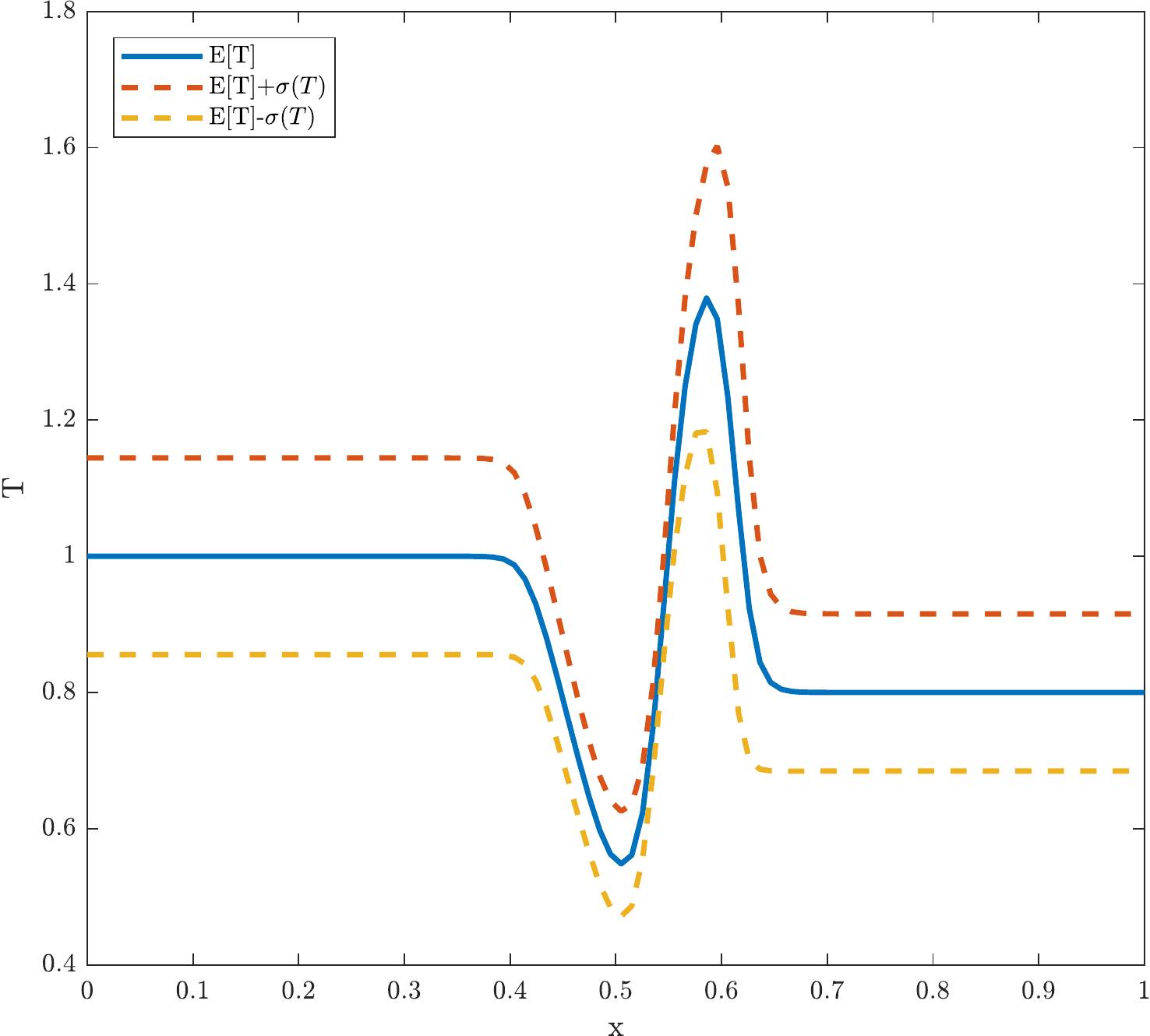}
		\captionof{figure}{Test 3. Sod test with uncertainty in the initial data. Temperature profile at final time : expectation and confidence bands. Top left: $\varepsilon=10^{-2}$. Top right: $\varepsilon=10^{-3}$. Bottom: $\varepsilon=2 \times \ 10^{-4}$.}
		\label{Figure8}
	\end{center}
\end{figure}

In Figure \ref{Figure9}, we report the various errors for the expected value of the temperature as a function of time. 
The number of samples used to compute the expected value of the solution is $M=10$ while the number of samples used to compute the control variate is $M_E=10^3$ for the left images and $M_E=10^4$ for the right ones. The optimal values of $\lambda^*(x,t)$ have been computed with respect to the temperature as described in Remark \ref{rk:33}.
We see that, when $M_E=10^3$ and for moderate Knudsen numbers the MSCV method with the BGK control variate gives the best result while when the Knudsen number diminishes, as expected, the two control variate strategies gives analogous results for $M_E=10^3$. For the case $M_E=10^4$, the MSCV method with the BGK control variate is again more accurate, however its cost is slightly higher than the cost of the original MC method applied to the full Boltzmann model with $M=10$.
\begin{figure}[ht!]
	\begin{center}
		\includegraphics[width=.4\textwidth]{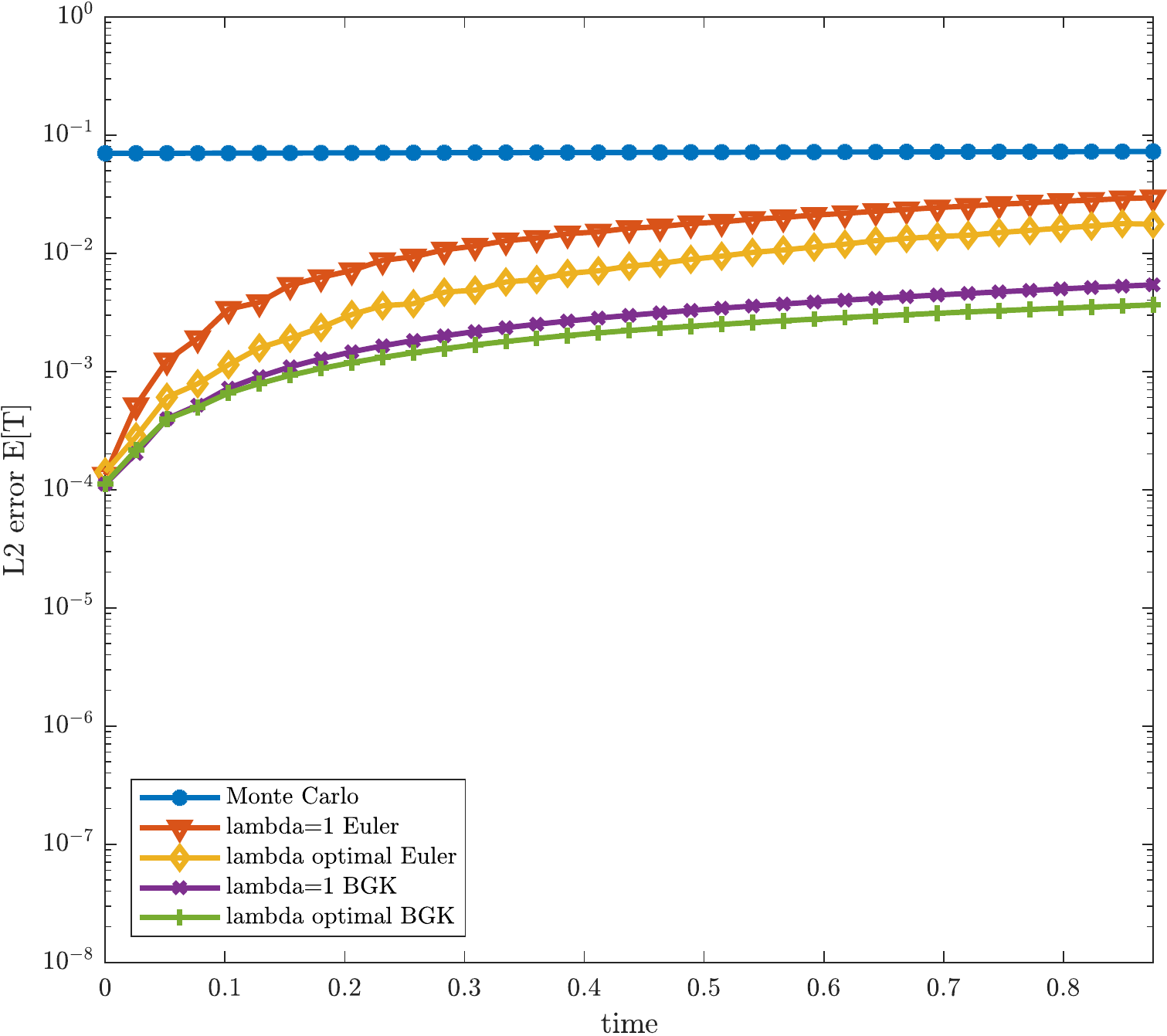}\hspace{1cm}
			\includegraphics[width=.4\textwidth]{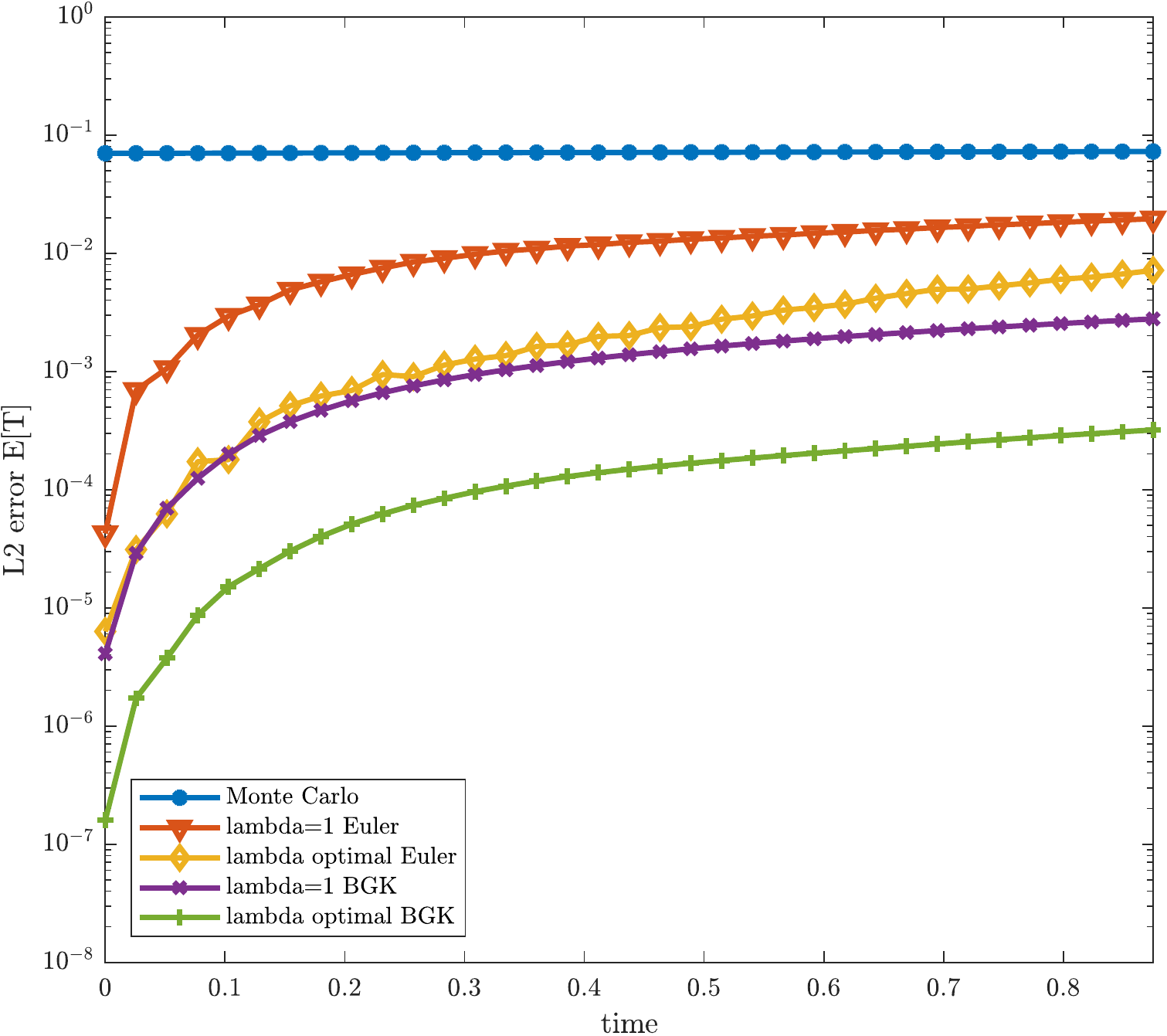}\\
		\includegraphics[width=.4\textwidth]{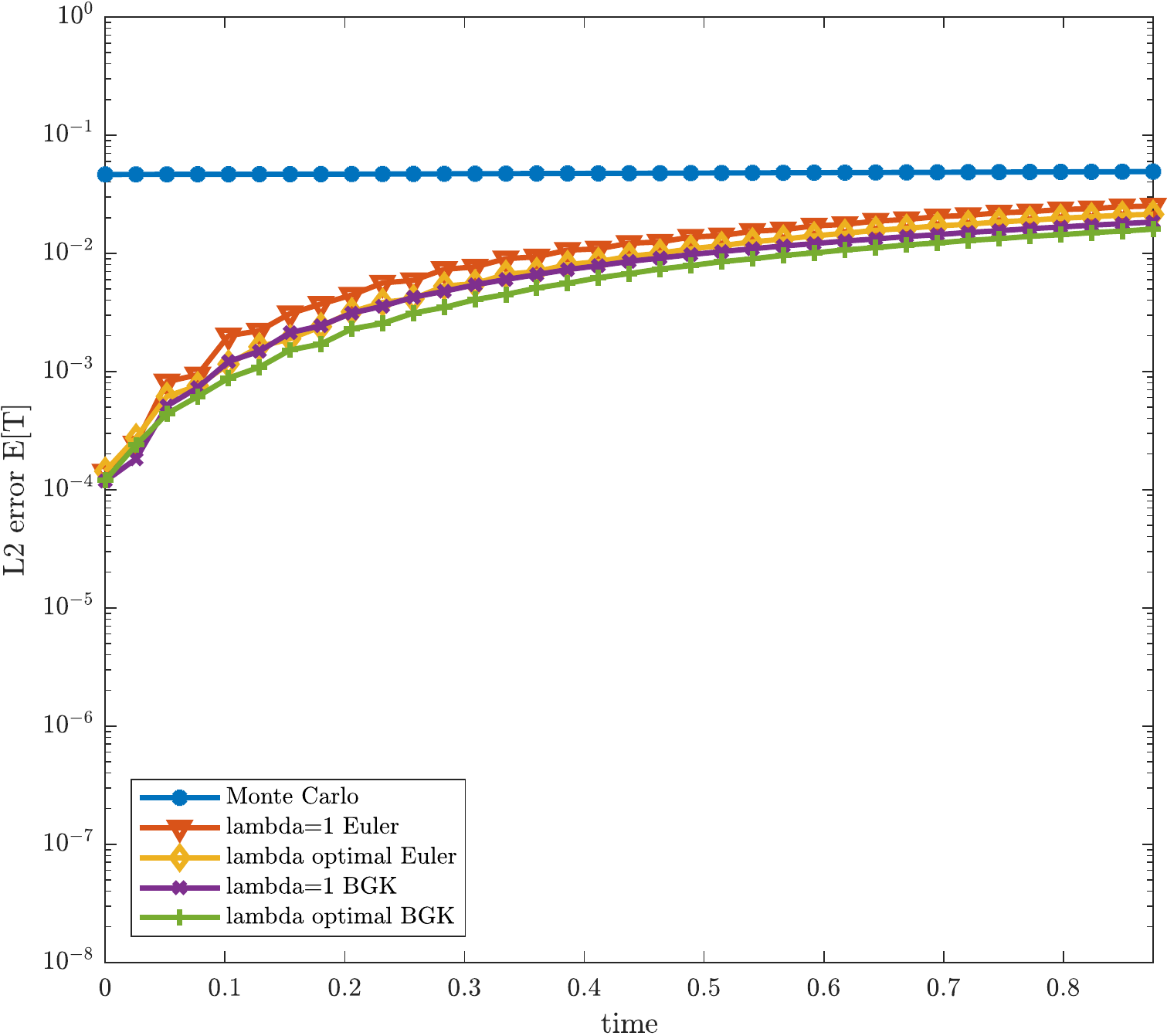}\hspace{1cm}
			\includegraphics[width=.4\textwidth]{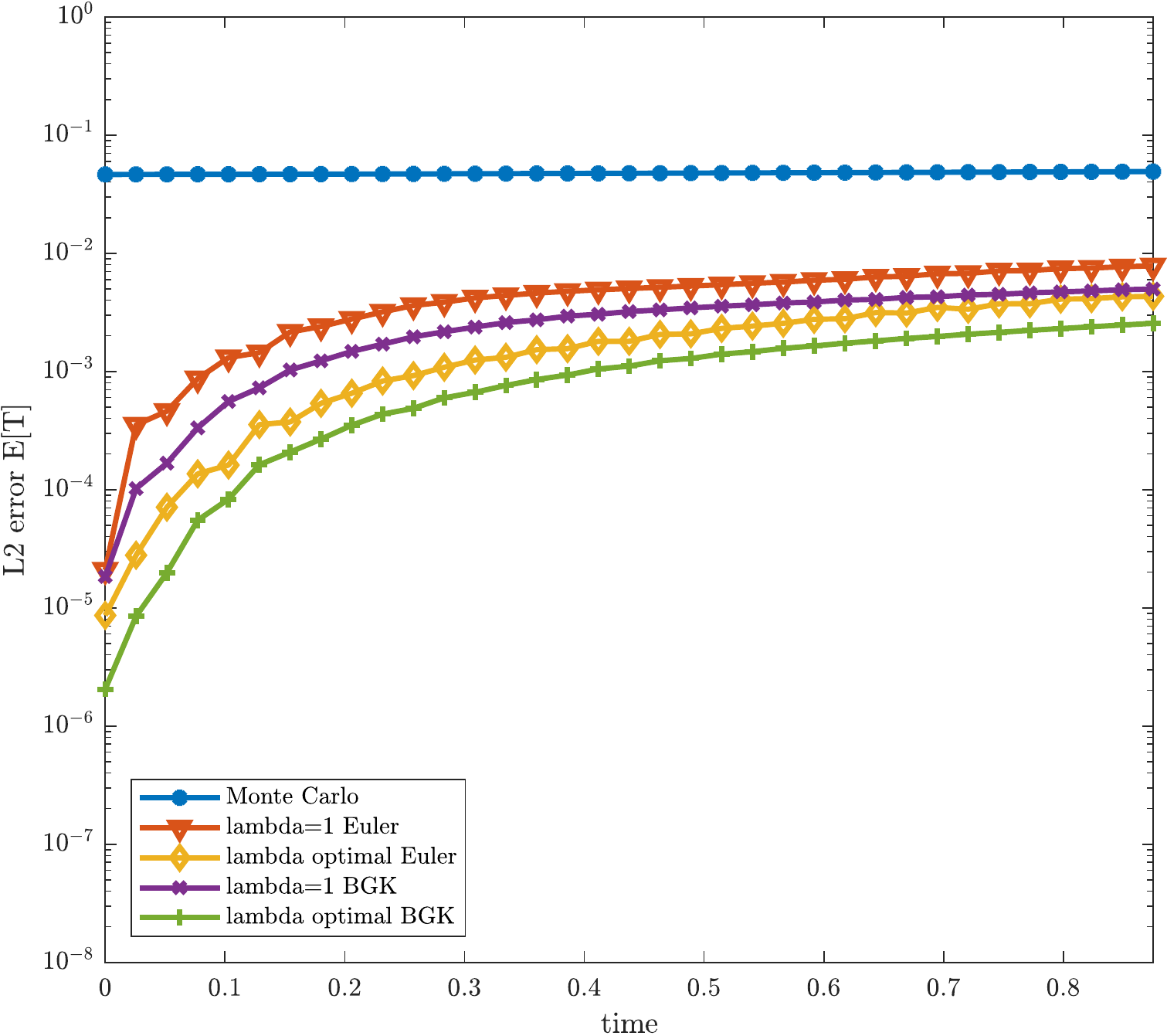}\\
		\includegraphics[width=.4\textwidth]{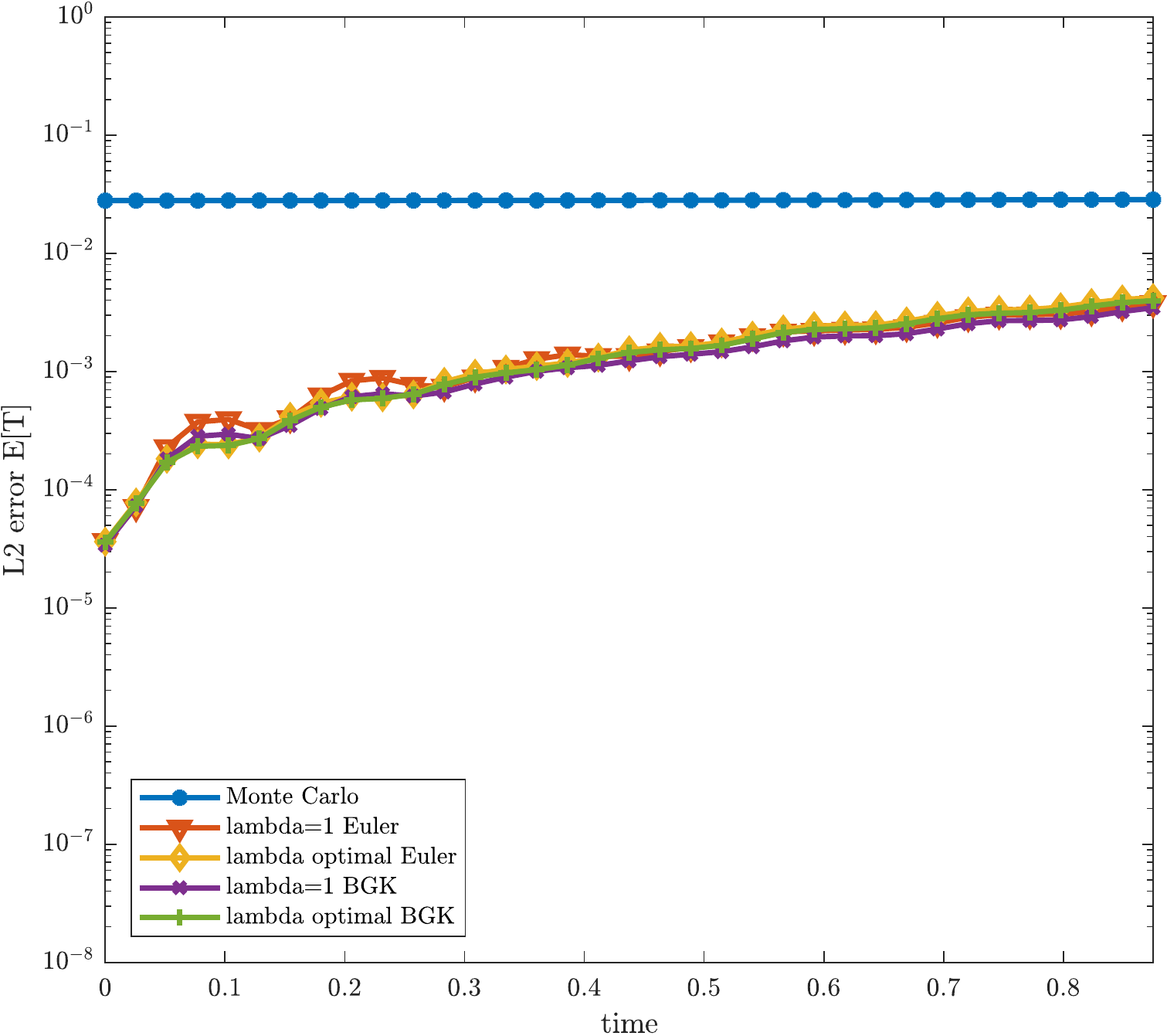}\hspace{1cm}
			\includegraphics[width=.4\textwidth]{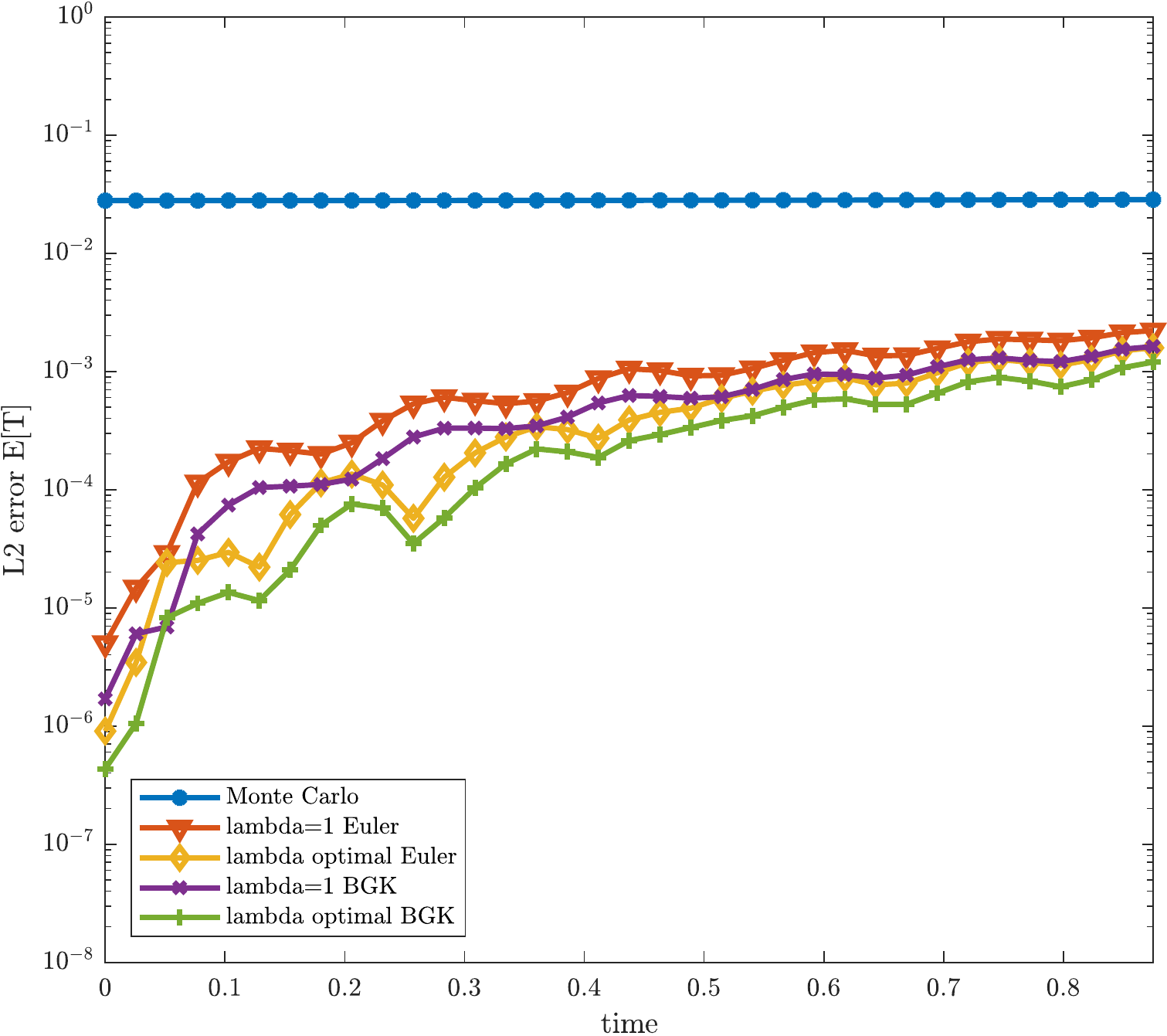}
		\captionof{figure}{Test 3. Sod test with uncertainty in the initial data for the Monte Carlo method and the MSCV method for various control variates strategies.  $L_2$ norm of the error for the expectation of the temperature with $M=10$ samples. Top left: $\varepsilon=10^{-2}$. Top right: $\varepsilon=10^{-3}$. Bottom: $\varepsilon=2 \times \ 10^{-4}$. Left panels: $M_E=10^3$. Right panels: $M_E=10^4$.}\label{Figure9}
	\end{center}
\end{figure}
In Figure \ref{Figure11}, we finally report the shape of the optimization coefficient $\lambda^*(x,t)$ in time for the BGK control variate for the different Knudsen numbers. 
\begin{figure}[ht!]
	\begin{center}
		\includegraphics[width=.46\textwidth]{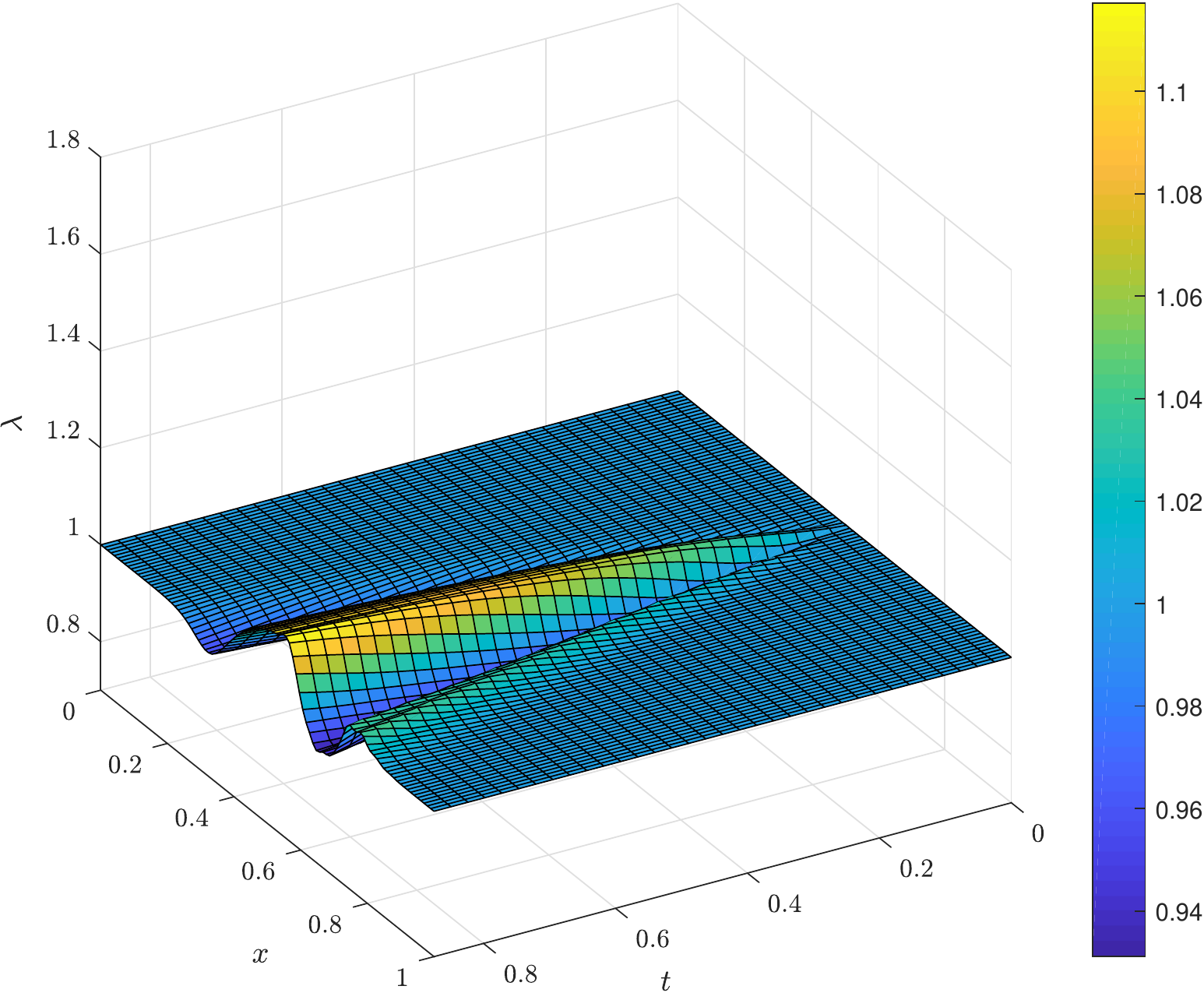}\hspace{1cm}
		\includegraphics[width=.46\textwidth]{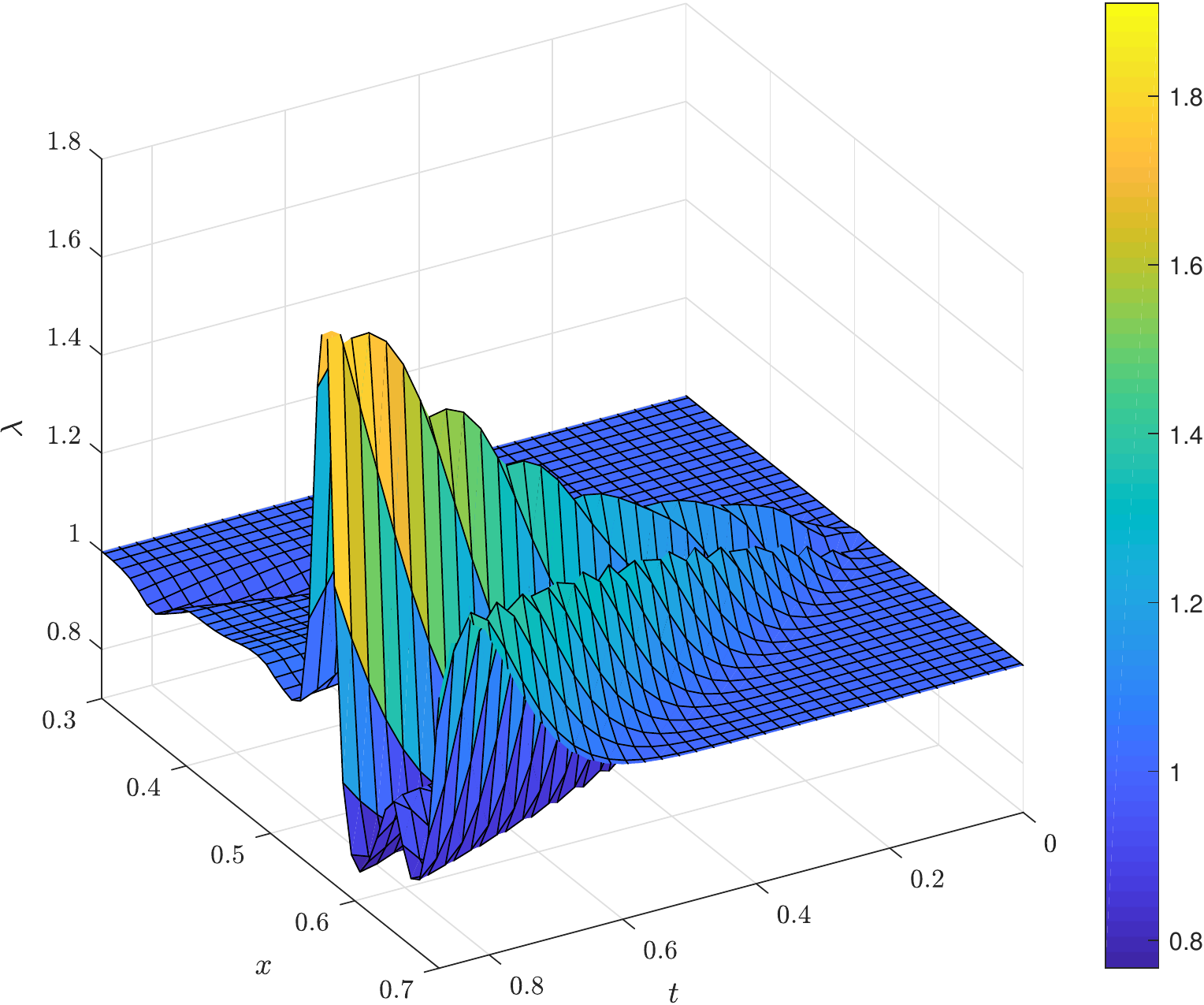}\\
		\includegraphics[width=.46\textwidth]{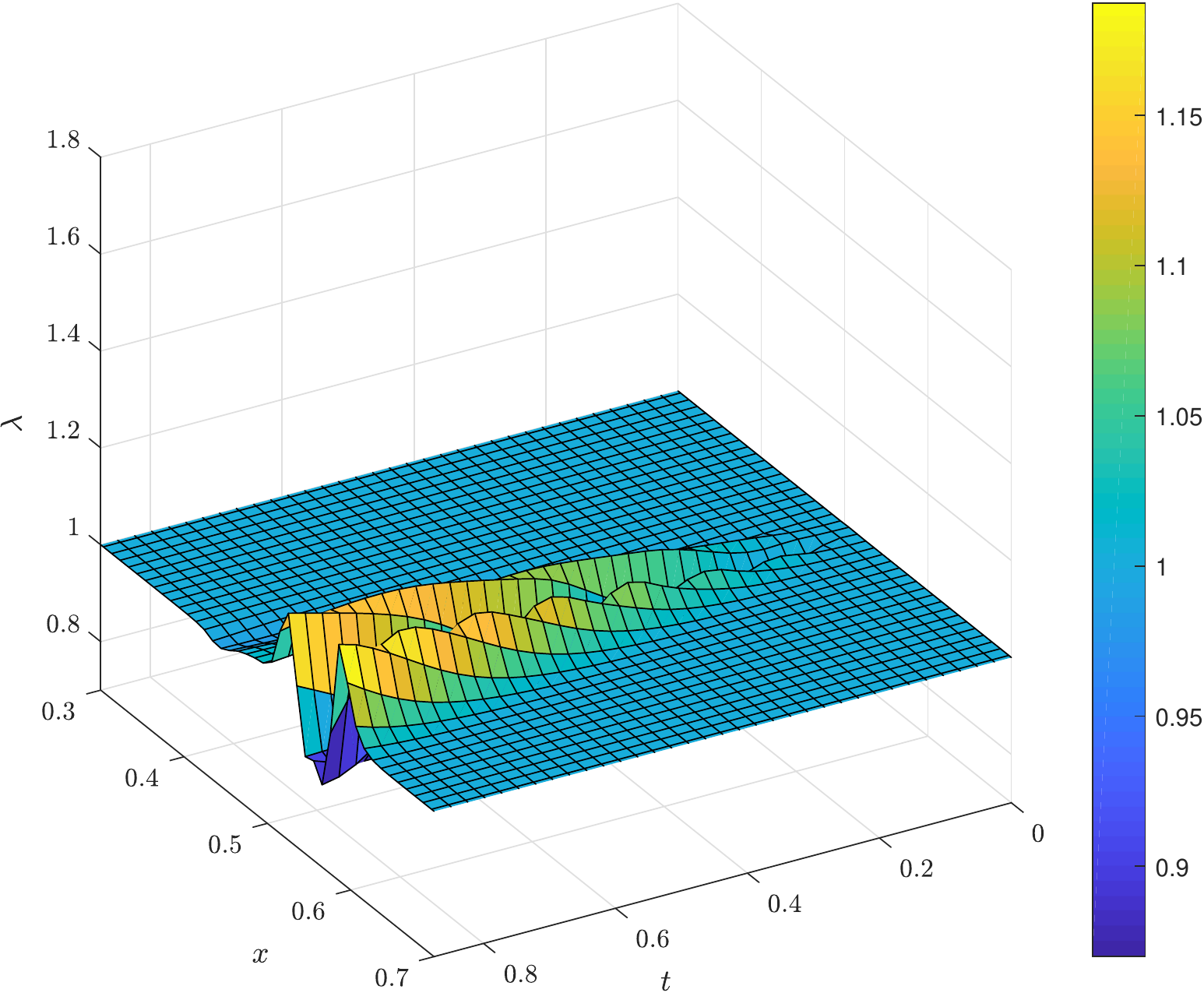}
		\captionof{figure}{Test 3. Sod test with uncertainty in the initial data. Optimal $\lambda^*(x,t)$ for MSCV method based on the BGK control variate strategy. Top left: $\varepsilon=10^{-2}$. Top right: $\varepsilon=10^{-3}$. Bottom: $\varepsilon=2 \times \ 10^{-4}$ at final time $T_f=0.0875$.}
		\label{Figure11}
	\end{center}
\end{figure}

\subsubsection{Test 4. Sod test with uncertain collision kernel}
The initial conditions are the deterministic Sod test data
\bea
&\rho_0(x)=1, \ \ T_0(x)=1 \qquad &\textnormal{if} \ \ 0<x<L/2 \\
&\rho_0(x)=0.125,\ T_0(x)=0.8  \qquad &\textnormal{if} \ \ L/2<x<1
\eea
with the deterministic equilibrium initial condition
\[
f_0(x,\w)=\frac{\rho_0(x)}{2\pi } \exp\left({-\frac{|v|^2}{2T_0(x)}}\right).
\] 
Now, the collision kernel contains the random variable
\[
B(z)=1+sz
\]
with $s=0.99$ and $z$ uniform in $(0,1)$. 

As before, the velocity space is truncated with $v_{\min}=v_{\max}=8$ and the time step is the same for all methods and is taken as $\Delta t=\min\{\Delta x/(2 v_{\max}), \varepsilon\}$ with $\varepsilon$ the Knudsen number fixed to $\varepsilon=5 \times 10^{-4}$ .
The final time is $0.875$. In Figure \ref{Figure12}, we report the expectation of the solution at the final time together with the confidence band $\EE[T]-\sigma_T, \EE[T]+\sigma_T$ with $\sigma_T$ the standard deviation. This solution has been computed by using orthogonal polynomials. We also plot the error for the expected temperature and the final expected solution given by the MC and by the MSCV method with the BGK control variate. Note, in fact, that the compressible Euler control variate coincides with the standard MC method since the equilibrium state is deterministic.
The number of samples used to compute the expected solution is $M=10$ while the number of samples used to compute the control variate is $M_E=10^3$. 
In the same Figure, we show two magnifications of the solution around the point $x=0.5$ and $x=0.7$, i.e. the maximum and the minimum values of the temperature. On the bottom right, we show the optimal parameter $\lambda^*(x,t)$ at different times. 

\begin{figure}[ht!]
	\begin{center}
		\includegraphics[width=0.45\textwidth]{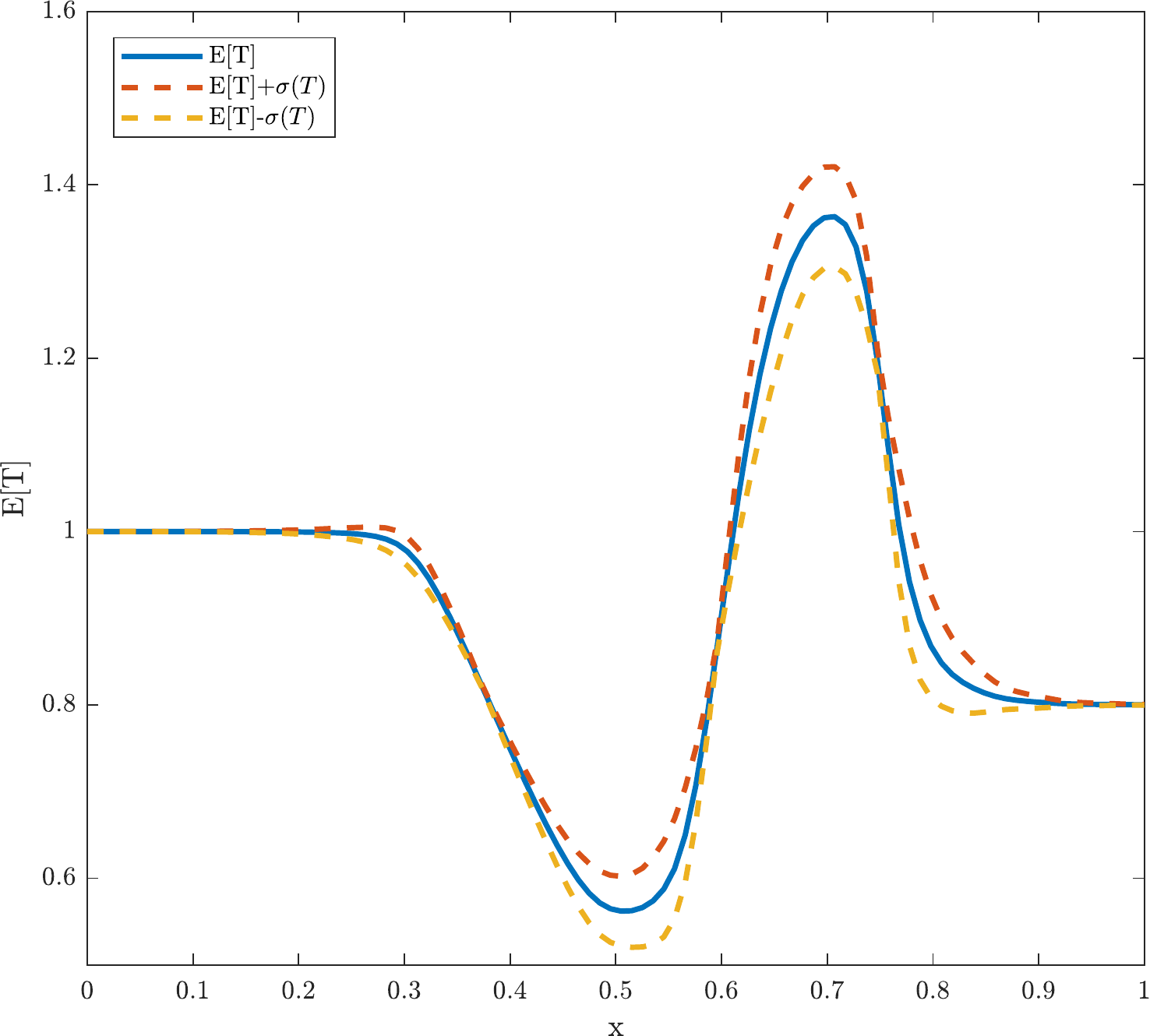}\hspace{1cm}
	\includegraphics[width=.45\textwidth]{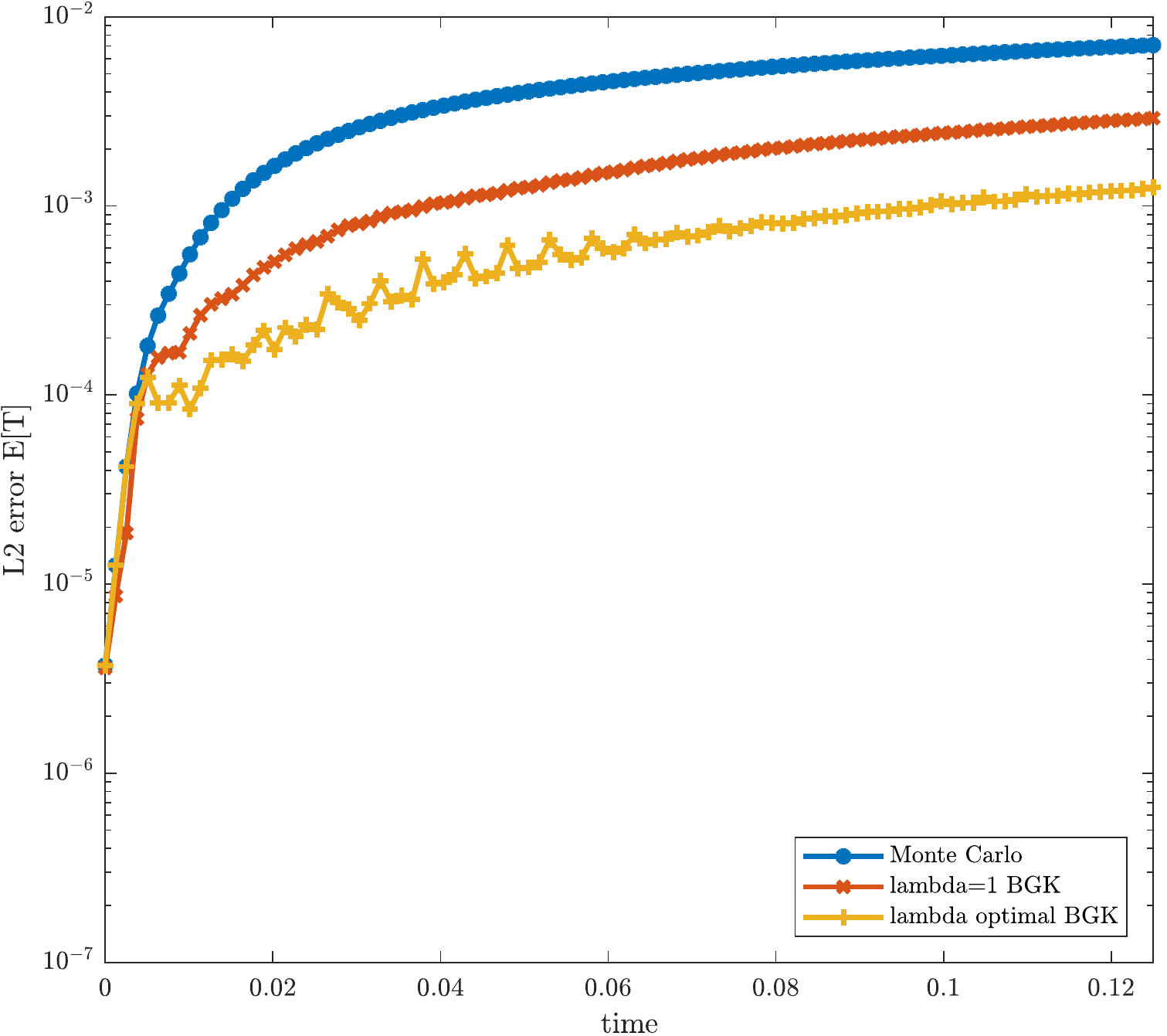}\\
		\includegraphics[width=.45\textwidth]{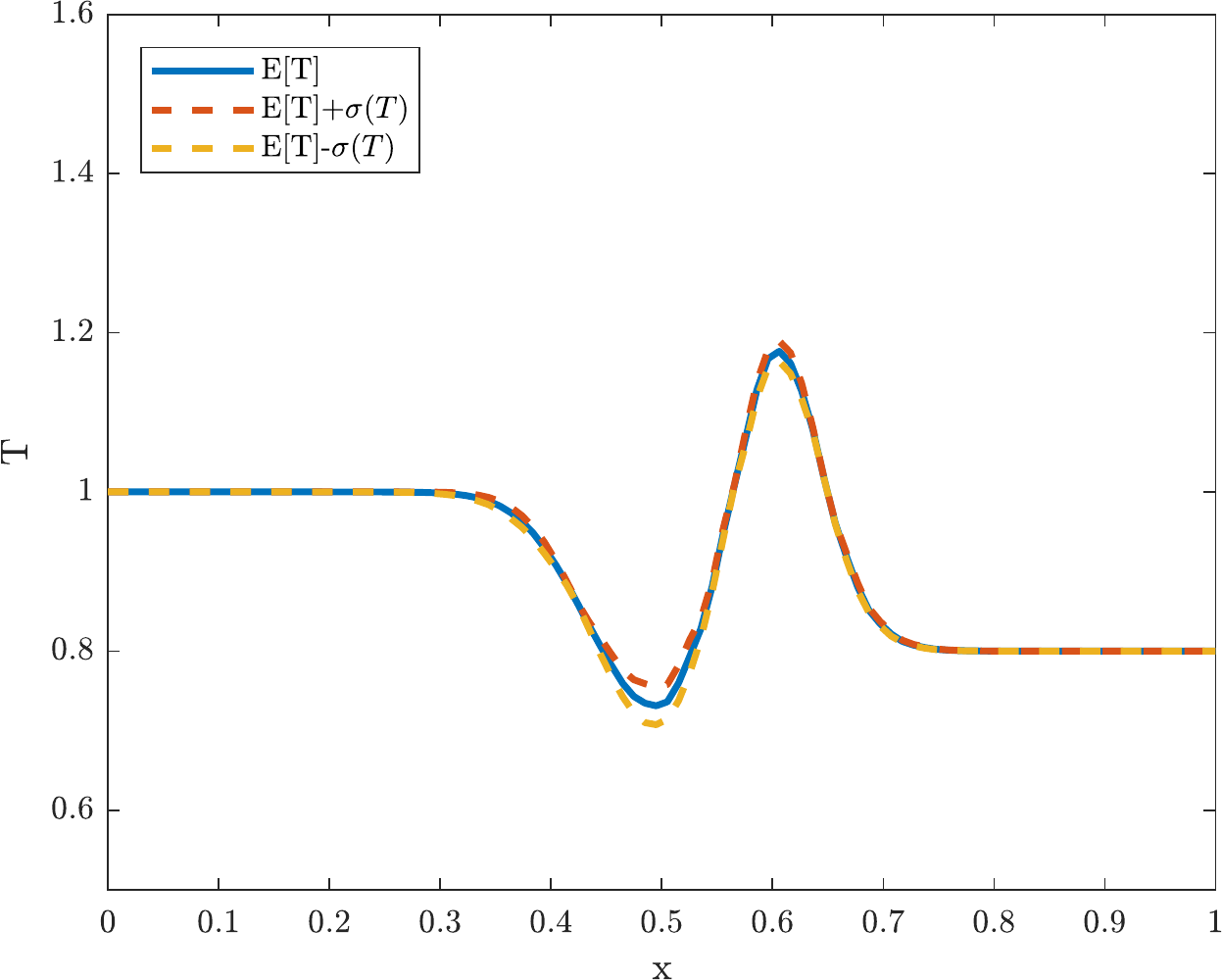}\hspace{1cm}
		\includegraphics[width=.45\textwidth]{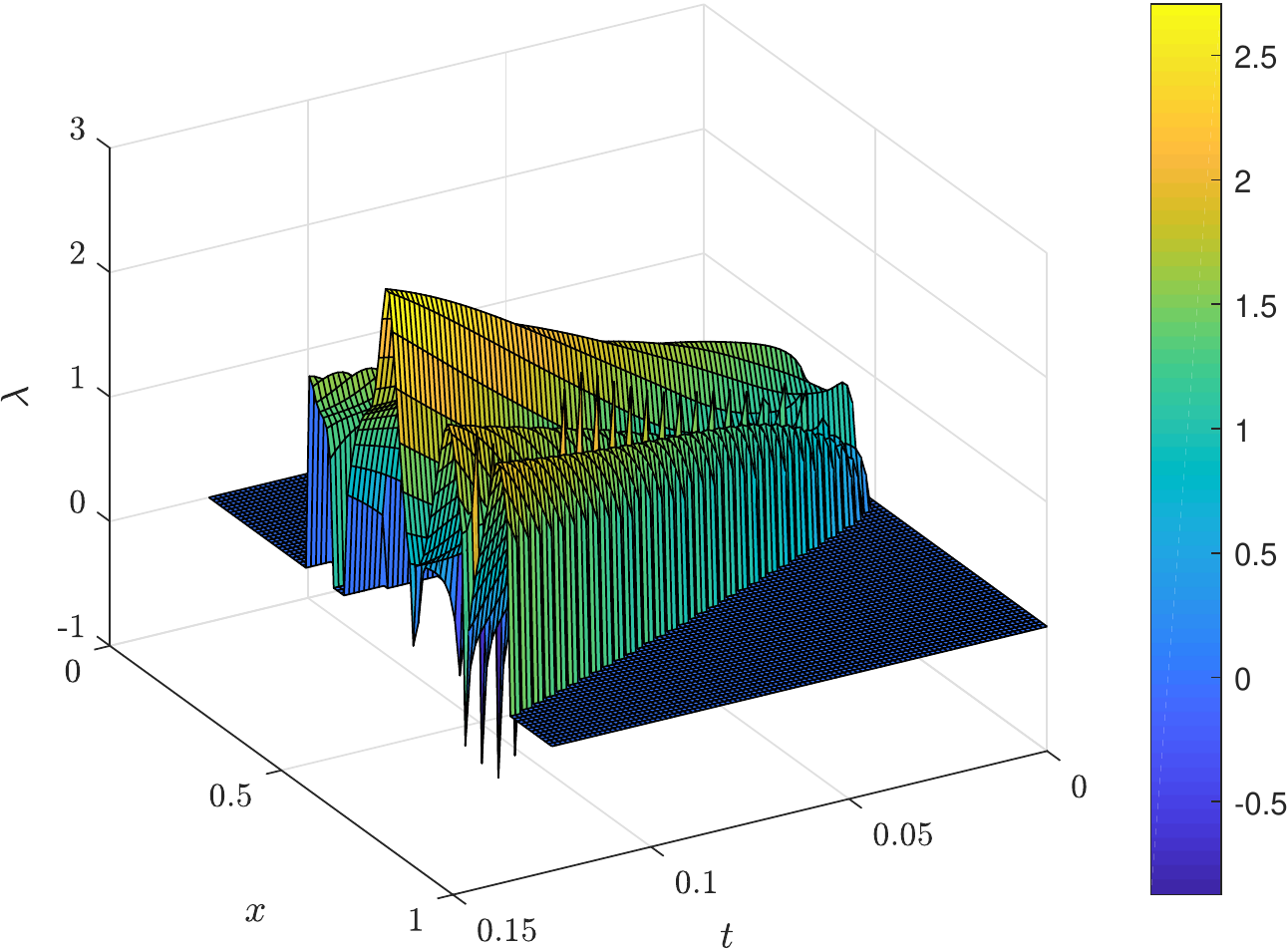}
		\captionof{figure}{Test 4. Sod test with uncertainty in the collision kernel. Top left: Temperature profile at final time. Top right: $L_2$ norm of the error for the expectation of the temperature with $M=10$ samples for the MC and the MSCV method. Bottom left: Temperature profile at final time obtained with the MC and the MSCV method. Bottom right: optimal $\lambda^*(x,t)$ in time and space.}
		\label{Figure12}
	\end{center}
\end{figure}

\subsubsection{Test 5. Sudden heating problem with uncertain boundary condition}
In the last test case, the initial condition is a constant state in space given by
\be
f_0(x,\w)=\frac{1}{2\pi T^0}e^{-\dfrac{\w^2}{2T^0}}, \ T^0=1, \qquad x\in[0,1].
\ee 
At time $t=0$, the temperature at the left wall suddenly changes and it starts to heat the gas. We assume diffusive equilibrium boundary conditions and uncertainty on the wall temperature:
\be
T_w(z)=2(T^0+sz), \ s=0.2,
\ee
and $z$ uniform in $[0,1]$.
The truncation of the velocity space as well the other numerical parameters are the same as in Test 3. The final time is $T_f=0.9$. In Figure \ref{Figure16}, we report the expectation of the solution at the final time together with the confidence band $\EE[T]-\sigma_T, \EE[T]+\sigma_T$ with $\sigma_T$ the standard deviation. This solution has been computed by using orthogonal polynomials. 
\begin{figure}[ht!]
	\begin{center}
		\includegraphics[width=0.45\textwidth]{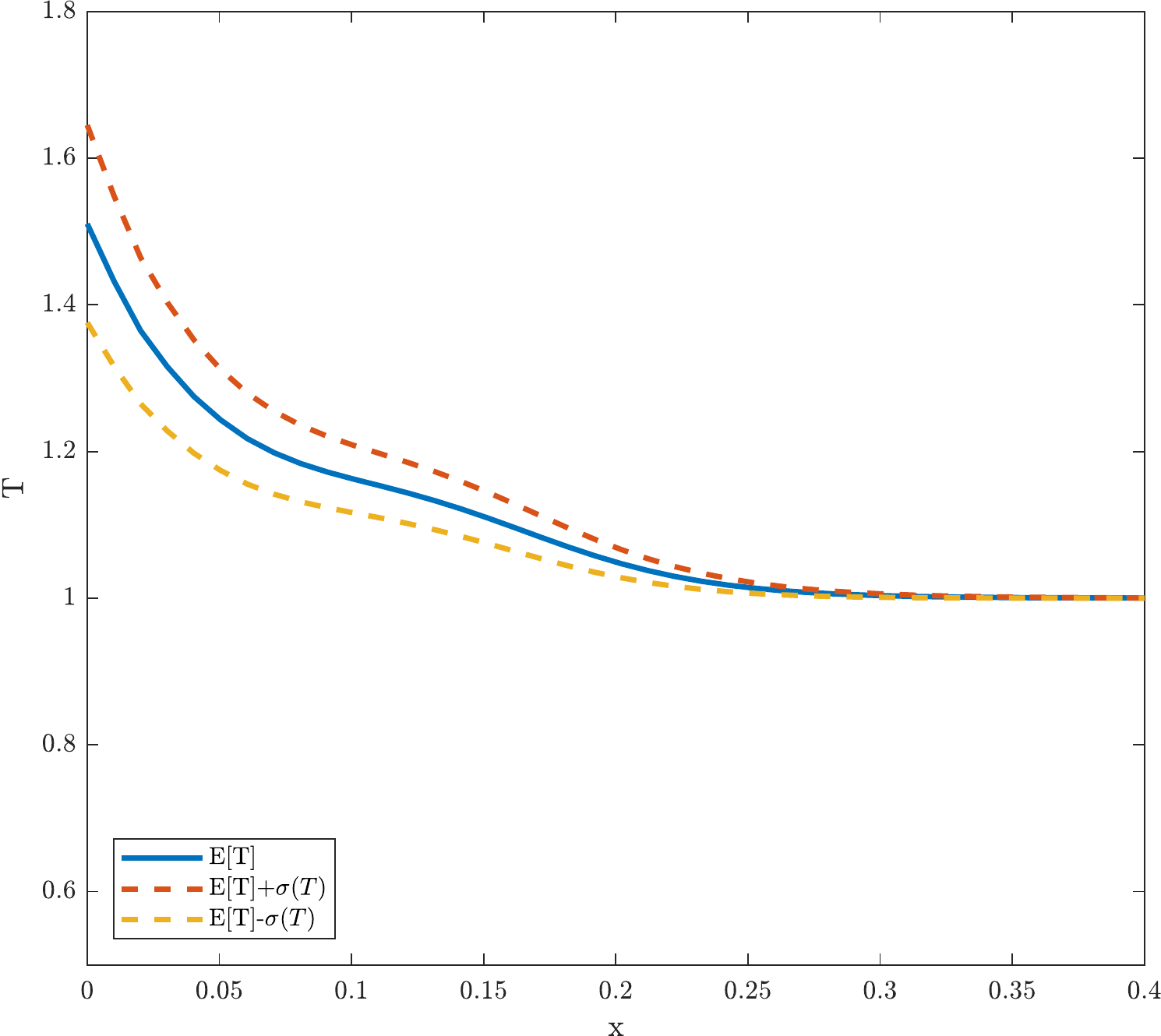}\hspace{1cm}
		\includegraphics[width=0.45\textwidth]{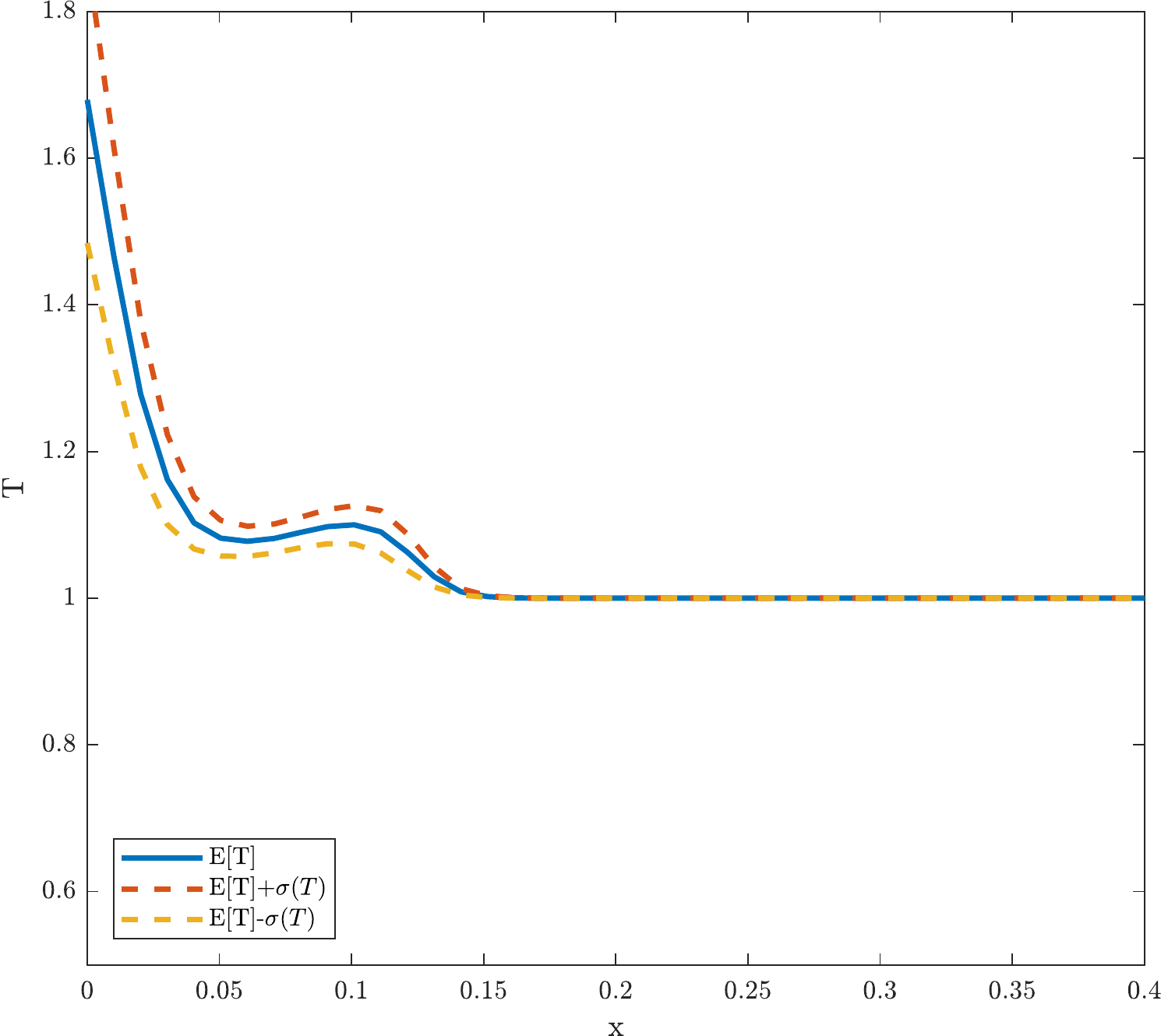}\\
		\includegraphics[width=0.45\textwidth]{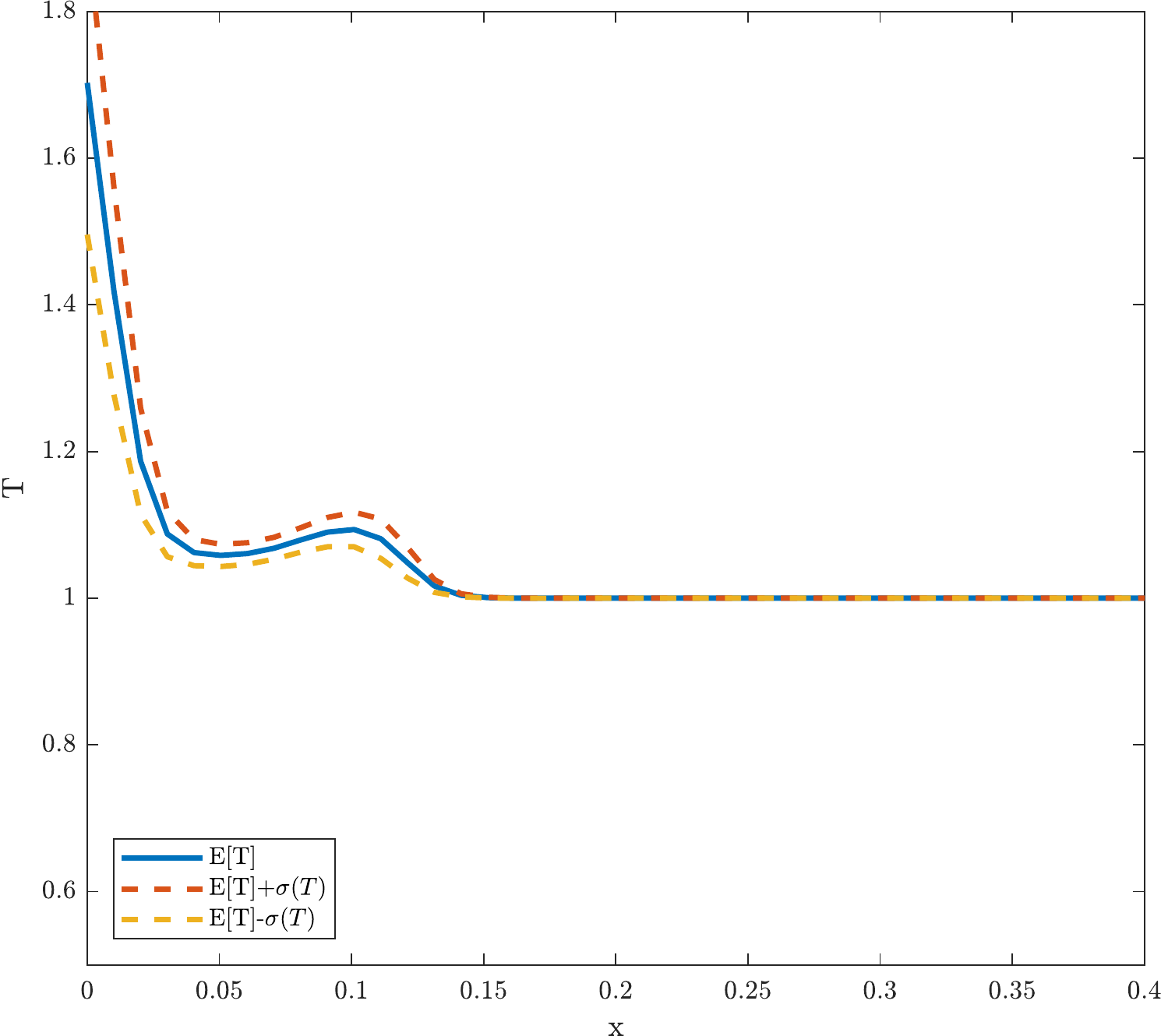}
		\captionof{figure}{Test 5. Sudden heating problem with uncertainty in the boundary condition. Temperature profile at final time. Top left: $\varepsilon=10^{-2}$. Top right: $\varepsilon=10^{-3}$. Bottom: $\varepsilon=2 \ 10^{-4}$. Magnification around the left boundary.}
		\label{Figure16}
	\end{center}
\end{figure}

In Figure \ref{Figure17}, we report the error for the expected value of the temperature as a function of time. The number of samples used is the same for both control variates, namely $M_E=10^3$ (on the left) and $M_E=10^4$ (on the right). 
\begin{figure}[ht!]
		\begin{center}
		\includegraphics[width=.4\textwidth]{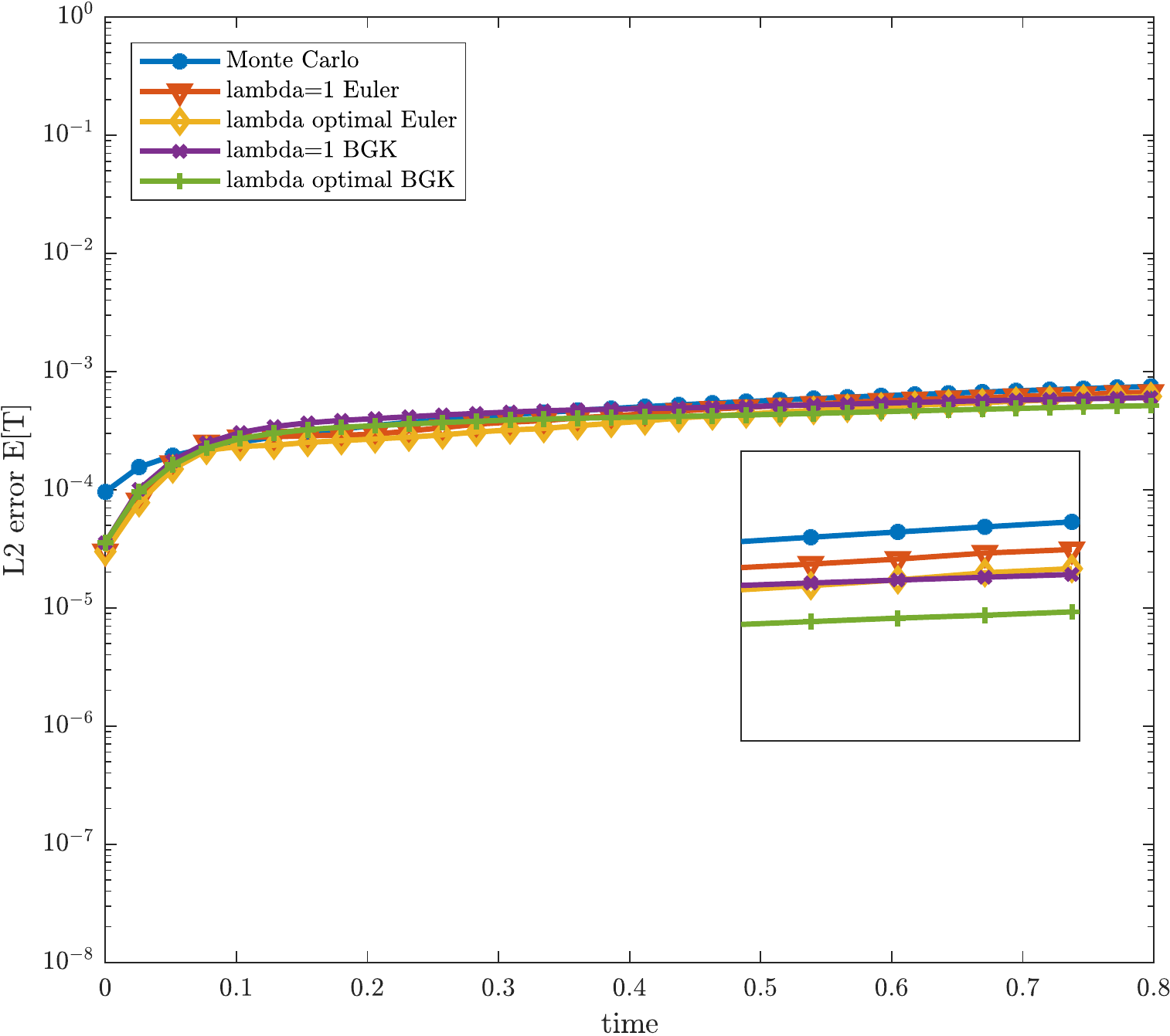}\hspace{1cm}
		\includegraphics[width=.4\textwidth]{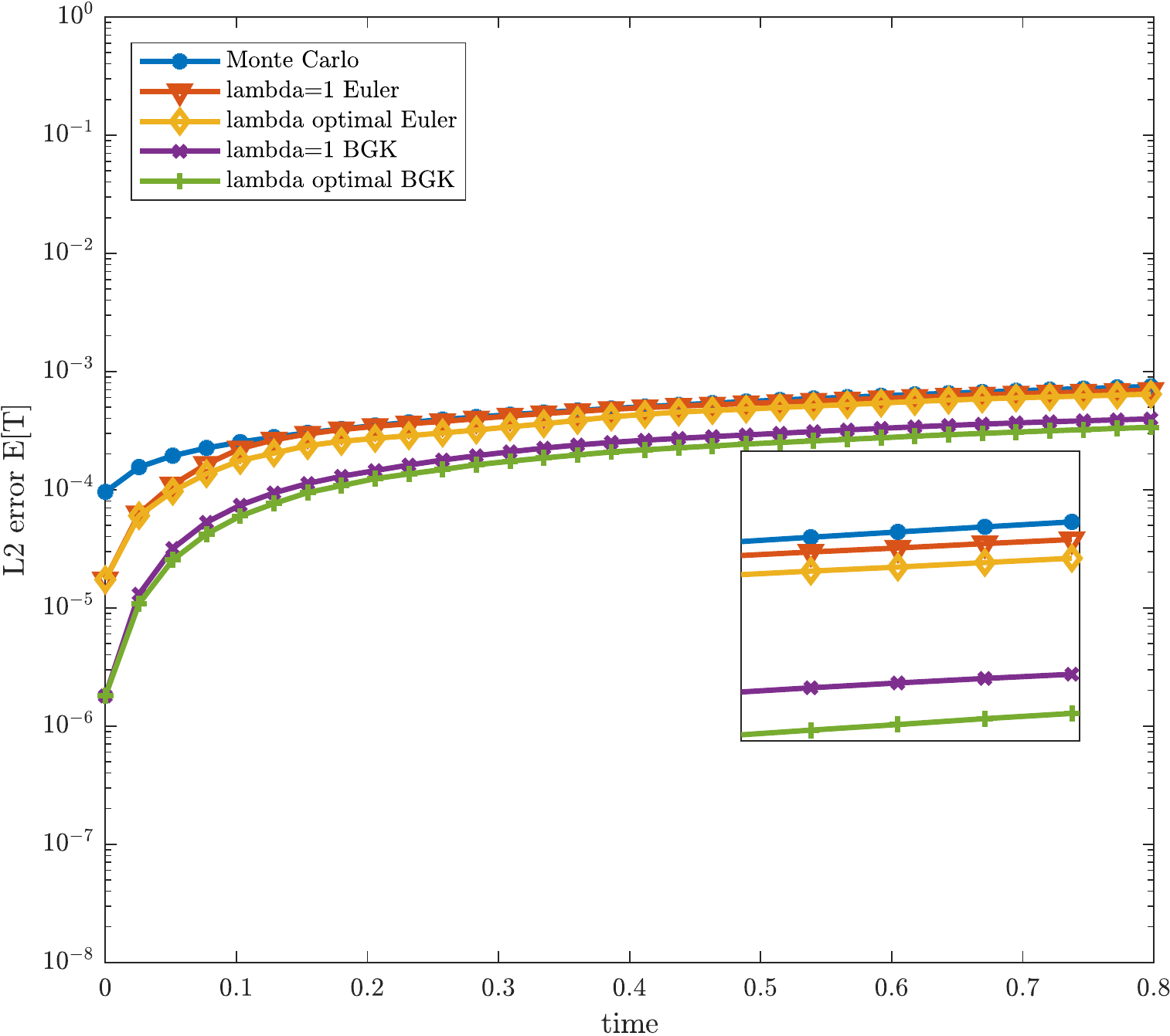}\\
		\includegraphics[width=.4\textwidth]{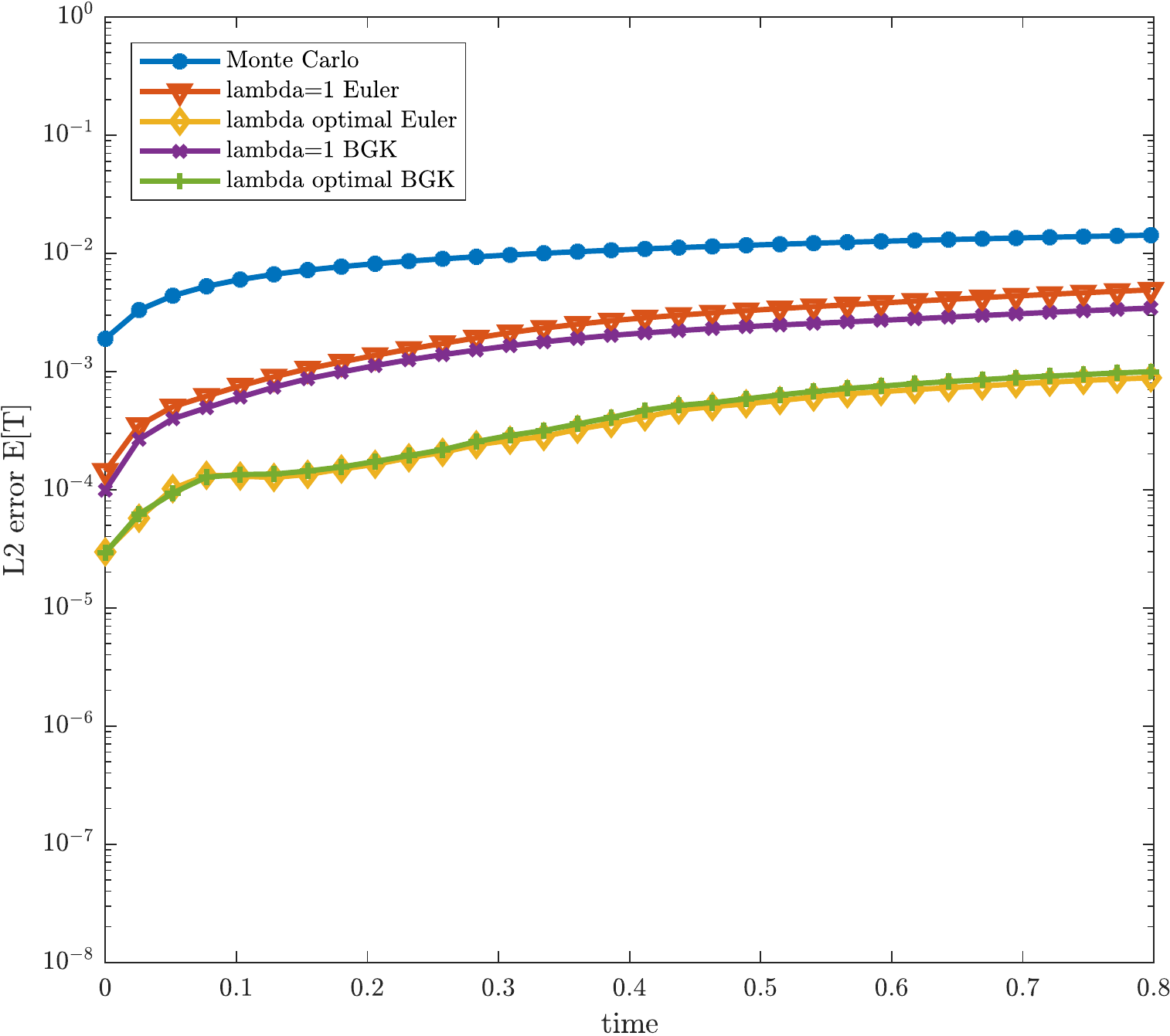}\hspace{1cm}
		\includegraphics[width=.4\textwidth]{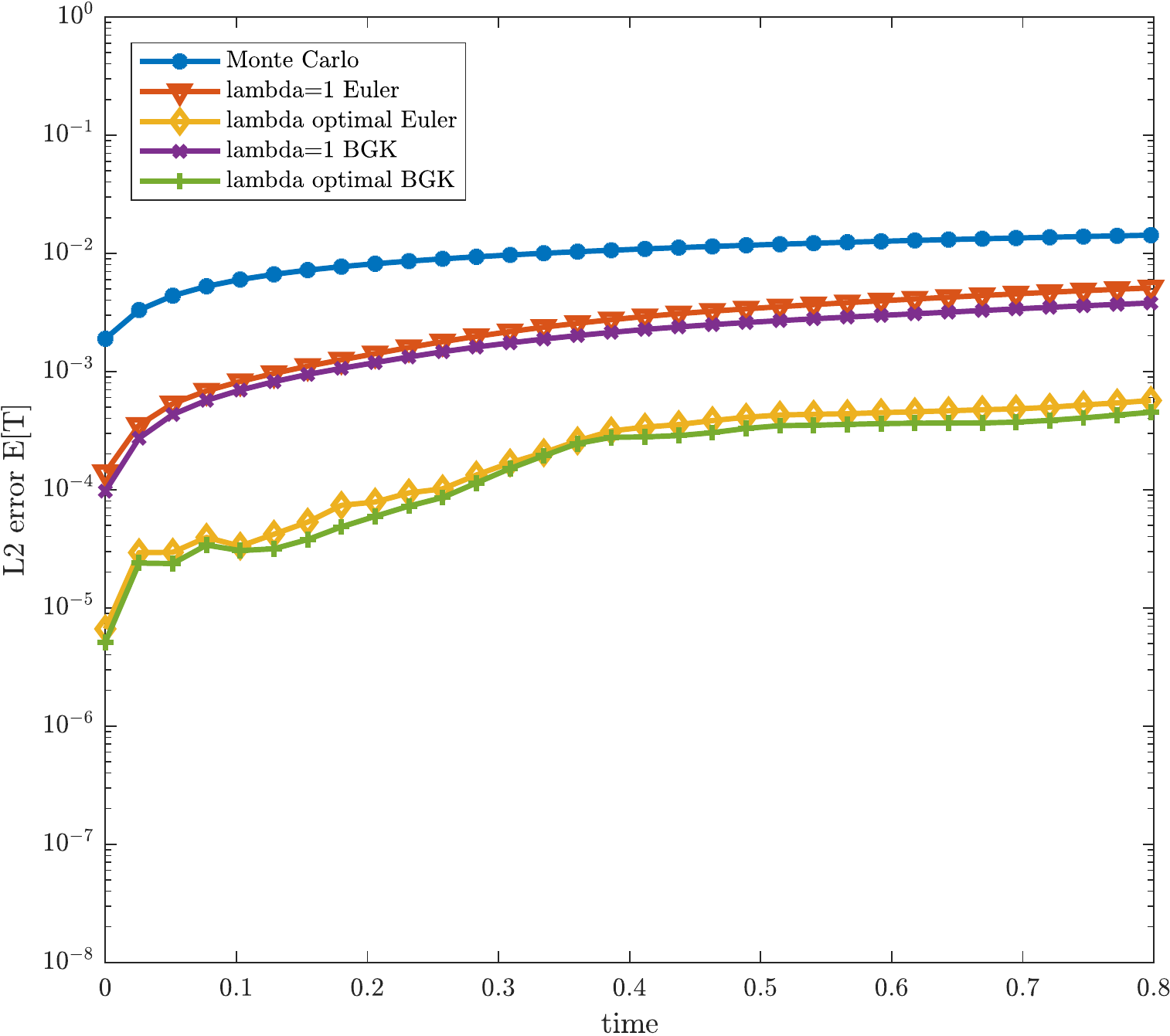}\\
		\includegraphics[width=.4\textwidth]{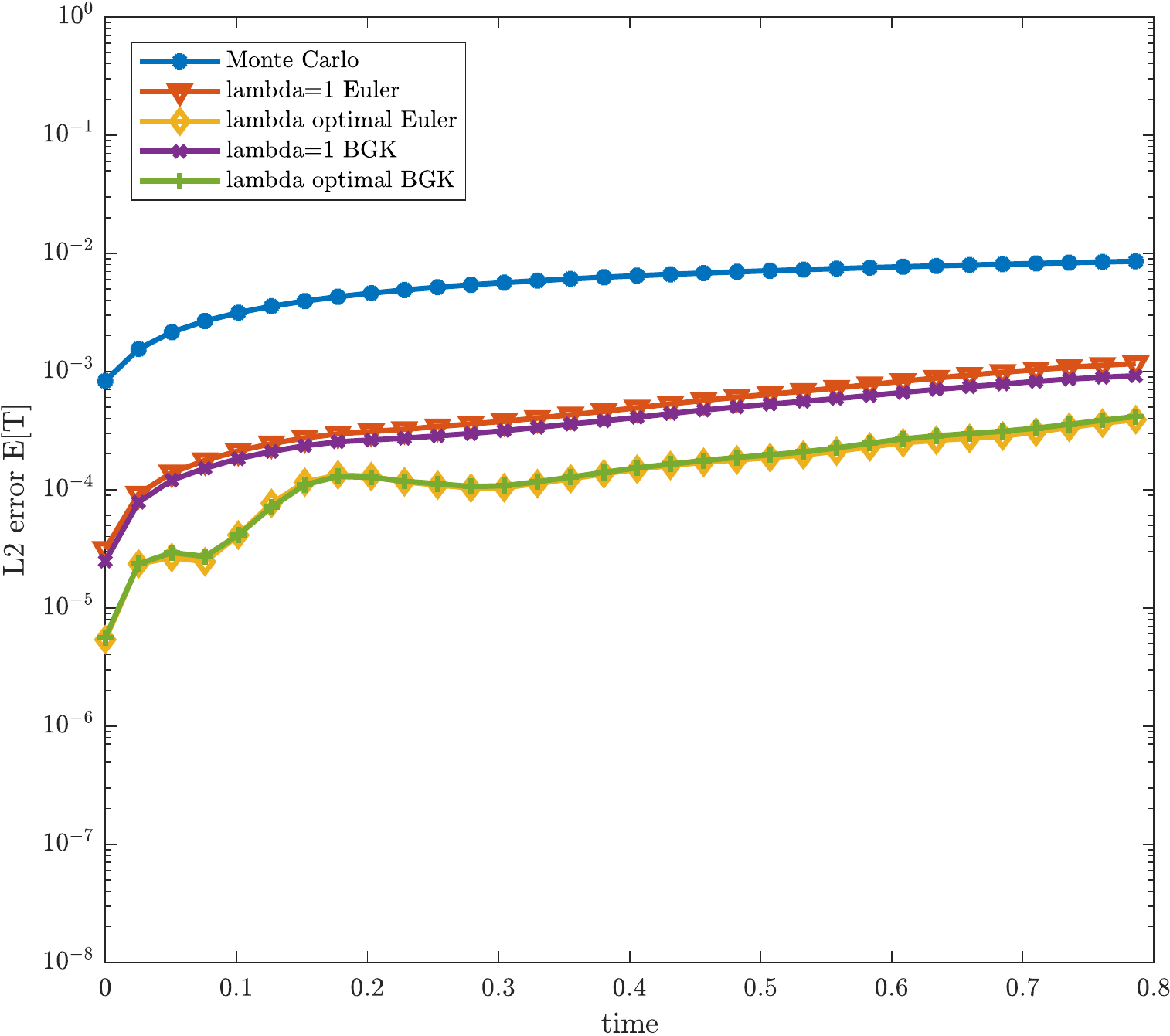}\hspace{1cm}
		\includegraphics[width=.4\textwidth]{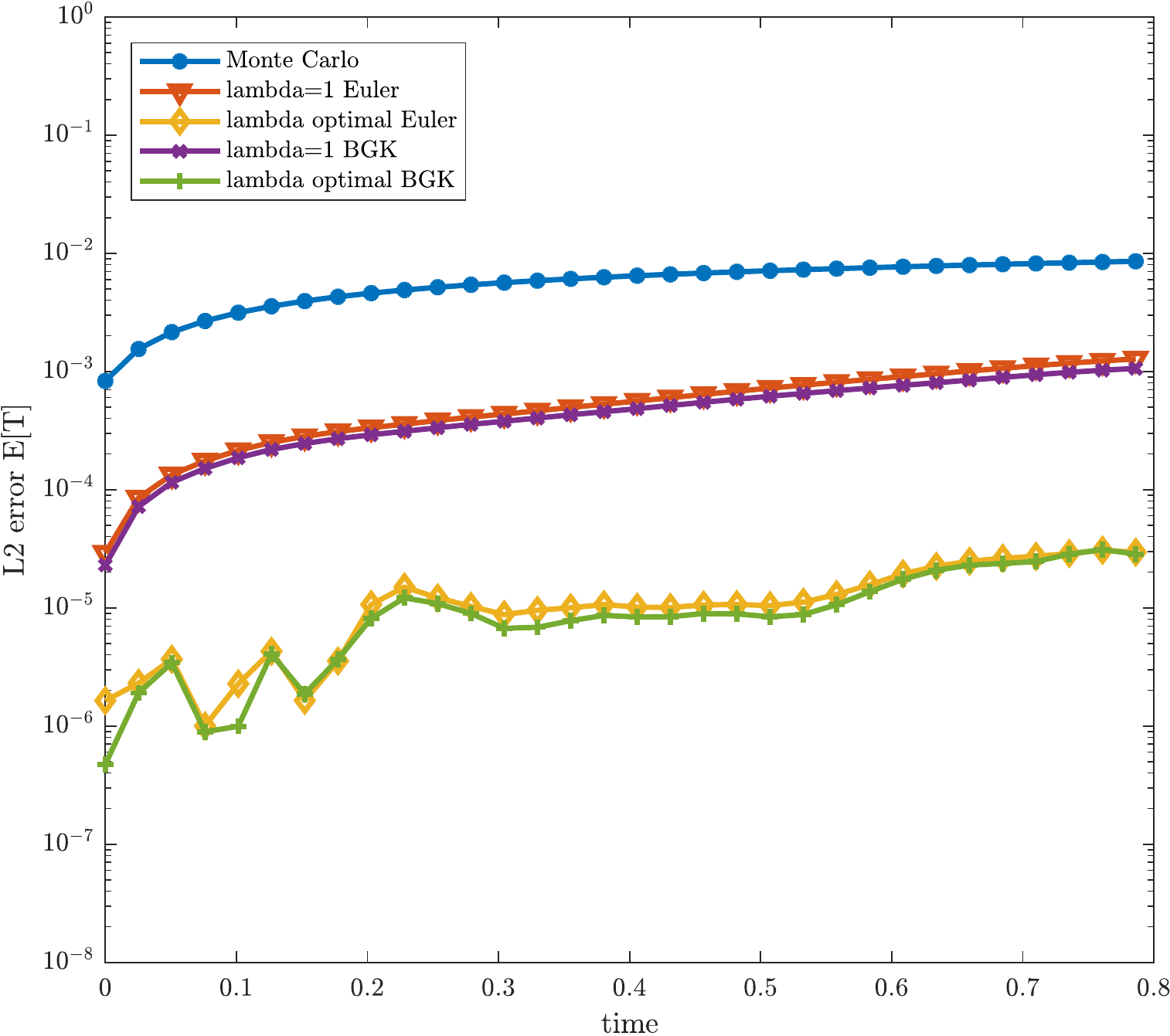}
		\captionof{figure}{Test 5. Sudden heating problem with uncertainty in the boundary condition for the Monte Carlo method and the MSCV method for various control variates strategies.  $L_2$ norm of the error for the expectation of the temperature with $M=10$ samples. Top left: $\varepsilon=10^{-2}$. Top right: $\varepsilon=10^{-3}$. Bottom: $\varepsilon=2 \times 10^{-4}$. Left panels: $M_E=10^3$. Right panels: $M_E=10^4$.}\label{Figure17}
	\end{center}
\end{figure}
\begin{figure}[ht!]
	\begin{center}
		\includegraphics[width=.4\textwidth]{Figure/test3_inhomo/Error1-eps-converted-to.pdf}\hspace{1cm}
		\includegraphics[width=.4\textwidth]{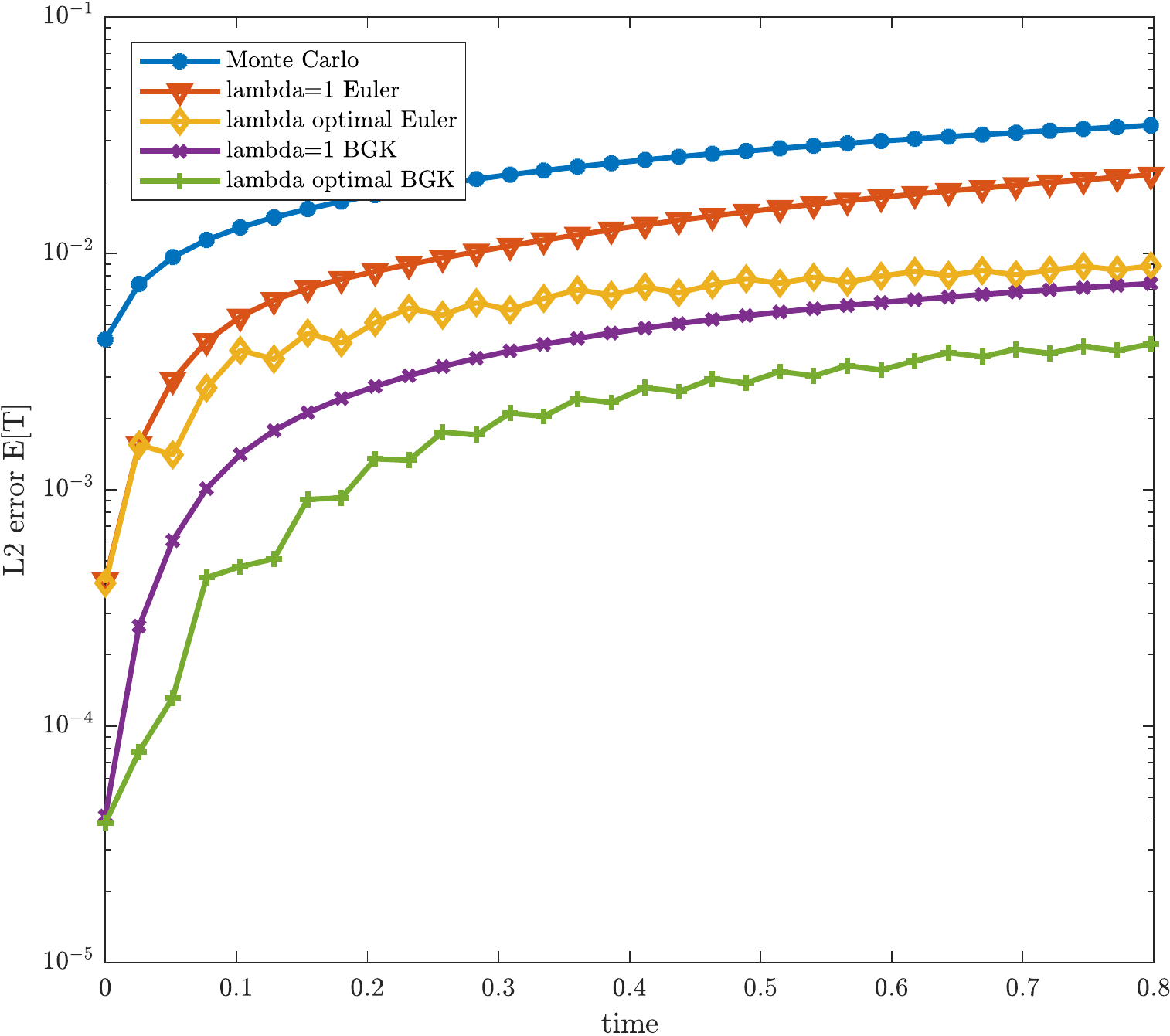}\\
		\captionof{figure}{Test 5. Sudden heating problem with uncertainty in the boundary condition for the Monte Carlo method and the MSCV method for various control variates strategies. $L_2$ norm of the error for the expectation of the temperature with $M=10$ samples, $\varepsilon=10^{-2}$ and $M_E=10^4$. Left variance $\sigma^2_f=1/12$, right $\sigma^2_f=10/12$.}\label{Figure18}
	\end{center}
\end{figure}
Again, the strongest improvement in terms of accuracy is obtained with the MSCV method based on the BGK control variate and optimal $\lambda^*$ independently on the Knudsen number value when $M_E$ is sufficiently large. However, for smaller values of $M_E$ the two control variates give similar results. In particular, the difference between the two control variates, as expected, diminishes for small values of $\varepsilon$.
 {In Figure \ref{Figure18}, we compare the errors for the expected value of the temperature as a function of time in the case in which the variance $\sigma^2_f$ in the stochastic variable $z$ is ten times larger. We only show the error curves in the case $\varepsilon=10^{-2}$ which gives the largest differences. In this situation, we see that the MSCV methods may give better performances when the variance of the stochastic variable is larger.}
 
In Figure \ref{Figure19}, we finally report the shape of the optimization coefficient $\lambda^*(x,t)$ for the different Knudsen numbers.
\begin{figure}[ht!]
	\begin{center}
		\includegraphics[width=.48\textwidth]{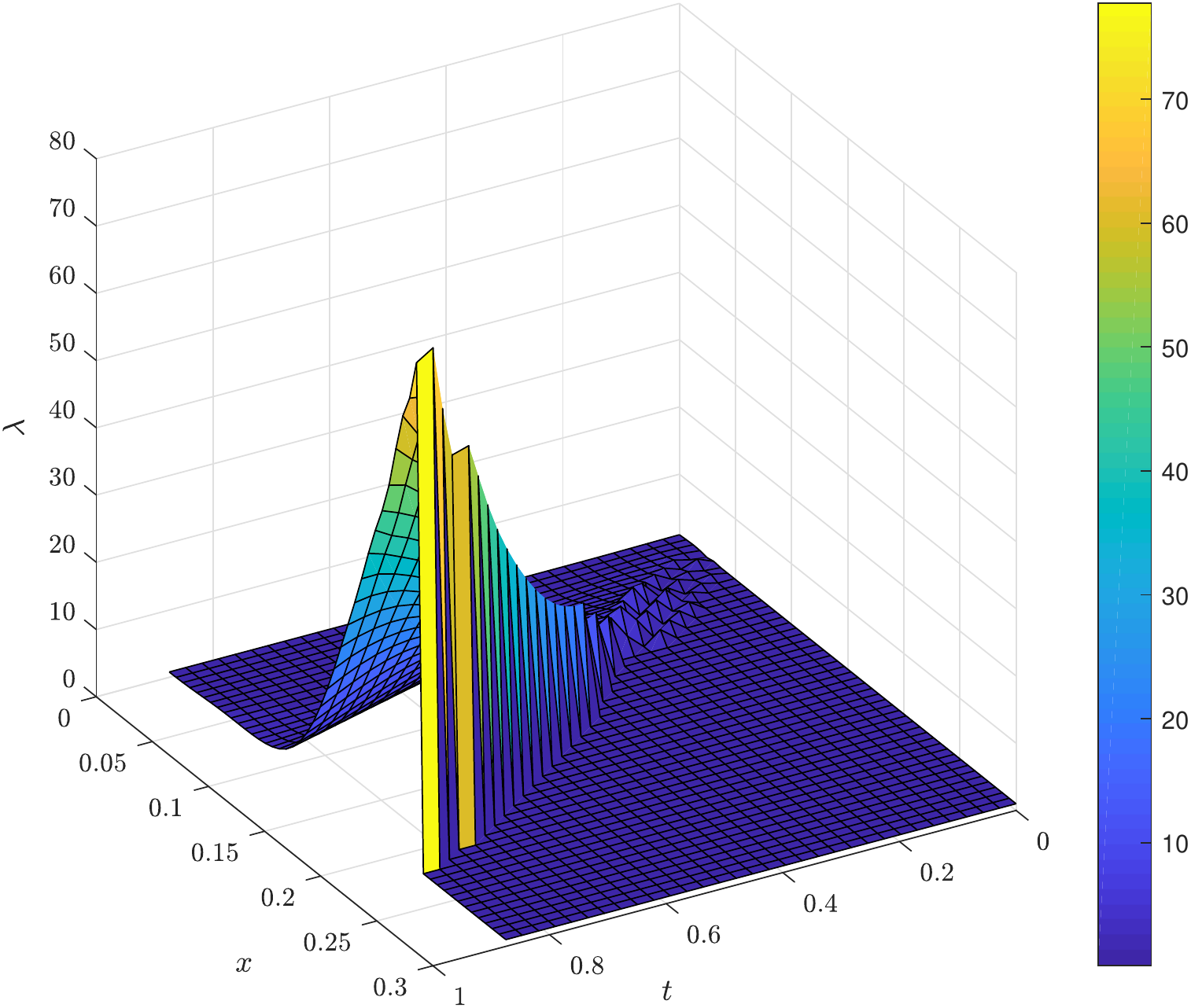}\hspace{0.5cm}
		\includegraphics[width=.48\textwidth]{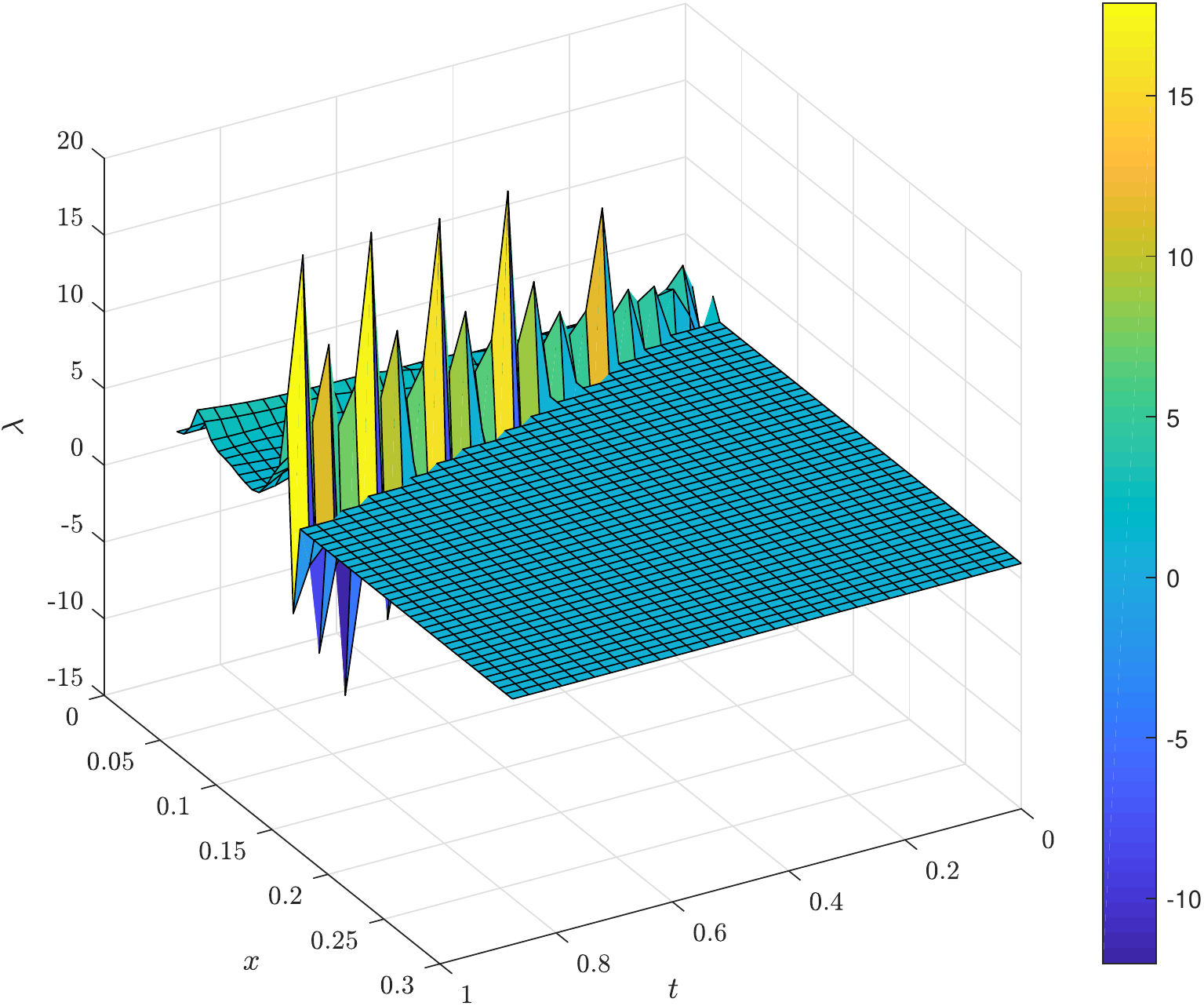}\\
		\includegraphics[width=.48\textwidth]{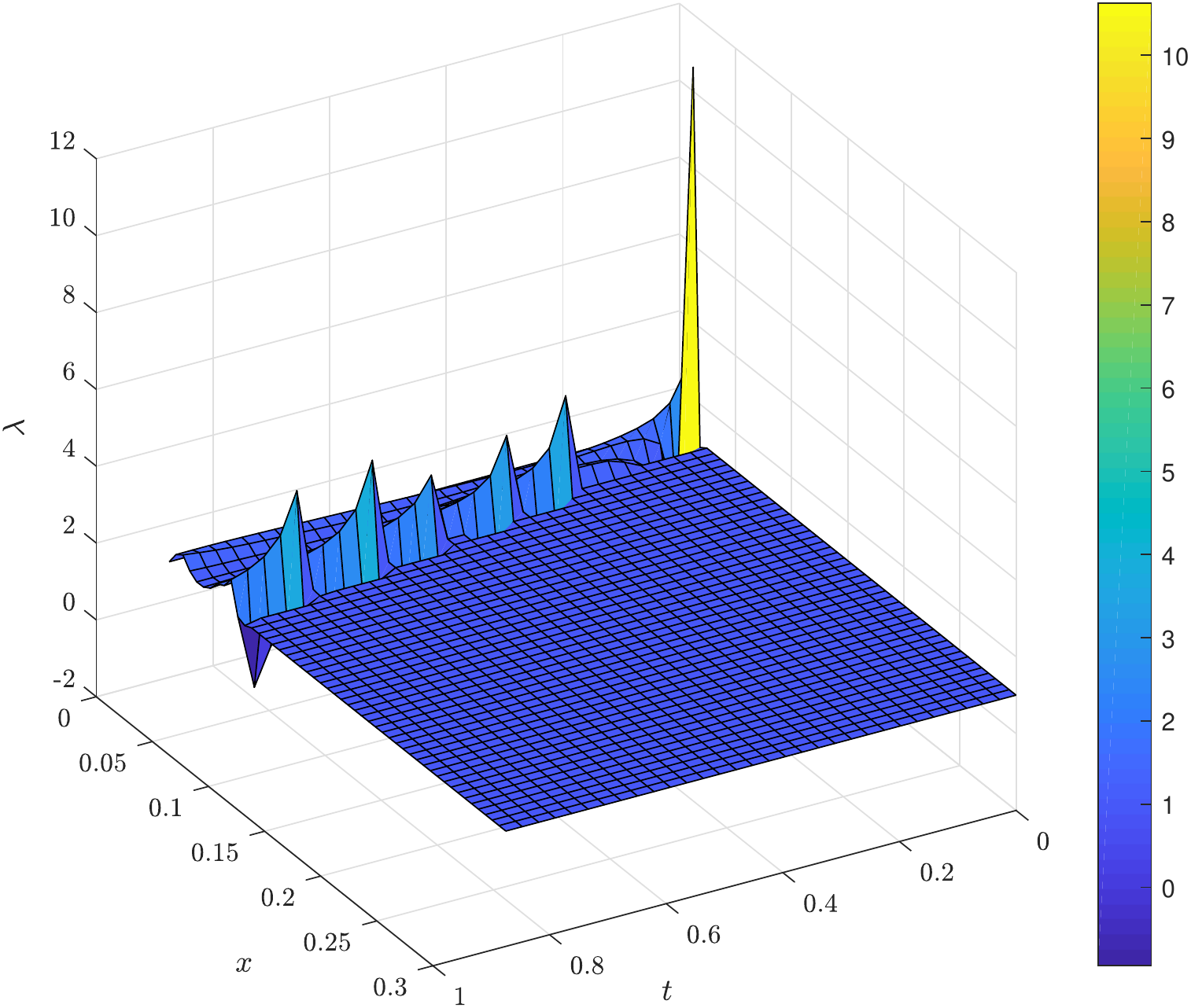}
		\captionof{figure}{Test 5. Sudden heating problem with uncertainty in the boundary condition. Optimal $\lambda^*(x,t)$ for MSCV model based on the BGK control variate. Top left: $\varepsilon=10^{-2}$. Top right: $\varepsilon=10^{-3}$. Bottom: $\varepsilon=2 \times 10^{-4}$ at final time $T_f=0.09$. Magnification around the left boundary.}
		\label{Figure19}
	\end{center}
\end{figure}

\section{Conclusions}
We introduced novel variance reduction techniques for uncertainty quantification of kinetic equations. The methods make use of suitable control variate models, which can be evaluated at a fraction of the cost of the full kinetic model, to accelerate the convergence of standard Monte Carlo methods. The new methods operate in a different way at the various space-time scales of the problem accordingly to the control variate adopted. In particular, the multi-scale control variate (MSCV) methods here presented share the common feature that for space homogeneous problems the variance vanishes for large times, whereas in a space non homogeneous setting the same property hods true in the fluid limit. The numerical results confirm the theoretical analysis and show that MSCV methods outperform standard Monte Carlo approaches in all regimes. Even if, in this work, we mainly focused on the challenging case of the Boltzmann equation, let us mention that the approach described is fully general and can be applied to a large class of kinetic equations and related problems. Several improvements and extensions are in principle possible, for example generalizing the approach to the case of multiple control variates. Research results in this direction will be presented in the nearby future.       


%
%
%

\end{document}